\documentclass[a4paper,10 pt]{amsart}

\usepackage[french]{babel}

\usepackage{amscd}
\usepackage[arrow, matrix, curve]{xy}
\usepackage{color}
\usepackage{relsize} 
\usepackage[bbgreekl]{mathbbol} 
\usepackage{amsfonts} 
\usepackage[T1]{fontenc}

\DeclareSymbolFontAlphabet{\mathbb}{AMSb} 
\DeclareSymbolFontAlphabet{\mathbbl}{bbold}

\newtheorem{Satz}{Satz}[section]
\newtheorem{Lemma_french}[Satz]{Lemme}
\newtheorem{Corollary_french}[Satz]{Corollaire}

\theoremstyle{definition}
\newtheorem{Definition_french}{Définition}
\newtheorem{Remark_french}{Remarque}
\newtheorem{Theorem_french}[Satz]{Théorème}

\newtheorem{proposition}[Satz]{Proposition}

\newtheorem{Conjecture}[Satz]{Conjecture}
\newtheorem{Question}[Satz]{Question}

\usepackage{amssymb}
\usepackage{fullpage}
\usepackage{hyperref}

\DeclareMathOperator{\Aut}{Aut}
\DeclareMathOperator{\Gal}{Gal}

\DeclareMathOperator{\Spec}{Spec}
\DeclareMathOperator{\Spa}{Spa}
\DeclareMathOperator{\Spd}{Spd}

\DeclareMathOperator{\PGL}{PGL}
\DeclareMathOperator{\GL}{GL}

\DeclareMathOperator{\Br}{Br}

\DeclareMathOperator{\Proj}{Proj}

\DeclareMathOperator{\colim}{colim}

\DeclareMathOperator{\Div}{Div}

\DeclareMathOperator{\Pic}{Pic}

\DeclareMathOperator{\Mat}{Mat}
\DeclareMathOperator{\alg}{alg}
\DeclareMathOperator{\ad}{ad}
\DeclareMathOperator{\Fib}{Fib}
\DeclareMathOperator{\cd}{cd}
\DeclareMathOperator{\et}{\acute{e}t}

\begin{document}
\pagenumbering{gobble}
\begin{titlepage}
\begin{center}

\title{Sur la cohomologie étale de la courbe de Fargues-Fontaine}
\author{Sebastian Bartling}
\maketitle
\end{center}

Résumé: On analyse la cohomologie étale des faisceaux constructibles de torsion sur le site étale de la courbe de Fargues-Fontaine associée à un corps $p$-adique et un corps algébriquement clos, non-archimedién de caractéristique $p>0.$ Dans le cas de $\ell\neq p$-torsion, on démontre deux conjectures de Fargues: l'annulation en degrés supérieure ou égal à trois de la cohomologie étale des faisceaux constructibles sur la version algébrique de la courbe de Fargues-Fontaine et la comparaison avec la cohomologie étale de la version adique. Dans le cas $\ell=p$ les résultats sont conditionnels: sous la condition qu'une hypothèse sur la géométrie de la courbe adique est vérifiée, on démontre l'annulation de la cohomologie étale en degrés supérieure ou égal à trois de ces faisceaux Zariski-constructibles sur la courbe adique qui sont algébriques.

\tableofcontents
\end{titlepage}
\newpage
\pagenumbering{arabic}
\setcounter{page}{1}
\section{Introduction}
Soient $p$ un nombre premier fixé et $E$ une extension finie de $\mathbb{Q}_{p},$ $F$ un corps algébriquement clos de caractéristique $p>0,$ extension du corps résiduel de $E,$ complet pour une valeur absolue non-archimédienne non-triviale. Dans le livre \cite{FarguesFontaine} Fargues-Fontaine ont associé à la donnée de $E$,$F$ un objet géométrique qui a attiré un intérêt majeur dans les années récentes. Une des raisons (entre autres...) pour lesquelles cet objet suscite l'intérêt est que la cohomologie étale des système locaux sur cet objet offre une perspective géométrique sur la cohomologie galoisienne du corps $E.$ Cet objet est dénommé la courbe de Fargues-Fontaine et il admet une incarnation algébrique comme schéma noethérien, régulier de dimension $1$ sur $\Spec(E)$ - étudiée dans le livre de Fargues-Fontaine - et une incarnation analytique - étudiée par Kedlaya-Liu \cite{KedlayaLiuFoundations} - comme un espace adique quasi-compact, séparé de dimension $1$ sur $\Spa(E).$ Afin de montrer que cet objet se comporte comme une courbe du point de vue de la cohomologie étale, il faut plus généralement étudier la cohomologie étale des faisceaux constructibles. C'est l'objectif de cet article.
\\
Aujourd'hui il est plus facile d'introduire la courbe en utilisant le point de vue analytique: soient $\mathcal{O}_{F}$ l'anneau des entiers de $F$ et $\varpi\in F$ une pseudo-uniformisante fixée. Soit de plus $\mathcal{O}_{E}$ l'anneau des entiers du corps $E$ et $\pi$ une uniformisante dans $\mathcal{O}_{E}$ et $\mathbb{F}_{q}=\mathcal{O}_{E}/\pi$ le corps résiduel. On considère l'anneau des vecteurs de Witt ramifiés de $\mathcal{O}_{F}:$
$$
W_{\mathcal{O}_{E}}(\mathcal{O}_{F})=W(\mathcal{O}_{F})\otimes_{W(\mathbb{F}_{q})}\mathcal{O}_{E}.
$$
C'est un anneau topologique, complet pour la topologie $(\pi,[\varpi])$-adique.  Il a été démontré dans (\cite[Thm. 2.1.]{FarguesQuelquesresultats}) que l'espace pré-adique
$$
Y_{E,F}=\Spa(W_{\mathcal{O}_{E}}(\mathcal{O}_{F}))-V(\pi\cdot [\varpi])
$$
est bien un espace adique (i.e. le pré-faisceau structurel est un faisceau). On pense à $Y_{E,F}$ comme l'analogue de la boule unité ouverte épointée 'sur $F$' en caractéristique-mixte. L'espace $Y_{E,F}$ est l'espace de module des dé-basculements sur $E$ du corps $F$ (\cite[Prop. II.1.18]{FarguesScholze}). C'est un espace de Stein non-quasi-compact et afin de le rendre compact par quotient on considère l'automorphisme de Frobenius $\varphi\colon W_{\mathcal{O}_{E}}(\mathcal{O}_{F})\rightarrow W_{\mathcal{O}_{E}}(\mathcal{O}_{F})$ qui induit une action proprement discontinue du groupe discret $\varphi^{\mathbb{Z}}$ sur l'espace $Y_{E,F}.$ On peut donc passer au quotient et considérer l'espace adique
$$
X^{\text{ad}}_{E,F}=Y_{E,F}/\varphi^{\mathbb{Z}}.
$$
Cet espace s'appelle la courbe adique de Fargues-Fontaine associée à $E$ et $F.$ La raison principale pour laquelle on appelle cet espace (quasi-compact, partiellement propre sur $\Spa(E)$) une 'courbe' est que localement il est de la forme $\Spa(B_{I},B_{I}^{+}),$ où $B_{I}$ est une $E$-algèbre de Banach qui est un anneau principal (c.f. \cite[Cor. II.1.12] {FarguesScholze}). Cette courbe veut vraiment être considérée comme une 'courbe propre et lisse sur $\Spa(F)$', mais cette phrase n'a aucun sens a priori. \footnote{On peut en donner un sens approximatif en passant aux diamants: $(X^{\text{ad}}_{E,F})^{\lozenge}$ n'a pas de morphisme structurel vers $\Spa(F),$ mais son cousin $\Div^{1}_{E,F}=\Spd(E)\times \Spa(F)/\varphi_{\Spd(E)}^{\mathbb{Z}}$, qui vérifie $((X^{\text{ad}}_{E,F})^{\lozenge})_{\text{ét}}\simeq (\Div^{1}_{E,F})_{\text{ét}},$ a un tel morphisme structurel et $\Div^{1}_{E,F}\rightarrow \Spa(F)$ est un morphisme de diamants propre et lisse (\cite[Prop. II.1.21.]{FarguesScholze}).} Néanmoins, même si la courbe n'est pas de type fini sur $\Spa(E),$ elle vérifie des propriétés de finitude 'absolues' assez fortes: d'après un résultat de Kedlaya l'espace $X^{\text{ad}}_{E,F}$ est fortement noethérien (\cite[Thm. 4.10]{KedlayaNoetherian}): en tout cas, on peut bien travailler avec cet objet!
\\
Comme toute surface de Riemann admet une algébrisation, la courbe adique $X^{\text{ad}}_{E,F}$ en admet aussi une: l'idée est de trouver un fibré en droites ample sur $X^{\text{ad}}_{E,F},$ ce qui donne directement une algébrisation possible. On peut simplement définir le fibré en droites $\mathcal{O}_{X^{\text{ad}}_{E,F}}(1)$ en prenant le fibré en droites trivial $\mathcal{O}_{Y_{E,F}}$ sur $Y_{E,F}$ avec structure $\varphi$-équivariante donnée par $\varphi(f)=\pi^{-1}f.$ D'après le résultat de Kedlaya-Liu \cite[Thm. II.2.6]{FarguesScholze}, c'est bien un fibré en droites 'ample' et avec un peu de courage on peut donc considérer le schéma suivant:
$$
X^{\text{alg}}_{E,F}=\Proj(\bigoplus_{n\geq 0} H^{0}(X^{\text{ad}}_{E,F},\mathcal{O}_{X^{\text{ad}}_{E,F}}(1)^{\otimes n})).
$$
Un des résultats principaux du livre (\cite{FarguesFontaine}) de Fargues-Fontaine est ainsi que c'est un $E$-schéma noethérien, régulier de dimension $1,$ recouvert par deux spectres des anneaux principaux et les corps résiduels des points fermés sont algébriquement clos; leur basculement s'identifie canoniquement à $F$ (\cite[Thm. 6.5.2]{FarguesFontaine}). De plus, la théorie des fibrés vectoriels sur $X^{\text{alg}}_{E,F}$ est bien comprise (\cite[Thm. 8.2.10.]{FarguesFontaine}). Le lien entre la courbe adique et sa version algébrique est le suivant: on peut construire un morphisme entre d'espaces localement annelés
$$
X^{\text{ad}}_{E,F}\rightarrow X^{\text{alg}}_{E,F}
$$
qui induit une bijection entre les points fermés $|X^{\text{alg}}_{E,F}|$ et les points classiques $|X^{\text{ad}}_{E,F}|^{\text{cl}}$ et par tiré en arrière une équivalence de GAGA entre les catégories des fibrés vectoriels (\cite[Prop. II.2.7.]{FarguesScholze}).
\subsection{La cohomologie étale des systèmes locaux sur la courbe}
Dans l'article \cite{FarguesGtorseurs}, Fargues a commencé l'étude de la cohomologie étale de la courbe de Fargues-Fontaine. Il s'est concentré surtout sur les systèmes locaux sur la courbe algébrique et trouve un lien intéressant avec la théorie du corps de classes local de $E:$ il montre par exemple que la classe fondamentale de la théorie du corps de classes dans $H^{2}(E,\mu_{n})$ s'envoie vers la classe fondamentale de la courbe dans $H^{2}(X^{\text{alg}}_{E,F},\mu_{n})$ via tiré en arrière le long du morphisme structurel
$$
f\colon X^{\text{alg}}_{E,F}\rightarrow \Spec(E),
$$
c.f. \cite[Prop. 3.4]{FarguesGtorseurs} et il utilise l'identification $B(G)\simeq H^{1}(X^{\text{alg}}_{E,F},G),$ où $G$ est un groupe réductif, connexe sur $E$ (\cite[Thm. 5.1.]{FarguesGtorseurs}) pour donner une nouvelle preuve de la théorie du corps de classes local de $E$ (\cite[Thm. 2.6]{FarguesGtorseurs}). De plus, comme la courbe est géométriquement simplement connexe tous les $\mathbb{Z}/n\mathbb{Z}$-systèmes locaux $\mathcal{L}$ sont de la forme $\mathcal{L}=f^{*}(M),$ où $M$ est un $\mathbb{Z}/n\mathbb{Z}$-module fini avec action continue de $\Gal(\overline{E}/E).$ Cette correspondance identifie la cohomologie étale de $\mathcal{L}$ en degrés $i\in \lbrace 0, 1, 2\rbrace $  à la $\Gal(\overline{E}/E)$-cohomologie de $M$ (\cite[Thm. 3.7]{FarguesGtorseurs}) en degrés $i\in \lbrace 0, 1, 2\rbrace .$ Néanmoins on ne sait pas encore que
$$
H^{i}(X^{\text{alg}}_{E,F},\,\mathcal{L})=0,
$$
pour $i\geq 3.$
On peut donc constater que la courbe peut 'voir' l'arithmétique du corps $p$-adique $E$ et l'étude de la cohomologie étale des systèmes locaux sur la courbe donne un point de vue géométrique à l'étude de la cohomologie galoisienne du corps $E$ (c.f. par exemple le résultat \cite[Cor. 3.8.]{FarguesGtorseurs} qui donne une interprétation de la dualité de Tate-Nakayama comme dualité de Poincaré sur la courbe $X^{\text{alg}}_{E,F}$). Le cas des systèmes locaux compris on se tourne vers les faisceaux constructibles arbitraires.
\subsection{Les résultats}
J'explique maintenant plus en détail le contenu de cet article. Soit $\mathcal{F}$ un $\mathbb{Z}/n\mathbb{Z}$-module constructible sur le site étale de $X^{\alg}_{E,F}.$ Dans la section \ref{Subsection Vergleich einige Spezialfaelle} en bas, on montre comment on peut construire un morphisme entre sites étales
$$
u\colon (X^{\text{ad}}_{E,F})_{\text{ét}}\rightarrow (X^{\alg}_{E,F})_{\text{ét}}.
$$
On peut alors considérer le tiré en arrière $\mathcal{F}^{\text{ad}}=u^{*}(\mathcal{F}).$ La première conjecture de Fargues prédit une comparaison entre la cohomologie étale de $\mathcal{F}^{\text{ad}}$ et celle de $\mathcal{F}.$
\begin{Conjecture}{(Fargues, \cite[Conjecture 3.11.]{FarguesGtorseurs})}\label{Vermutung Vergleich}
Soit $\mathcal{F}$ un $\mathbb{Z}/n\mathbb{Z}$-module constructible sur le site étale de $X^{\alg}_{E,F},$ alors on a pour tout $i\geq 0,$
$$
H^{i}(X^{\text{alg}}_{E,F},\,\mathcal{F})\simeq H^{i}(X^{\text{ad}}_{E,F},\,\mathcal{F}^{\text{ad}}).
$$
\end{Conjecture}
La deuxième conjecture est une conjecture d'annulation:
\begin{Conjecture}{(Fargues, \cite[Conjecture 3.9.]{FarguesGtorseurs})}\label{Vermutung Verschwinden}
Soit $\mathcal{F}$ un $\mathbb{Z}/n\mathbb{Z}$-module constructible sur le site étale de $X^{\alg}_{E,F},$
alors on a
$$
H^{i}(X^{\text{alg}}_{E,F},\,\mathcal{F})=0,
$$
pour $i\geq 3.$
\end{Conjecture}
Voici le résultat principal:
\begin{Theorem_french}\label{Hauptsatz}
Soit $\mathcal{F}$ un $\mathbb{Z}/n\mathbb{Z}$-module constructible sur le site étale de $X^{\alg}_{E,F}.$
\begin{enumerate}
\item[(a):] Si $n=\ell\neq p,$ les deux conjectures sont vraies.
\item[(b):] Si $n=p,$ alors on a
$$
H^{i}(X^{\text{ad}}_{E,F},\,\mathcal{F}^{\text{ad}})=0,
$$
pour $i\geq 3,$ si une hypothèse technique est satisfaite.
\end{enumerate}
De plus, sous cette hypothèse technique la conjecture \ref{Vermutung Verschwinden} est équivalente à la conjecture \ref{Vermutung Vergleich}.
\end{Theorem_french}
\begin{Remark_french}\label{Bemerkung zu pre-perfectoide Stein}
Soit $U$ un ouvert dans $X^{\text{alg}}_{E,F}$ avec adification $U^{\text{ad}}\subset X^{\text{ad}}_{E,F}.$ Alors l'hypothèse technique est la suivante: on suppose que l'on peut écrire
$$
U^{\text{ad}}=\bigcup_{n\in \mathbb{N}}U_{n},
$$
où $U_{n}\subset U^{\text{ad}}$ sont des ouverts affinoïdes, tels que $U_{n}\subseteq U_{n+1},$ l'image de l'application de restriction $\text{res}\colon \mathcal{O}(U_{n+1})\rightarrow \mathcal{O}(U_{n})$ est dense; si on note par $E_{\infty}$ la $\mathbb{Z}_{p}$-extension cyclotomique de $E,$ les $$U_{n}\times_{\Spa(E)}\Spa(\widehat{E}_{\infty})$$ sont affinoïdes perfectoïdes. Je vais appeler un tel espace adique un espace pré-perfectoïde Stein.\footnote{L'adjectif 'pré-perfectoïde' est là seulement pour une raison peu sympathique: on ne sait pas démontrer qu'un ouvert affinoïde dans un espace perfectoïde est affinoïde perfectoïde...}. Dans la remarque \ref{Bermerkung Diskussion Hypothese Stein} on trouve une petite discussion de cette hypothèse et une raison pour laquelle je pense que cette hypothèse devrait être vraie.
\end{Remark_french}
\begin{Remark_french}
La question si tout faisceau Zariski-constructible $\mathcal{G}$ sur le site étale de $X^{\ad}_{E,F}$ est de la forme $\mathcal{F}^{\ad}$ pour un faisceau constructible $\mathcal{F}$ sur le site étale de $X^{\alg}_{E,F}$ semble intéressante. Pour répondre à cette question on pourrait essayer de résoudre le problème d'extension de Riemann (comme dans \cite{LutkebohmertRiemannExtension}) pour la courbe adique.
\end{Remark_french}
\subsubsection{Le cas $\ell\neq p$}
Maintenant je donne plus de détails sur la preuve du point (a) ci-dessus; i.e. la preuve dans le cas plus facile de $\ell\neq p$-torsion. L'idée est d'expliquer d'abord pourquoi la conjecture \ref{Vermutung Vergleich} est équivalente à la conjecture \ref{Vermutung Verschwinden}. On suppose pour un moment que la conjecture \ref{Vermutung Verschwinden} soit vérifiée. La première observation (Lemme \ref{Reduktion von allgemein konstru zu konstant verzweigte}) - observation plus au moins standard dans la cohomologie étale - est qu'il suffit de démontrer que
$$
H^{i}(X^{\prime},\,\mathbb{F}_{\ell})\simeq H^{i}(X^{\prime \text{ad}},\,\mathbb{F}_{\ell}),
$$
pour tout revêtement ramifié $X^{\prime}\rightarrow X^{\text{alg}}_{E,F}$ (c.f. Déf. \ref{Def verzweigte Ueberdeckung} pour cette notion). Pour vérifier la comparaison en degrés $i=0,1,$ on utilise l'équivalence de GAGA pour ces revêtements ramifiés (Lemme \ref{GAGA verzweigte Sachen}). Cette équivalence de GAGA pour des revêtements ramifiés est déduite de l'équivalence de GAGA connue pour la courbe en utilisant le fait que le morphisme $X^{\prime}\rightarrow X^{\text{alg}}_{E,F}$ et le morphisme $X^{\prime \text{ad}}\rightarrow X^{\text{ad}}_{E,F}$ sont localement libres. Pour la comparaison en degré $i=2$ on démontre un analogue adique d'un résultat de Gabber, qui compare les groupe de Brauer cohomologique et cohérent d'un espace adique, fortement noethérien sur un corps de caractéristique $0,$ qui est la réunion séparée des deux affinoïdes (Lemme \ref{Gabbers resultat fuer adische Raueme}). De fait, dans cette étape on traite dans un premier temps le cas d'un espace affinoïde (Lemme \ref{Gabbers theorem affinoide adische Raeume}) en utilisant un théorème de Huber qui compare la cohomologie étale d'un affinoïde $\Spa(A,A^{+})$ à celle de $\Spec(A)$ et le résultat de Gabber qui identifie le groupe de Brauer au groupe de Brauer cohomologique pour tout schéma affine. Ensuite on utilise des faisceaux tordus localement libres sur des gerbes pour l'étape de recollement: ici je me suis fortement inspiré de la thèse de Lieblich \cite{LieblichPhD}.
\\
Enfin, il en résulte que la conjecture \ref{Vermutung Verschwinden} implique la conjecture \ref{Vermutung Vergleich}.\footnote{De fait, tous ces arguments marchent aussi dans le cas $\ell=p,$ si on savait que $H^{i}(X^{\prime},\mathbb{F}_{p})=0$ pour $i\geq 3$ et tout revêtement ramifié.} 
\\
Après on démontre directement que $\cd_{\ell}(X^{\text{ad}}_{E,F})\leq 2$ (Prop. \ref{l-torsion adische Kurve}). Cette borne suit facilement de la présentation du diamant associé à $X^{\text{ad}}_{E,F}$ suivante:
$$
(X^{\text{ad}}_{E,F})^{\lozenge}\simeq (\mathbb{B}^{1,\circ,*,\text{perf}}_{F}/\varphi_{F}^{\mathbb{Z}})/\underline{\mathbb{Z}_{p}}.
$$
En particulier, il est évident que la conjecture \ref{Vermutung Vergleich} implique la conjecture \ref{Vermutung Verschwinden} (au moins si $\ell\neq p$). On donne après une preuve directe du fait que $\cd_{\ell}(X^{\text{alg}}_{E,F})\leq 2$ (Prop \ref{l-kohomologische Dimension algebraische Kurve}), $\ell\neq p.$ Pour la preuve, on vérifie d'abord que $\cd_{\ell}(E(X^{\text{alg}}_{E,F}))\leq 1$ (Prop. \ref{Prop l dimension Funktionenkoerper}), où $E(X^{\text{alg}}_{E,F})$ est le corps de fonctions de la courbe $X^{\text{alg}}_{E,F}.$ Après la borne sur la dimension cohomologique de la courbe algébrique est simplement une application facile soit de la pureté absolue soit du fait que les corps résiduels de tous les points fermés des revêtements ramifiés $X^{\prime}\rightarrow X^{\text{alg}}_{E,F}$ sont algébriquement clos.
\subsubsection{Le cas $\ell=p$}
J'explique la preuve (donnée dans la proposition \ref{Prop Fall l=p}) pour $\ell=p$ de mon résultat principal; dans ce cas là mes résultats sont plutôt faibles c.f. Théorème \ref{Hauptsatz}(b): on utilise d'abord la méthode de la trace pour ramener l'énoncé
$$
H^{i}(X^{\text{ad}}_{E,F},\,\mathcal{F}^{\text{ad}})=0,
$$
où $\mathcal{F}$ est un $\mathbb{F}_{p}$-module constructible sur $X^{\text{alg}}_{E,F}$ et $i\geq 3,$ à
$$
H^{i}(X^{\prime},\,j^{\prime \text{ad}}_{!}(\mathbb{F}_{p}))=0,
$$
où
$i\geq 3,$ $\pi\colon X^{\prime}\rightarrow X^{\text{alg}}_{E,F}$ est un revêtement ramifié de degré premier à $p,$ normalisation d'un revêtement étale $V\rightarrow U\subseteq X^{\text{alg}}_{E,F}.$ On note $j^{\prime}\colon V\rightarrow X^{\prime}$ et $i^{\prime}\colon Z^{\prime}\rightarrow X^{\prime}$ l'inclusion du fermé complémentaire. Après passage au monde adique, on peut résumer la situation dans un diagramme commutatif 
$$
\xymatrix{
V^{\text{ad}} \ar[r]^{j^{\prime ^{\text{ad}}}} \ar[d]^{\dot{\pi}^{\text{ad}}} & X^{\prime ^{\text{ad}}} \ar[d]^{\pi^{\text{ad}}} &  \ar[l]^{i^{\prime ^{\text{ad}}}}Z^{\prime ^{\text{ad}}} \ar[d] \\
U \ar[r]^{j^{\text{ad}}} & X^{\text{ad}}_{E,F} & \ar[l]^{i^{\text{ad}}} Z^{\text{ad}}.
}
$$
Si l'hypothèse technique dans la remarque \ref{Bemerkung zu pre-perfectoide Stein} est vérifiée, l'ouvert adique $U^{\text{ad}}$ est un espace pré-perfectoïde Stein et on en déduit $H^{i}(V^{\text{ad}},\,\mathbb{F}_{p})=0,$ pour $i\geq 3.$ Maintenant, on analyse la cohomologie du faisceau $j^{\prime \text{ad}}_{!}\mathbb{F}_{p}$ en utilisant le triangle exact suivant:
$$
\xymatrix{
(j^{\prime \text{ad}})_{!}\mathbb{F}_{p} \ar[r] & Rj^{\prime \text{ad}}_{*}\mathbb{F}_{p} \ar[r] & i^{\prime \text{ad}}_{*} i^{\prime \text{ad} *}Rj^{\prime \text{ad}}_{*} \mathbb{F}_{p} \ar[r]^-{+1} &.
}
$$
Il faut donc démontrer que le complexe
$$
i^{\prime \text{ad}}_{*} i^{\prime \text{ad} *}Rj^{\prime \text{ad}}_{*} \mathbb{F}_{p}
$$
n'a pas de cohomologie en degrés $i\geq 2.$ On analyse ici la situation en utilisant le $\underline{\mathbb{Z}_{p}}$-recouvrement pro-étale perfectoïde $X^{\prime \text{ad} \lozenge}\times_{\Spd(E)}\Spd(\widehat{E}_{\infty})\rightarrow X^{\prime \text{ad} \lozenge},$ où $E_{\infty}$ est la $\mathbb{Z}_{p}$-extension cyclotomique du corps $E.$
\\
La formule suivante est vérifiée:
$$
R\Gamma(Z^{\prime \text{ad}},\,i^{\prime \text{ad} *}Rj^{\prime \text{ad}}_{*}\mathbb{F}_{p}) = \underset{W^{\prime}}\colim R\Gamma(W^{\prime}-Z^{\prime \text{ad}},\, \mathbb{F}_{p})=R\Gamma(\mathbb{Z}_{p},\underset{\widetilde{W}^{\prime}}\colim R\Gamma(\widetilde{W}^{\prime}-(Z^{\prime \text{ad}})_{\infty},\,\mathbb{F}_{p})),
$$
où la première colimite porte sur les voisinages ouverts, qcqs $W^{\prime}$ de $Z^{\prime \text{ad}}$ et la deuxième sur les voisinages ouverts qcqs $\underline{\mathbb{Z}_{p}}$-invariants $\widetilde{W}^{\prime}\subset (X^{\prime \text{ad} \lozenge})_{\infty}$ de $(Z^{\prime \text{ad}})_{\infty}.$ 
\\
Le fait que $\cd_{p}(\mathbb{Z}_{p})=1$ pose le problème principal; le point clé pour répondre à cette difficulté est le suivant: l'extension $E_{\infty}$ de $E$ est de la forme $E_{\infty}=\bigcup_{n}E_{n},$ où les $E_{n}$ sont des extensions galoisiennes avec groupe de Galois $G_{n}=\Gal(E_{n}/E)\simeq \mathbb{Z}/p^{n}\mathbb{Z}$. En utilisant un système de voisinages fondamentaux $\underline{\mathbb{Z}}_{p}$-invariants de $Z^{\prime \text{ad}}_{\infty}$ bien choisi, on construit un isomorphisme entre $\mathbb{Z}_{p}=\lim_{n}G_{n}$-représentations 
$$
\underset{\widetilde{W}^{\prime}} \colim H^{j}(\widetilde{W}^{\prime}-(Z^{\prime \text{ad}})_{\infty},\,\mathbb{F}_{p})=\underset{n\in \mathbb{N}} \colim \text{ } \text{Ind}_{\lbrace e \rbrace}^{G_{n}}(M_{n}^{(j)}),
$$
où les $M_{n}^{(j)}$ sont des $\mathbb{F}_{p}$-modules discrets, munis de l'action triviale de $G_{n}:$ ils dépendent de $j$ et forment un système direct, i.e. on a une application $M_{n}^{(j)}\rightarrow M_{n+1}^{(j)}.$ Après on vérifie (c.f. Lemme \ref{Lemma zum Verschwinden der Gruppenkohomologie}) qu'une représentation du groupe $\mathbb{Z}_{p}$ de la forme 
$$
\underset{n\in \mathbb{N}} \colim \text{ } \text{Ind}_{\lbrace e \rbrace}^{G_{n}}(M_{n}^{(j)})
$$
n'a pas de $\mathbb{Z}_{p}$-cohomologie. Comme l'inclusion $U^{\text{ad}}\hookrightarrow X^{\text{ad}}_{E,F}$ est évidemment 'localement de Stein', on peut en déduire l'énoncé d'annulation suivant
$$
\underset{\widetilde{W}^{\prime}}  \colim H^{i}(\widetilde{W}^{\prime}-(Z^{\prime \text{ad}})_{\infty},\,\mathbb{F}_{p})=0,
$$
pour $i\geq 2$ (c.f. lemme \ref{Lemma Verschwinden lokal Stein}). Cela permet de terminer la preuve de la proposition \ref{Prop Fall l=p}.
\begin{Remark_french}
Il serait intéressant d'analyser aussi le faisceau $R^{1}j^{\prime \ad}_{*}\mathbb{F}_{p}$ afin d'établir la pureté absolue dans cette situation.
\end{Remark_french}
\subsection{Aperçu du contenu}
Voici un bref aperçu du contenu des sections. 
La première partie, section \ref{Section Einige Resultate Vergleich}, concerne le problème de comparer la cohomologie étale de la courbe algébrique avec celle de la courbe adique. Dans la sous-section \ref{Subsection Vergleich einige Spezialfaelle} un morphisme  entre les sites étales
$$
u\colon (X^{\text{ad}}_{E,F})_{\text{ét}}\rightarrow (X^{\alg}_{E,F})_{\text{ét}}
$$
est construit et une première réduction, lemme \ref{Reduktion von allgemein konstru zu konstant verzweigte}, du problème de comparaison à un problème plus concret est donné. Dans la sous-section \ref{subsection GAGA fuer verzweigte Sachen} l'équivalence de GAGA pour tout revêtement ramifié de la courbe algébrique est démontré. Après, dans la sous-section \ref{subsection Verschwinden impliziert Vergleich}, on explique comment l'annulation implique la comparaison. Le plan est le suivant: dans un premier temps, sous-sous-section \ref{subusbsection Widerholungen zu Brauergruppe und getwisteten Garben}, on donne quelques rappels sur les groupes de Brauer et les faisceaux tordus dans le monde adique. Après, dans sous-sous-section \ref{subsubsection Gabbers satz fuer adische Raueme}, j'explique la preuve de l'analogue du résultat de Gabber concernant le groupe de Brauer et le groupe de Brauer cohomologique pour les espaces adiques. Ces résultats sont appliqués dans le sous-sous-section \ref{subsubsection Verschwinden impliziert Vergleich Beweise} pour démontrer que si la conjecture d'annulation, conjecture \ref{Vermutung Verschwinden}, est vérifiée la conjecture de comparaison, conjecture \ref{Vermutung Vergleich}, en résulte.
\\
La deuxième partie, section \ref{section Einige resultate Verschwinden}, est consacrée à quelques résultats d'annulation de la cohomologie étale de la courbe adique et algébrique. Dans la sous-section \ref{subsection ell koho Dim der adischen Kurve} on démontre que l'on a $\cd_{\ell}(X^{\text{ad}}_{E,F})\leq 2,$ où $\ell\neq p$ est un nombre premier. Après, sous-section \ref{subsection ell koho dim des Funktionenkoerpers}, on démontre une borne sur la $\ell$-dimension cohomologique du corps de fonctions de la courbe algébrique. Ce résultat est utilisé dans la sous-section \ref{subsection ell koho dim der algebraischen Kurve} pour borner la $\ell$-dimension cohomologique de la courbe algébrique par $2.$ Dans la sous-section \ref{subsection der Fall l=p} j'explique ce que je sais faire dans le cas $\ell=p.$ Dans la sous-section finale \ref{subsection Quatsch zu nicht-Zariski konstru} j'analyse un peu le cas d'un $\mathbb{F}_{p}$-module constructible sur la courbe adique qui n'est pas Zariski-constructible.
\subsection{Remerciements}
Il y a beaucoup de mathématiciens et mathématiciennes que j'ai le plaisir de remercier: tout d'abord mon directeur de thèse Laurent Fargues, qui m'a dit que je devrais réfléchir sur la cohomologie étale de la courbe. Merci beaucoup de m'avoir donné la possibilité de réfléchir sur la courbe, pour des discussions et le partage généreux de votre intuition profonde. De plus, tout le long de mon travail j'ai eu le plaisir de discuter très régulièrement avec Benoît Stroh et Arthur-César le Bras; sans leur intérêt et soutien cet article n'existerait pas et leur remarques ont influencé ce travail. J'aimerais remercier Johannes Anschütz pour des commentaires sur une version préliminaires, son intérêt et pour des suggestions utiles. Mes remerciements vont également à Lennart Gehrmann pour des discussion concernant la cohomologie des groupes et Guido Bosco pour des discussion sur des espaces de Stein.
\\
Finalement, je veux remercier Sophie Morel chaleureusement pour tout le travail qu'elle a fait en tant que rapporteuse de thèse; grâce à sa relecture détaillée et ses remarques utiles les maths sont devenus plus vraies et la présentation plus claire.
\\
L'auteur a était financé par l'ERC Advanced Grant 742608 GeoLocLang.
\newpage
\section{Quelques résultats concernant la comparaison}\label{Section Einige Resultate Vergleich}
\subsection{Énoncé et première réduction}\label{Subsection Vergleich einige Spezialfaelle}
L'objectif ici est d'abord d'introduire un morphisme entre les sites étales
$$
u\colon (X^{\text{ad}}_{E,F})_{\text{ét}}\rightarrow (X^{\alg}_{E,F})_{\text{ét}},
$$
ce qui correspond à un foncteur $u\colon \text{Ét}(X^{\alg}_{E,F})\rightarrow \text{Ét}(X^{\text{ad}}_{E,F});$ une fois ce foncteur construit, on peut énoncer la conjecture \ref{Vermutung Vergleich} de comparaison ci-dessus. Ensuite, on explique comment réduire la comparaison entre la cohomologie étale de $X^{\alg}_{E,F}$ et $X^{\text{ad}}_{E,F}$ avec des coefficients constructibles arbitraires au problème de comparer la cohomologie étale des revêtements ramifiés de $X^{\text{alg}}_{E,F}$ avec celle du pendant adique, c.f. le lemme \ref{Reduktion von allgemein konstru zu konstant verzweigte}.
\\
La construction du foncteur $u\colon \text{Ét}(X^{\alg}_{E,F})\rightarrow \text{Ét}(X^{\text{ad}}_{E,F})$ utilise de façon cruciale un résultat de Kedlaya \cite[Thm. 4.10]{KedlayaNoetherian} qui assure que l'espace adique sous-perfectoïde $X^{\text{ad}}_{E,F}$ est fortement noethérien: en utilisant l'interprétation comme espace de module du $\text{Proj}$ dans la catégorie des espaces localement annelés (c.f. \cite[Prop. II.2.7.]{FarguesScholze}), on construit un morphisme 
$$
X^{\text{ad}}_{E,F}\rightarrow X^{\text{alg}}_{E,F}.
$$
Puis on peut appliquer un résultat de Huber \cite[Proposition 3.8]{HuberFormalScheme}, pour construire pour tout morphisme étale $U\rightarrow X^{\text{alg}}_{E,F}$ un espace adique fortement noethérien $U^{\text{ad}}$ qui vit dans un diagramme commutatif
$$
\xymatrix{
U^{\text{ad}} \ar[d] \ar[r] & U \ar[d] \\
X^{\text{ad}}_{E,F} \ar[r] & X^{\text{alg}}_{E,F}.
}
$$
L'espace $U^{\text{ad}}$ est localement fortement noethérien et uniquement caractérisé par une propriété universelle appropriée c.f. loc.cit.
\\
On observe que le morphisme d'espaces adiques $U^{\text{ad}}\rightarrow X^{\text{ad}}_{E,F}$ est toujours étale (c.f. \cite[Corollary 1.7.3]{HuberBuch} pour l'énoncé qu'on peut utiliser après on a utilisé que la propriété d'être étale est localement sur la cible), ce qui définit un foncteur $u\colon \text{Ét}(X^{\alg}_{E,F})\rightarrow \text{Ét}(X^{\text{ad}}_{E,F}).$ De plus, ce foncteur envoie des recouvrements vers des recouvrements et il commute aux limites finies: cela suffit pour construire le morphisme $u\colon (X^{\text{ad}}_{E,F})_{\text{ét}}\rightarrow (X^{\alg}_{E,F})_{\text{ét}}.$
\\
Soit $\mathcal{F}$ un $\mathbb{Z}/n\mathbb{Z}$-module constructible sur $(X^{\alg}_{E,F})_{\text{ét}}.$ On peut appliquer le foncteur dérivé $R\Gamma(X^{\alg}_{E,F},.)$ au morphisme d'adjonction
$$
\mathcal{F}\rightarrow Ru_{*}Ru^{*}\mathcal{F}=Ru_{*}u^{*}\mathcal{F}
$$
pour construire un morphisme
$$
\Phi_{\mathcal{F}}\colon R\Gamma(X^{\alg}_{E,F},\, \mathcal{F})\rightarrow R\Gamma(X^{\text{ad}}_{E,F},\,\mathcal{F}^{\text{ad}}),
$$
où $\mathcal{F}^{\text{ad}}=u^{*}\mathcal{F}.$ Fargues a posé la conjecture \cite[Conjecture 3.11]{FarguesGtorseurs} que $\Phi_{\mathcal{F}}$ soit toujours un isomorphisme. 
\\
Maintenant on explique comment réduire le problème de démontrer que $\Phi_{\mathcal{F}}$ est un isomorphisme en un problème plus concret:
\begin{Lemma_french}\label{Reduktion von allgemein konstru zu konstant verzweigte}
Si $\Phi_{\mathcal{F}}$ est un isomorphisme pour tout faisceau constructible de la forme $\pi_{*}(M),$ où $\pi\colon X^{\prime}\rightarrow X^{\text{alg}}_{E,F}$ est un morphisme fini plat qui est fini étale au-dessus d'un ouvert dense, où $X^{\prime}$ est un schéma noethérien, régulier de dimension $1,$ $M$ est un $\mathbb{Z}/n\mathbb{Z}$-module fini, alors $\Phi_{.}$ est un isomorphisme pour tout faisceau constructible.
\end{Lemma_french}
Comme les objets $\pi\colon X^{\prime}\rightarrow X^{\text{alg}}_{E,F}$ jouent un rôle essentiel pour la suite, on leur donne un nom:
\begin{Definition_french}\label{Def verzweigte Ueberdeckung}
On appelle revêtement ramifié de $X^{\text{alg}}_{E,F}$ une paire $(X^{\prime},\pi),$ où $X^{\prime}$ est un schéma noethérien, régulier de dimension $1$ et $\pi\colon X^{\prime}\rightarrow X^{\text{alg}}_{E,F}$ est un morphisme fini, plat, qui est fini étale au-dessus d'un ouvert dense de $X^{\text{alg}}_{E,F}.$
\end{Definition_french}
Voici la preuve du lemme \ref{Reduktion von allgemein konstru zu konstant verzweigte} ci-dessus:
\begin{proof}
La preuve est simplement une modification de la preuve de \cite[Tag 09Z6]{stacks} (un résultat standard dans la théorie de la cohomologie étale pour les schémas). Pour le lecteur impatient on peut déjà dire que le point clé est que la courbe schématique (en caractéristique 0!) admet un recouvrement par les spectres des anneaux de Nagata; ce qui implique que la normalisation reste bien finie. 
\\
Voici les détails: soit $\mathcal{F}$ un $\mathbb{Z}/n\mathbb{Z}$-module constructible sur $X^{\text{alg}}_{E,F}.$ D'abord, on peut trouver un ouvert dense $j\colon U\subseteq X^{\text{alg}}_{E,F},$ tel que $\mathcal{F}\!\mid_{U}=:\mathcal{L}$ est un système local.\footnote{Soit $X=\coprod_{i\in I}X_{i}$ une stratification constructible, tel que $\mathcal{F}\!\mid_{X_{i}}$ soit un système local. Soit $i\in I,$ tel que $\eta\in X_{i}.$ Ici $\eta\in X^{\text{alg}}_{E,F}$ est le point générique. On peut trouver un voisinage ouvert $U$ de $\eta,$ qui est contenu dans $X_{i}$ (\cite[Tag 0AAW]{stacks}); c'est l'ouvert dense cherché.} Soit $Z=\lbrace x_{1},...,x_{n} \rbrace$ le fermé complémentaire; tous les $x_{i}$ sont des points fermés et les corps résiduels $\kappa(x_{i})$ sont algébriquement clos, car $F$ l'est par hypothèse.\footnote{On rappelle que les $\kappa(x_{i})$ sont des de-basculements de $F.$} Soit $i_{x_{i}}$ l'immersion fermée du point $x_{i}.$ On a donc $i_{x_{i}}^{*}(\mathcal{F})=:M_{i},$ $i=1,...,n$ où $M_{i}$ est un $\mathbb{Z}/n\mathbb{Z}$-module fini. De plus, on peut trouver un recouvrement fini étale
$$
\dot{\pi}\colon V\rightarrow U,
$$
tel que $\dot{\pi}^{*}(\mathcal{L})=\underline{M}_{V},$ où $M$ est encore un $\mathbb{Z}/n\mathbb{Z}$-module fini. Par adjonction, on a un morphisme
$$
\phi\colon \mathcal{F}\rightarrow j_{*}\dot{\pi}_{*}\dot{\pi}^{*}j^{*}(\mathcal{F})\oplus \bigoplus_{i=1}^{n}i_{x_{i} *}i_{x_{i}}^{*}(\mathcal{F})=:\mathcal{G}_{0}.
$$
On affirme que le morphisme $\phi$ est injectif. 
\\
Soient $x\in U$ et $\bar{x}$ un point géométrique au-dessus de $x$ et soit $\bar{y}$ un point géométrique de $V,$ qui s'envoie sur $\bar{x}.$ On obtient ainsi le diagramme commutatif suivant:
$$
\xymatrix{
\mathcal{F}_{\bar{x}} \ar[r] \ar[d]^{=} & (j_{*}\dot{\pi}_{*}(\underline{M}_{V}))_{\bar{x}} \ar[d] \\
(\dot{\pi}^{*}j^{*}(\mathcal{F}))_{\bar{y}} \ar[r]^{=} & (\underline{M}_{V})_{\bar{y}}.
}
$$
Ce diagramme implique alors que le morphisme d'adjonction
$$
\phi_{V}\colon \mathcal{F}\rightarrow j_{*}\dot{\pi}_{*}\dot{\pi}^{*}j^{*}(\mathcal{F})
$$
est injectif pour tout point $x\in U.$ De la même façon, on démontre que les morphismes d'adjonctions
$$
\phi_{x_{i}}\colon \mathcal{F}\rightarrow i_{x_{i} *}i_{x_{i}}^{*}(\mathcal{F})
$$
sont injectifs dans les points $x_{i},$ pour $i=1,...,n.$ Il en résulte que le morphisme $\phi$ est injectif. 
\\
Maintenant on affirme qu'on a
$$
R\Gamma(X^{\text{alg}}_{E,F},\,\mathcal{G}_{0})\simeq R\Gamma(X^{\text{ad}}_{E,F},\,\mathcal{G}_{0}^{\text{ad}}).
$$
Soit $\pi\colon X^{\prime}\rightarrow X^{\text{alg}}_{E,F}$ la normalisation de $X^{\text{alg}}_{E,F}$ dans $V.$ On obtient alors un diagramme 
$$
\xymatrix{
V \ar[r]^{j^{\prime}} \ar[d]^{\dot{\pi}} & X^{\prime} \ar[d]^{\pi}\\
U \ar[r]^{j} & X^{\text{alg}}_{E,F}.
}
$$
On affirme que le morphisme $\pi$ est fini plat: il est fini parce que $X^{\text{alg}}_{E,F}$ admet un recouvrement par des anneaux principaux qui sont de plus des $E$-algèbres; ce sont donc des anneaux de Nagata (\cite[Tag 0335]{stacks}). Le morphisme est plat parce qu'il n'y a pas de torsion dans le faisceau structurel $\mathcal{O}_{X^{\prime}}.$ De plus, $X^{\prime}$ est un schéma intègre, normal (de dimension $1,$ donc régulier) et le résultat \cite[Tag 09Z5]{stacks} implique alors que $j^{\prime}_{*}(\underline{M}_{V})\simeq \underline{M}_{X^{\prime}}.$ On en déduit que
\begin{align*}
j_{*}\dot{\pi}_{*}\dot{\pi}^{*}j^{*}(\mathcal{F}) & =(j\circ \dot{\pi}_{*})(\underline{M}_{V}) \\
 & = (\pi \circ j^{\prime})_{*}(\underline{M}_{V}) \\
 & = \pi_{*}(\underline{M}_{X^{\prime}}).
\end{align*}
On sait maintenant qu'on peut trouver pour tout faisceau constructible une injection vers un faisceau constructible pour lequel on sait déjà que le morphisme de comparaison est un isomorphisme; sous les hypothèses imposées dans l'énoncé. On a alors une suite exacte
$$
\xymatrix{
0 \ar[r] & \mathcal{F} \ar[r] & \mathcal{G}_{0} \ar[r] & \mathcal{G}_{1} \ar[r] & ...,
}
$$
où on sait déjà que $\Phi_{\mathcal{G}_{i}}$ sont des isomorphismes - on en déduit que $\Phi_{\mathcal{F}}$ lui aussi est un isomorphisme. 
\end{proof}
\subsection{GAGA: revêtements ramifiés}\label{subsection GAGA fuer verzweigte Sachen}
Soit $\pi\colon X^{\prime} \rightarrow X^{\alg}_{E,F}$ un revêtement ramifié, c.f. définition \ref{Def verzweigte Ueberdeckung}. D'après la proposition de Huber (\cite[Proposition 3.8.]{HuberFormalScheme}) qu'on vient d'utiliser, on peut construire un espace adique, fortement noethérien $X^{\prime \text{ad}}$ sur $\Spa(E),$ qui se place dans un diagramme commutatif 
$$
\xymatrix{
X^{\prime \text{ad}} \ar[r]^{\pi^{\text{ad}}} \ar[d]^{f^{\prime}} & X^{\text{ad}}_{E,F} \ar[d]^{f} \\
X^{\prime} \ar[r]^{\pi} & X^{\alg}_{E,F},
}
$$
tel que $X^{\prime \text{ad}}$ est caractérisé par une propriété universelle appropriée (c.f. l'énoncé de la proposition de Huber qu'on vient de cité).
De fait, il n'est pas très difficile de décrire l'espace $X^{\prime \text{ad}}$ localement d'une façon explicite: soient $t_{1},t_{2}\in H^{0}(X^{\alg}_{E,F},\mathcal{O}_{X^{\alg}_{E,F}}(1))$ et $D_{+}(t_{i})=\Spec(B_{e_{i}}),$ où $B_{e_{i}}=(B[1/t_{i}])^{\varphi=1},$ tels que 
$$
D_{+}(t_{1})\cup D_{+}(t_{2})=X^{\alg}_{E,F},
$$
i.e. $t_{1},t_{2}$ ne sont pas colinéaires dans le $E$-espace vectoriel $H^{0}(X^{\alg}_{E,F},\mathcal{O}_{X^{\alg}_{E,F}}(1)).$ Soient $I_{1},I_{2}$ des intervalles compacts dans $[0,1],$ tels que $\varphi(I_{i})\cap I_{i}=\emptyset$ et $V(t_{i})\cap Y_{I_{i}}=\emptyset,$ tel que $Y_{I_{1}}\cup Y_{I_{2}}=X^{\text{ad}}_{E,F}.$ Les morphismes $B\rightarrow B_{I_{i}}$ induisent donc des morphismes $B_{e_{i}}\rightarrow B_{I_{i}}.$ Comme le morphisme $\pi\colon X^{\prime}\rightarrow X^{\alg}_{E,F}$ est fini et plat, 
$$
X^{\prime}\times_{X^{\alg}_{E,F}}\Spec(B_{e_{i}})=\Spec(B_{e_{i}}^{\prime}),
$$
où $B_{e_{i}}^{\prime}$ est une $B_{e_{i}}$-algèbre finie plate (donc localement libre, parce que $B_{e_{i}}$ est noethérien). On peut alors considérer $B_{I_{i}}^{\prime}=B_{e_{i}}^{\prime}\otimes_{B_{e_{i}}}B_{I_{i}}.$ C'est une $E$-algèbre topologique munie de la topologie canonique en tant que $B_{I_{i}}$-module fini (en particulier c'est une $E$-algèbre de Banach). Soit $B_{I_{i}}^{\prime +}$ la clôture intégrale de $B_{e_{i}}^{\prime}\otimes_{B_{e_{i}}} B_{I_{i}}^{+}$ dans $B_{I_{i}}^{\prime}.$ Alors $(B_{I_{i}}^{\prime}, B_{I_{i}}^{\prime +})$ est une paire affinoïde de Huber et on a
$$
X^{\prime \text{ad}}\times_{X^{\text{ad}}_{E,F}}\Spa(B_{I_{i}},B_{I_{i}}^{+})=\Spa(B_{I_{i}}^{\prime}, B_{I_{i}}^{\prime +}),
$$
en effet; pour le voir on peut démontrer que $\Spa(B_{I_{i}}^{\prime}, B_{I_{i}}^{\prime +})$ vérifie bien la propriété universelle de la proposition \cite[Proposition 3.8.]{HuberFormalScheme}.
En particulier, on voit très bien que
$$
\pi^{\text{ad}}\colon X^{\prime \text{ad}}\rightarrow X^{\text{ad}}_{E,F}
$$
est un morphisme localement libre, i.e. $\pi^{\text{ad}}_{*}(\mathcal{O}_{X^{\prime}})$ est un fibré vectoriel sur $X^{\text{ad}}_{E,F};$ si $\mathcal{E}^{\prime}$ est un fibré vectoriel sur $X^{\prime \text{ad}},$ alors $\pi^{\text{ad}}_{*}(\mathcal{E}^{\prime})$ l'est aussi.
\begin{proposition}\label{GAGA verzweigte Sachen}
Il y a une équivalence entre les catégories des fibrés vectoriels sur $X^{\prime}$ et $X^{\prime \text{ad}},$ qui est induite par 
$$
f^{\prime *}\colon \text{Fib}(X^{\prime})\rightarrow \text{Fib}(X^{\prime \text{ad}}),
$$
où $f^{\prime}\colon X^{\prime \text{ad}}\rightarrow X^{\prime}$ est le morphismes d'espaces localement annelés introduit ci-dessus.
\end{proposition}

\begin{proof}
On observe d'abord que le foncteur $\pi_{*}(.)$ de $\text{Fib}(X^{\prime})$ vers la catégorie des paires $(\mathcal{E}, \text{act}),$ où $\mathcal{E}$ est un $\mathcal{O}_{X^{\text{alg}}_{E,F}}$-module localement libre et $\text{act}$ est une structure de $\mathcal{A}=\pi_{*}(\mathcal{O}_{X^{\prime}})$-module sur $\mathcal{E},$ est pleinement fidèle. De fait, c'est le cas parce que $\pi\colon X^{\prime}\rightarrow X^{\alg}_{E,F}$ est fini localement libre. De plus, parce que $\pi^{\text{ad}}\colon X^{\prime \text{ad}}\rightarrow X^{\text{ad}}_{E,F}$ lui aussi est fini localement libre, $\pi^{\text{ad}}_{*}(.)$ est un foncteur pleinement fidèle de $\text{Fib}(X^{\prime \text{ad}})$ vers la catégorie des pairs $(\mathcal{E}^{\text{ad}},\text{act}),$ où $\mathcal{E}^{\text{ad}}\in \text{Fib}(X^{\text{ad}}_{E,F})$ et $\text{act}$ est une structure de $\mathcal{A}^{\text{ad}}=\pi^{\text{ad}}_{*}(\mathcal{O}_{X^{\prime \text{ad}}})$-module sur $\mathcal{E}^{\text{ad}}.$ D'après \cite[Prop. II.2.7]{FarguesScholze} il y a une équivalence de GAGA
$$
\Fib(X^{\alg}_{E,F})\simeq \Fib(X^{\text{ad}}_{E,F}).
$$
La donnée de $\text{act}$ est équivalente à la donnée d'un morphisme d'$\mathcal{O}_{X^{\alg}_{E,F}}$-algèbres
$$
\mathcal{A}\rightarrow \mathcal{E}nd_{\mathcal{O}_{X^{\alg}_{E,F}}}(\mathcal{E}).
$$
De la même façon, la donnée d'une action de $\mathcal{A}^{\text{ad}}$ sur $\mathcal{E}^{\text{ad}}$ est équivalente à la donnée d'un morphisme d'$\mathcal{O}_{X^{\text{ad}}_{E,F}}$-algèbres
$$
\mathcal{A}^{\text{ad}}\rightarrow \mathcal{E}nd_{\mathcal{O}_{X^{\text{ad}}_{E,F}}}(\mathcal{E}^{\text{ad}}).
$$
D'après l'équivalence de GAGA pour la courbe, on déduit une équivalence entre les catégories des paires $(\mathcal{E},\text{act})$ et $(\mathcal{E}^{\text{ad}},\text{act}).$ Dans cette étape on a utilisé que $\mathcal{A}$ et $\mathcal{A}^{\text{ad}}$ sont des fibrés vectoriels et que $\mathcal{A}$ s'envoie vers $\mathcal{A}^{\text{ad}}$ dans l'équivalence $
\Fib(X^{\alg}_{E,F})\simeq \Fib(X^{\text{ad}}_{E,F})$.
\\
Il résulte que le foncteur induit par $f^{\prime *}$ 
$$
\Fib(X^{\prime})\rightarrow \Fib(X^{\prime \text{ad}})
$$
est pleinement fidèle. Il faut donc se convaincre qu'il est également essentiellement surjectif. L'image essentielle de $\Fib(X^{\prime})$ dans la catégorie des paires $(\mathcal{E},\text{act})$ sous le foncteur $\pi_{*}(.)$ est caractérisé par la condition que $\mathcal{E}$ soit également fini localement libre comme $\mathcal{A}$-module; l'énoncé pareil est vrai sur la côté adique.
\\
Soit maintenant $(\mathcal{E}^{\text{ad}}, \text{act})$ une paire, tel que $\mathcal{E}^{\text{ad}}$ est aussi fini localement libre comme $\mathcal{A}^{\text{ad}}$-module et $(\mathcal{E},\text{act})$ la paire qui lui correspond. Il faut donc démontrer que $\mathcal{E}$ est aussi un $\mathcal{A}$-module fini localement libre. On observe d'abord qu'on peut trouver un recouvrement affine $X^{\alg}_{E,F}=\bigcup_{i\in I}U_{i},$ tel que $\mathcal{E}\!\mid_{U_{i}}$ est un $\mathcal{A}\! \mid_{U_{i}}$-module fini. De fait cela se déduit de l'observation suivante: soit $R\rightarrow R^{\prime}$ un morphisme d'anneaux et $M$ un $R^{\prime}$-module tel que $M$ est fini comme $R$-module, alors il est fini comme $R^{\prime}$-module. En effet, soit $R^{n}\twoheadrightarrow M$ une surjection, alors en appliquant $.\otimes_{R}R^{\prime}$ on a une surjection $R^{\prime n}\rightarrow M\otimes_{R} R^{\prime}$ de $R^{\prime}$-modules et parce que $M$ est aussi un $R^{\prime}$-module on a une surjection $M\otimes_{R}R^{\prime}\twoheadrightarrow M$ de $R^{\prime}$-modules.  Par le lemme de Nakayama, l'ensemble
$$
\lbrace x\in X^{\alg}_{E,F}\colon \mathcal{E}_{x} \text{ est un }\mathcal{A}_{x}\text{-module libre} \rbrace
$$
est ouvert. 
\\
Ensuite, on réinterprète la paire $(\mathcal{E},\text{act})$ comme un $\mathcal{O}_{X^{\prime}}$-module cohérent $\mathcal{E}^{\prime}$. Il faut démontrer que pour tout point fermé $x^{\prime}\in |X^{\prime}|$ les germes $\mathcal{E}^{\prime}_{x^{\prime}}$ sont libres sur $\mathcal{O}_{X^{\prime},x^{\prime}}$, si on suppose que $f^{\prime *}(\mathcal{E}^{\prime})=\mathcal{E}^{\prime \text{ad}}$ est localement libre comme $\mathcal{O}_{X^{\prime \text{ad}}}$-module.
 On sait qu'on a une bijection entre points fermés $|X^{\text{alg}}_{E,F}|$ et points classiques $|X^{\text{ad}}_{E,F}|^{\text{cl}}.$ Si $x\in |X^{\text{alg}}_{E,F}|$ correspond à $x^{\text{ad}}\in |X^{\text{ad}}_{E,F}|^{\text{cl}},$ il y a une bijection entre la fibre de $\pi\colon X^{\prime}\rightarrow X^{\text{alg}}_{E,F}$ au-dessus du point $x$ et la fibre de $\pi^{\text{ad}}\colon X^{\prime \text{ad}}\rightarrow X^{\text{ad}}_{E,F}$ au-dessus du point $x^{\text{ad}}.$ On trouve ainsi un point maximal $x^{\prime \text{ad}}\in |X^{\prime \text{ad}}|,$ tel que $x^{\prime \text{ad}}\mapsto x^{\prime}.$
\\
Par hypothèse
$$
\mathcal{E}^{\text{ad}}_{x^{\prime \text{ad}}}\simeq \mathcal{E}_{x^{\prime}}\otimes_{\mathcal{O}_{X^{\prime},x^{\prime}}}\mathcal{O}_{X^{\prime \text{ad}},x^{\prime \text{ad}}}
$$
est libre. Le lemme \ref{Technisches Lemma GAGA fpqc} ci-dessous implique que le morphisme $\mathcal{O}_{X^{\prime},x^{\prime}}\rightarrow \mathcal{O}_{X^{\prime \text{ad}},x^{\prime \text{ad}}}$ est fpqc et on peut en déduire que $\mathcal{E}_{x^{\prime}}$ est libre. 
\end{proof}
On a utilisé le lemme suivant:
\begin{Lemma_french}\label{Technisches Lemma GAGA fpqc}
\begin{enumerate}
\item[(a):] L'anneau local $\mathcal{O}_{X^{\prime \text{ad}},x^{\prime \text{ad}}}$ est noethérien.
\item[(b):] Le morphisme $\mathcal{O}_{X^{\prime},x^{\prime}}\rightarrow \mathcal{O}_{X^{\prime \text{ad}},x^{\prime \text{ad}}}$ est fpqc.
\end{enumerate}
\end{Lemma_french}
\begin{proof}
On commence par la preuve de (a). Il suffit de démontrer que $\mathcal{O}_{X^{\text{ad}}_{E,F},x^{\text{ad}}}$ est noethérien pour un point classique. Comme $x^{\prime \text{ad}}$ est un point maximal, l'anneau local $\mathcal{O}_{X^{\prime \text{ad}},x^{\prime \text{ad}}}$ est alors fini au-dessus de l'anneau local $\mathcal{O}_{X^{\text{ad}}_{E,F},x^{\text{ad}}}$ (c.f. \cite[Prop. 1.5.4.]{HuberBuch}).
\\
Ensuite on suit l'argument pour le résultat analogue dans la géométrie rigide classique (c.f. \cite[Prop. 7 dans 7.3.2.]{BoschGuentzerRemmert} ): l'anneau local se calcule dans un voisinage affinoïde $\Spa(B_{I})$ et $x^{\text{ad}}$ correspond ainsi à un ideal maximal $\mathfrak{m}_{x^{\text{ad}}}$ dans $B_{I}.$ La complétion $\mathfrak{m}_{x^{\text{ad}}}$-adique de l'anneau $B_{I}$ est isomorphe à la complétion de l'anneau $\mathcal{O}_{X^{\text{ad}},x^{\text{ad}}};$ c'est simplement $B_{dR}^{+}(\kappa(x^{\text{ad}})).$ Il en résulte que $\mathcal{O}_{\Spa(B_{I}),x^{\text{ad}}}\hookrightarrow B_{dR}^{+}(\kappa(x^{\text{ad}})):$ il suffit de démontrer que si l'image $f_{x}\in \mathcal{O}_{X^{\prime \text{ad}},x^{\text{ad}}}$ d'une fonction $f\in B_{I}$ est dans $\cap_{n} \mathfrak{m}^{n}_{x^{\text{ad}}}\mathcal{O}_{X^{\text{ad}},x^{\text{ad}}},$ alors $f=0.$ Il en utilisant $B_{I}/\mathfrak{m}^{n}_{x^{\text{ad}}}\simeq \mathcal{O}_{X^{\text{ad}},x^{\text{ad}}}/\mathfrak{m}^{n}_{x^{\text{ad}}}\mathcal{O}_{X^{\text{ad}},x^{\text{ad}}},$ on déduit $f\in \cap_{n}\mathfrak{m}^{n}_{x^{\text{ad}}}.$ Comme la localisation $B_{I, \mathfrak{m}_{x^{\text{ad}}}}$ est noethérienne, par le théorème de Krull, on a $f=0.$ Soit $\xi$ un générateur de l'idéal maximal $\mathfrak{m}_{x^{\text{ad}}}.$ L'idéal maximal de l'anneau local $\mathcal{O}_{X^{\text{ad}},x^{\text{ad}}}$ est engendré par l'image de $\xi$ dans cet anneau. Soit $\mathfrak{p}\in \Spec(\mathcal{O}_{X^{\text{ad}},x^{\text{ad}}}),$ tel que $\xi \notin \mathfrak{p}.$ Alors $\mathfrak{p}\subseteq (\xi),$ ce qui implique que pour $a\in \mathfrak{p},$ aussi $a/\xi \in \mathfrak{p},$ i.e. $\xi\mathfrak{p}=\mathfrak{p}.$ Comme $B_{dR}^{+}(\kappa(x^{\text{ad}}))$ est un anneau de valuation discrète, on peut en déduire $\mathfrak{p}=0.$ Cela implique que $\Spec(\mathcal{O}_{X^{\text{ad}},x^{\text{ad}}})=\lbrace 0, (\xi) \rbrace;$ l'anneau local est donc bien un anneau de valuation discrète. 
\\
Pour (b), on observe d'abord que d'après (a) on sait que le morphisme $\mathcal{O}_{X,x}\rightarrow \mathcal{O}_{X^{\text{ad}},x^{ \text{ad}}}$ est fpqc. Soit $\pi^{-1}(x)=\lbrace x_{1}^{\prime},...,x_{n}^{\prime} \rbrace$ et $(\pi^{\text{ad}}) ^{-1}(x^{\text{ad}})=\lbrace x_{1}^{\prime \text{ad}},...,x_{n}^{\prime \text{ad}} \rbrace.$ On note qu'il suffit de démontrer que le morphisme
$$
\prod_{i=1}^{n} \mathcal{O}_{X^{\prime},x^{\prime}_{i}}\rightarrow \prod_{i=1}^{n}\mathcal{O}_{X^{\prime \text{ad}},x^{\prime \text{ad}}_{i}}
$$
est fpqc. Néanmoins cette application s'identifie au changement de base de $\mathcal{O}_{X,x}\rightarrow \mathcal{O}_{X^{\text{ad}},x^{ \text{ad}}}$ le long de $B_{e}\rightarrow B_{e}^{\prime}.$ Elle est donc bien fpqc.
\end{proof}
\subsection{L'annulation implique la comparaison}\label{subsection Verschwinden impliziert Vergleich}
Dans cette section on démontre (c.f. Corollaire \ref{Verschwinden impliziert Vergleich}) que si on savait 
\begin{equation}\label{Gleichung Verschwinden auf beiden Seiten}
H^{i}(X^{\text{alg}}_{E,F},\,\mathcal{F})=H^{i}(X^{\text{ad}}_{E,F},\,\mathcal{F}^{\text{ad}})=0,
\end{equation}
pour $i\geq 3$ et tous les faisceaux constructibles $\mathcal{F}$ sur $X^{\text{alg}}_{E,F},$ la comparaison, Conjecture \ref{Vermutung Vergleich}, est vraie. Plus tard on démontre l'égalité (\ref{Gleichung Verschwinden auf beiden Seiten}) dans le cas de $\ell$-torsion, $\ell\neq p$ (c.f. proposition \ref{l-torsion adische Kurve} pour la courbe adique et proposition \ref{l-kohomologische Dimension algebraische Kurve} pour la courbe algébrique).
\subsubsection{Qu'est-ce qu'il est fait exactement dans cette section?}\label{subsubsection Ueberblick}
Je rappelle d'abord qu'on a déjà réduit la conjecture de comparaison à l'énonce
\begin{equation}\label{Gleichung Vergleich Verzweigte Sachen konstant}
H^{i}(X^{\prime},\mathbb{F}_{\ell})\simeq H^{i}(X^{\prime \text{ad}},\mathbb{F}_{\ell}),
\end{equation}
où $i\geq 0,$ $\pi\colon X^{\prime}\rightarrow X^{\text{alg}}_{E,F}$ est un revêtement ramifié et $\ell$ un nombre premier arbitraire (c.f. lemme \ref{Reduktion von allgemein konstru zu konstant verzweigte}).
\\
Comme on suppose l'égalité (\ref{Gleichung Verschwinden auf beiden Seiten}), on a
$$
H^{i}(X^{\prime},\,\mathbb{F}_{\ell})=H^{i}(X^{\prime \text{ad}},\,\mathbb{F}_{\ell})=0,
$$
pour tout $\ell$ et $i\geq 3.$\footnote{Plus tard on démontre $H^{i}(X^{\prime \text{ad}},\,\mathbb{F}_{p})=0$ pour tout revêtement ramifié sous l'hypothèse (Stein), c.f. proposition \ref{Prop Fall l=p}.}
\\
En utilisant l'équivalence de GAGA pour tout revêtement ramifié (proposition \ref{GAGA verzweigte Sachen}) on voit tout de suite que l'égalité (\ref{Gleichung Vergleich Verzweigte Sachen konstant}) est bien vérifié pour $i=0,1$ (c.f. la preuve de la proposition \ref{Vergleich in Graden 0,1,2 verzweigte Ueberdeckungen} pour plus de détails).
\\
Le point essentiel est ainsi de traiter le cas où $i=2:$ encore par GAGA et la suite exacte de Kummer, il suffit de démontrer que le groupe de Brauer de $X^{\prime \text{ad}}$ est isomorphe au groupe de Brauer cohomologie. Pour cette étape je me suis fortement inspiré d'un résultat de Gabber; il a démontré dans sa thèse \cite{GabberAzumayaAlgebras} le théorème suivant: si $S$ est un schéma séparé qui admet un recouvrement par deux affines, alors on a
$$
\Br(S)\simeq \Br^{\prime}(S)
$$
(c.f. théorème \ref{Gabbers satz ueber Brauergruppen}). Comme $X^{\prime \text{ad}}$ est un espace adique séparé, recouvert par deux affinoides, il est naturel d'essayer d'adapter la preuve du théorème de Gabber. Pour cette étape j'ai trouvé que le mieux est de suivre la preuve du résultat de Gabber donnée par Lieblich dans sa thèse (\cite{LieblichPhD}): il utilise des faisceaux tordus sur des gerbes et dans mon approche je le suis presque mot pour mot.
\subsubsection{Rappels sur le groupe de Brauer et les faisceaux tordus}\label{subusbsection Widerholungen zu Brauergruppe und getwisteten Garben}
Soit $X$ un espace adique fortement noethérien sur un corps non-archimédien $(K,\mathcal{O}_{K}),$ qu'on suppose d'être de caractéristique $0.$\footnote{Pour pouvoir utiliser l'exactitude de la suite exacte de Kummer.} On adopte la définition d'une algèbre d'Azumaya suivante:
\begin{Definition_french}
Soit $\mathcal{A}$ un faisceau analytique d'$\mathcal{O}_{X}$-algèbres. Le faisceau $\mathcal{A}$ s'appelle algèbre d'Azumaya, si le $\mathcal{O}_{X}$-module sous-jacent est fini localement libre et il existe un recouvrement étale $f\colon Y\rightarrow X$ et un fibré vectoriel $\mathcal{E}$ sur $Y,$ tel que
$$
f^{*}(\mathcal{A})\simeq \mathcal{E}nd_{\mathcal{O}_{Y}}(\mathcal{E}).
$$
Le degré d'une algèbre d'Azumaya $\mathcal{A}$ est le rang d'un fibré vectoriel qui trivialise $\mathcal{A}.$
\end{Definition_french}
Comme pour les schémas il existe une description équivalente qui est purement 'cohérente':
\begin{Lemma_french}\label{Andere Beschreibung von Azumaya Algebren}
Soit $\mathcal{A}$ un faisceau analytique en $\mathcal{O}_{X}$-algèbres, tel que le $\mathcal{O}_{X}$-module sous-jacent est fini localement libre. Les deux conditions sont équivalentes:
\begin{enumerate}
\item[(a):] $\mathcal{A}$ est une algèbre d'Azumaya,
\item[(b):] le morphisme canonique d'$\mathcal{O}_{X}$-algèbres
$$
\psi\colon \mathcal{A}\otimes_{\mathcal{O}_{X}}\mathcal{A}^{\text{op}}\rightarrow \mathcal{E}nd_{\mathcal{O}_{X}}(\mathcal{A}),
$$
donné par $a\otimes b\mapsto (x\mapsto axb),$ est un isomorphisme.
\end{enumerate}
\end{Lemma_french}
\begin{proof}
On démontre d'abord que (a) implique (b). Afin d'expliquer que $\psi$ est bien un isomorphisme, on peut travailler sur la localisation stricte de l'espace adique $X$ dans un point géométrique $\xi$ (c.f. \cite[Def. 2.5.11]{HuberBuch}). Sur la localisation stricte on peut utiliser qu'on sait déjà que $\psi$ est un isomorphisme pour les algèbres de matrices usuelles.
\\
(b) implique (a): on peut supposer d'abord que $X$ est affinoïde et soit $x\in X.$ Soit $\bar{x}$ un point géométrique au-dessus de $x$ et soit $X(\bar{x})=(\mathcal{O}_{X(\bar{x})},\mathcal{O}_{X(\bar{x})}^{+})$ la localisation stricte de $X$ au point $\bar{x}.$ L'algèbre $\mathcal{A}\! \mid_{X(\bar{x})}$ correspond à une algèbre sur $\mathcal{O}_{X(\bar{x})},$ tel que le morphisme $\psi$ soit un isomorphisme. Il en résulte, qu'on a un isomorphisme
$$
\mathcal{A}\!\mid_{X(\bar{x})}\simeq \text{Mat}_{m}(\mathcal{O}_{X(\bar{x})}).
$$
On peut donc trouver un voisinage étale affinoïde $U$ de $x,$ tel que $\mathcal{A}\!\mid_{U}\simeq \text{Mat}_{n}(\mathcal{O}_{U}).$ En prenant l'union disjointe de tous ces voisinages étales, on trouve le recouvrement étale cherché.
\end{proof}
On considère le schéma affine en groupes $\PGL_{n, K}$ sur $K.$ Alors on note $$\PGL_{n}^{\text{ad}}=\PGL_{n}\times_{\Spec(K)}\Spa(K,\mathcal{O}_{K}),$$ où on a pris le produit fibré au sens de la proposition \cite[Proposition 3.8.]{HuberFormalScheme}. Pour tout espace affinoïde $\Spa(R,R^{+})$ au-dessus de $\Spa(K,\mathcal{O}_{K}),$ on a
$$
\PGL_{n}^{\text{ad}}(R,R^{+})=\PGL_{n}(R).
$$
\begin{Lemma_french}
Il y a un isomorphisme entre faisceaux étales dans $\widetilde{X}_{\et},$
$$
\PGL_{n, X}^{\text{ad}}=\PGL_{n, K}^{\text{ad}}\times_{\Spa(K,\,\mathcal{O}_{K})} X\simeq \underline{\Aut}(\Mat_{n}(\mathcal{O}_{X_{\et}})).
$$
\end{Lemma_french}
\begin{proof}
Il existe un morphisme entre faisceaux étales dans $\widetilde{X}_{\et},$
$$
\PGL_{n, X}^{\text{ad}}\rightarrow \underline{\Aut}(\Mat_{n}(\mathcal{O}_{X_{\et}}))
$$
et il suffit de démontrer qu'il induit un isomorphisme sur les localisations strictes. Ici on peut utiliser l'énonce analogue pour des schémas.
\end{proof}
Par la descente étale il résulte de ce lemme qu'on a une bijection
$$
\lbrace \text{algèbres d'Azumaya sur }X\text{ de degré }n \rbrace / \text{iso} \simeq H^{1}_{\text{ét}}(X,\PGL_{n, X}^{\text{ad}}).
$$
\begin{Definition_french}
Soit $X$ un espace adique fortement noethérien sur $\Spa(K,\mathcal{O}_{K}).$ Le groupe de Brauer ('cohérent') de $X$ est le groupe (via le produit tensoriel) de toutes les algèbres d'Azumaya sur $X,$ modulo la relation d'équivalence suivante: on dit que $\mathcal{A}\sim \mathcal{B},$ s'il existe des fibrés vectoriels $\mathcal{E},\mathcal{E}^{\prime}$ sur $X,$ et un isomorphisme
$$
\mathcal{A}\otimes_{\mathcal{O}_{X}}\mathcal{E}nd_{\mathcal{O}_{X}}(\mathcal{E})\simeq \mathcal{B}\otimes_{\mathcal{O}_{X}} \mathcal{E}nd_{\mathcal{O}_{X}}(\mathcal{E}^{\prime}).
$$
\end{Definition_french}
On a le résultat suivant de type GAGA pour le groupe de Brauer:
\begin{Lemma_french}\label{GAGA fuer die Brauergruppe}
\begin{enumerate}
\item[(a):] Soit $(A,A^{+})$ une paire affinoïde fortement noethérienne sur $\Spa(K,\mathcal{O}_{K}),$ alors on a $$\Br(\Spec(A))\simeq \Br(\Spa(A,A^{+})).$$
\item[(b):] Soit $\pi\colon X^{\prime}\rightarrow X^{\text{alg}}_{E,F}$ un revêtement ramifié, alors on a
$$
\Br(X^{\prime})\simeq \Br(X^{\prime \ad}).
$$
\end{enumerate}
\end{Lemma_french}
\begin{proof}
En utilisant le lemme \ref{Andere Beschreibung von Azumaya Algebren} ci-dessus, les deux points (a) et (b) résultent des équivalences de catégories $\text{Fib}(\Spa(A,A^{+}))\simeq \text{Fib}(\Spec(A))$ (un résultat de Kedlaya-Liu \cite[Thm. 2.7.7.]{KedlayaLiuFoundations}) et $\text{Fib}(X^{\prime})\simeq \text{Fib}(X^{\prime \text{ad}})$ (Prop. \ref{GAGA verzweigte Sachen}).
\end{proof}
Maintenant on explique le lien avec le groupe de Brauer cohomologique; ce groupe est défini de la façon suivante:
$$
\Br^{\prime}(X):=H^{2}(X,\,\mathbb{G}_{m}^{\text{ad}})_{\text{tors}}.
$$
En utilisant la suite exacte courte 
\begin{equation}
\xymatrix{
1 \ar[r] & \mathbb{G}_{m}^{\text{ad}} \ar[r] & \GL_{n}^{\text{ad}} \ar[r]^{\text{det}} & \PGL_{n}^{\text{ad}} \ar[r] & 1, 
}
\end{equation}
on a une injection\footnote{L'argument usuel s'applique aussi: en négligeant les algèbres d'Azumaya de la forme $\mathcal{E}nd(\mathcal{E})$ on force l'injectivité du morphisme de connexion c.f. \cite[V.4. 4.4.]{GiraudNonAbelienne}}
$$
\iota\colon \Br(X)\rightarrow \Br^{\prime}(X)
$$
et la classe de cohomologie d'une algèbre d'Azumaya $\mathcal{A}$ est notée $[\mathcal{A}]\in \Br^{\prime}(X).$
\\
Le prochain objectif est d'expliquer un critère pour déterminer quand une classe $c\in \Br^{\prime}(X)$ est de la forme $[\mathcal{A}]$ pour une algèbre d'Azumaya $\mathcal{A}$ sur $X.$ Ce critère est formulé en utilisant des faisceaux tordus localement libres sur des gerbes. Comme je n'ai pas trouvé dans la littérature une bonne théorie des champs 'géométriques' pour des espaces adiques il n'existe donc a priori pas une bonne théorie des faisceaux sur ces champs. Il faut faire ainsi un peu attention!
\\
Soit $X$ encore un espace adique fortement noethérien sur $\Spa(K,\mathcal{O}_{K}),$ où $K$ est un corps NA de caractéristique $0.$ Comme Huber dans \cite{HuberBuch} on peut ainsi considérer le petit site étale de $X,$ noté $(X)_{\text{ét}}.$ On suit la définition du stacks-project d'un champ dans le site $(X)_{\text{ét}},$ i.e. on prend la définition \cite[Tag 026F]{stacks} (i.e. une catégorie fibrée en groupoïdes au-dessus de $(X)_{\text{ét}},$ telle que le foncteur des morphismes entre objets dans la catégorie fibrée est un faisceau sur $(X)_{\text{ét}}$ et toute donnée de descente est effective). Un tel champ $\mathfrak{X}$ admet un gros site étale, noté $(\mathfrak{X})_{\text{Ét}}.$ Les objets sont des espaces adiques fortement noethériens $S,$ étales au-dessus de $X,$ avec la donnée d'un morphisme entre catégories fibrées $f\colon S\rightarrow \mathfrak{X}$ (i.e. une section du champ $\mathfrak{X}$) et les morphismes entre $f\colon S\rightarrow \mathfrak{X}$ et $g\colon S^{\prime}\rightarrow \mathfrak{X}$ sont des morphismes $\psi\colon S\rightarrow S^{\prime}$ au-dessus de $\mathfrak{X},$ i.e. un morphisme $S\rightarrow S^{\prime}$ au-dessus de $X$ avec un $2$-isomorphisme 
$$
\phi\colon g \circ \psi \simeq f.
$$ 
Les recouvrements sont engendrés par les recouvrements étales $S\rightarrow S^{\prime}.$ Avec cette définition j'affirme qu'il s'agit bien d'un site: les isomorphismes sont des recouvrements et on peut tester localement sur $S$ la propriété d'être un recouvrement; les recouvrements sont stables par changement de base parce qu'on ne travaille qu'avec des espaces adiques qui sont (localement) fortement noethériens. Sur ce site on a le faisceau des anneaux, noté $\mathcal{O}_{\mathfrak{X}},$ qui associe à $(f\colon S\rightarrow \mathfrak{X})\in (\mathfrak{X})_{\text{Ét}}$ les sections globales $H^{0}(S,\mathcal{O}_{S}).$ On affirme qu'on a un morphisme de topoi annelés
$$
\pi\colon (\widetilde{(\mathfrak{X})_{\text{Ét}}},\mathcal{O}_{\mathfrak{X}})\rightarrow (\widetilde{(X)_{\text{ét}}},\mathcal{O}_{X_{\text{ét}}}).
$$
De fait, si $\mathcal{F}\in \widetilde{(X)_{\text{ét}}},$ alors $\pi^{*}(\mathcal{F})\in \widetilde{(\mathfrak{X})_{\text{Ét}}}$ est le faisceau qui associée à $(S\rightarrow \mathfrak{X})\in (\mathfrak{X})_{\text{Ét}}$ l'objet $\mathcal{F}(S);$ on a donc par définition que $\mathcal{O}_{\mathfrak{X}}=\pi^{*}\mathcal{O}_{X_{\text{ét}}}.$ Si $p\colon \mathfrak{Y}\rightarrow \mathfrak{X}$ est un morphisme entre champs sur $(X)_{\text{ét}},$ on a un morphisme de topoi induit $p\colon \widetilde{(\mathfrak{Y})_{\text{Ét}}}\rightarrow \widetilde{(\mathfrak{X})_{\text{Ét}}}:$ si $\mathcal{F}$ un faisceau dans $\widetilde{(\mathfrak{X})_{\text{Ét}}}$ et $(f\colon S\rightarrow \mathfrak{Y})\in (\mathfrak{Y})_{\text{Ét}}$ un objet dans le site, alors on a $p^{*}\mathcal{F}(S)=\mathcal{F}(p\circ f).$
\\
Maintenant on introduit quelques exemples de champs qui vont être importants pour la suite. 
\begin{Definition_french}
Une $\mathbb{G}_{m}^{\text{ad}}$-gerbe sur $(X)_{\text{ét}}$ est la donnée d'un champ $\mathfrak{X}$ sur $(X)_{\text{ét}}$ et, pour toute section $x\in \mathfrak{X}(S),$ d'un isomorphisme entre faisceaux étales sur $S,$
$$
i_{x}\colon \underline{\Aut}(x)\simeq \mathbb{G}_{m S}^{\text{ad}},
$$
tel que les conditions suivantes sont satisfaites:
\begin{enumerate}
\item[(a):] Soit $S\in (X)_{\text{ét}},$ alors il existe un recouvrement étale $S^{\prime}\rightarrow S,$ tel qu'il existe une section dans $\mathfrak{X}(S^{\prime}).$
\item[(b):] Soit $x,y\in \mathfrak{X}(S),$ il existe un recouvrement étale $S^{\prime}\rightarrow S,$ tel que $x\!\mid_{S^{\prime}}\simeq y\!\mid_{S^{\prime}}$ dans $\mathfrak{X}(S^{\prime}).$
\item[(c):] Soit $x\simeq y$ dans $\mathfrak{X}(S),$ le diagramme suivant est commutatif:
$$
\xymatrix{
\mathbb{G}_{m S}^{ad} \ar[r]^{i_{x}} \ar[d]^{i_{y}}  & \underline{\Aut}(y) \ar[ld] \\
\underline{\Aut}(x).
}
$$
\end{enumerate}
\end{Definition_french}
D'après un théorème de Giraud (c.f. \cite[IV.3.]{GiraudNonAbelienne}) on sait qu'il existe une bijection naturelle entre classes d'isomorphisme de $\mathbb{G}_{m}^{\text{ad}}$-gerbes et classes de cohomologie dans $H^{2}_{\text{ét}}(X,\mathbb{G}_{m}^{\text{ad}}).$ 
\begin{Remark_french}\label{Bemerkung nach Def zu getwisteten Garben}
\begin{enumerate}
\item[(a):] En remplaçant le groupe $\mathbb{G}_{m}^{\text{ad}}$ par le groupe $\mu_{n}^{\text{ad}},$ on obtient la notion d'une $\mu_{n}^{\text{ad}}$-gerbe sur $X.$
\item[(b):] L'exemple de base d'une $\mathbb{G}_{m}^{\text{ad}}$-gerbe est le champ classifiant $B\mathbb{G}_{m}^{\text{ad}},$ qui associe à $S\in (X)_{\text{ét}}$ le groupoïde des fibrés en droites sur $S.$ Localement pour la topologie étale sur $X,$ toute $\mathbb{G}_{m}^{\text{ad}}$-gerbe sur $X$ est de cette forme.
\item[(c):] Soit $\mathcal{A}$ une algèbre d'Azumaya sur $X.$ Soit $\text{Triv}(\mathcal{A})$ le champs sur $(X)_{\text{ét}},$ qui associe à $S\in (X)_{\text{ét}}$ le groupoïde des paires $(\mathcal{W},\psi),$ où $\mathcal{W}$ est un fibré vectoriel sur $S$ et $\psi\colon \mathcal{A}\!\mid_{S}\simeq \mathcal{E}nd_{\mathcal{O}_{S}}(\mathcal{W})$ est un isomorphisme de $\mathcal{O}_{S}$-algèbres. $\text{Triv}(\mathcal{A})$ est ainsi une $\mathbb{G}_{m}^{\text{ad}}$-gerbe sur $X$ (c.f. \cite[V.4. 4.2.]{GiraudNonAbelienne}), telle que
$$
[\text{Triv}(\mathcal{A})]=[\mathcal{A}]\in H^{2}(X,\,\mathbb{G}_{m}^{\text{ad}}),
$$
(c.f. \cite[V.4. Remarque 4.5.]{GiraudNonAbelienne}).
\end{enumerate}
\end{Remark_french}
On observe que sur la $\mathbb{G}_{m}^{\text{ad}}$-gerbe $\text{Triv}(\mathcal{A}),$ il existe une certaine structure supplémentaire; celle d'un fibré vectoriel $\mathcal{W}^{\text{uni}}$ satisfaisant la propriété que les deux actions de $\mathbb{G}_{m}^{\text{ad}},$ l'une donnée par la structure $\mathcal{O}_{\text{Triv}(\mathcal{A})}$-linéaire sur $\mathcal{W}^{\text{uni}}$ et l'autre par l'action inertielle (c.f. en bas), coïncident. Même si cette observation semble un peu triviale, sur des $\mathbb{G}_{m}^{\text{ad}}$-gerbes générales c'est exactement la bonne condition pour vérifier si la classe de cohomologie de cette gerbe est induite par une algèbre d'Azumaya.
\\
Soit maintenant $\mathfrak{X}\rightarrow X$ une $\mathbb{G}_{m}^{\text{ad}}$ (resp. $\mu_{n}^{\text{ad}}$) gerbe sur $X.$ On considère un faisceau $\mathcal{F}$ sur $(\mathfrak{X})_{\text{Ét}}.$ Dans cette situation on a une action
$$
\mathcal{F}\times \mathbb{G}_{m}^{\text{ad}}\rightarrow \mathcal{F}.
$$
De fait, soit $s\colon S\rightarrow \mathfrak{X}$ une section de la gerbe et soit $\lambda_{S}\in \mathbb{G}_{m}^{\text{ad}}(S)\simeq \underline{\text{Aut}}(s)(S).$ Par définition de ce qu'est un faisceau $\mathcal{F}$ sur $(\mathfrak{X})_{\text{Ét}},$ $\lambda_{S}$ induit un isomorphisme
$$
\lambda_{S}^{*}(\mathcal{F})\simeq \mathcal{F}.
$$ Sur les sections cela induit une action
$$
\lambda_{S}^{*}\colon \mathcal{F}(S)\rightarrow \mathcal{F}(S).
$$
Cette action s'appelle l'action inertielle (c.f. \cite[Def. 2.2.1.6.]{LieblichTwistedSheavesCompositio}).
\\
On parvient à la définition d'un faisceau tordu sur une $\mathbb{G}_{m}^{\text{ad}}$-gerbe:
\begin{Definition_french}\label{Def getwistete Garben}
Soit $\mathfrak{X}$ une $\mathbb{G}_{m}^{\text{ad}}$-gerbe (resp. $\mu_{n}^{\text{ad}}$) sur $X.$
\begin{enumerate}
\item[(a):] Un fibré vectoriel (de rang constant $n$) sur $\mathfrak{X}$ est un morphisme de champs sur $(X)_{\text{ét}},$
$$
\mathfrak{X}\rightarrow \Fib_{n},
$$
où $\Fib_{n}$ est le champs sur $(X)_{\text{ét}}$ qui classifie les fibrés vectoriels de rang $n.$
\item[(b):] Soit $\mathcal{E}$ un $\mathcal{O}_{\mathfrak{X}}$-module sur $(\mathfrak{X})_{\text{Ét}}.$ La structure $\mathcal{O}_{\mathfrak{X}}$-linéaire sur $\mathcal{E}$ induit une action par multiplication avec les scalaires
$$
\mathbb{G}_{m}^{\text{ad}}\times \mathcal{E}\rightarrow \mathcal{E}.
$$
Le faisceau $\mathcal{E}$ s'appelle faisceau tordu si l'action sur la gauche associée à l'action inertielle (une action sur la droite) coïncide avec l'action par multiplication avec les scalaires.
\item[(c):] Un faisceau $n$-tordu sur une $\mathbb{G}_{m}^{\text{ad}}$ gerbe $\mathfrak{X}$ sur $X$ est un $\mathcal{O}_{\mathfrak{X}}$-module $\mathcal{E}$ sur $(\mathfrak{X})_{\text{Ét}},$ tel que l'action sur la gauche associée à l'action inertielle coïncide avec l'action par multiplication avec des scalaires relevée à la puissance $n$-ième. Un faisceau $1$-tordu est ainsi un faisceau tordu. Un faisceau localement libre $n$-tordu est un faisceau $n$-tordu $\mathcal{E}$ sur $\mathfrak{X},$ tel que $\mathcal{E}$ est un fibré vectoriel.
\end{enumerate}
\end{Definition_french}
Voici les propriétés de base qui seront nécessaire pour la suite.
\begin{Lemma_french}\label{Lemma zu getwisteten Garben}
Soit $\mathfrak{X}$ une $\mathbb{G}_{m}^{\text{ad}}$-gerbe (resp. une $\mu_{n}^{\text{ad}}$-gerbe) sur $X.$
\begin{enumerate}
\item[(a):] Soient $\mathcal{E}$ resp. $\mathcal{F}$ des faisceaux $n$- resp. $m$-tordus, localement libres sur $\mathfrak{X}.$ Alors $\mathcal{E}\otimes \mathcal{F}$ est un faisceau $n+m$-tordu, localement libre, $\mathcal{H}om_{\mathcal{O}_{\mathfrak{X}}}(\mathcal{E},\mathcal{F})$ est un faisceau $n-m$-tordu, localement libre.
\item[(b):] Le tiré en arrière le long des morphismes de topoi annelés
$$
\pi\colon \widetilde{(\mathfrak{X})_{\text{ét}}}\rightarrow \widetilde{(X)_{\text{ét}}}
$$
induit une équivalence entre la catégorie des faisceaux localement libres $0$-tordus et les $\mathcal{O}_{(X)_{\text{ét}}}$-modules localement libres.
\item[(c):] On peut recoller les faisceaux localement libres, tordus le long des recouvrements analytique de $X.$ Plus précisément: soit $X=U\cup V$ un recouvrement par des ouverts analytiques. On suppose qu'on se donne des faisceaux tordus localement libres $\mathcal{P}\neq 0$ sur $\mathfrak{X}\times_{X}U$ et $\mathcal{Q}$ sur $\mathfrak{X}\times_{X}V,$ tel que 
$$
\mathcal{Q}\!\mid_{U\cap V}\simeq \mathcal{P}\!\mid_{U\cap V}.
$$ 
Alors il existe un unique faisceau tordu, localement libre $\mathcal{M}$ sur $\mathfrak{X},$ tel que $\mathcal{M}\!\mid_{U}\simeq \mathcal{P}$ et $\mathcal{M}\!\mid_{V}\simeq \mathcal{Q}.$
\item[(d):] On suppose que $c\mapsto c^{\prime}$ par $H^{2}(X,\, \mu_{n}^{\text{ad}})\rightarrow H^{2}(X,\,\mathbb{G}_{m}^{\text{ad}})$ et que $\mathfrak{X}_{c}$ resp. $\mathfrak{X}_{c^{\prime}}$ sont des $\mu_{n}^{\text{ad}}$-gerbes resp. $\mathbb{G}_{m}^{\text{ad}}$-gerbes sur $X,$ représentants  de $c$ resp. $c^{\prime}.$ Alors il existe un $1$-morphisme 
$$
F\colon \mathfrak{X}_{c}\rightarrow \mathfrak{X}_{c^{\prime}},
$$
tel que pour toute $x=(f\colon S\rightarrow \mathfrak{X}_{c})\in \mathfrak{X}_{c}(S),$ le morphisme induit par $F$
$$
\underline{\Aut(x)}\simeq \mu_{n, S}^{\text{ad}}\rightarrow \underline{\Aut(F(x))}\simeq \mathbb{G}_{m, S}^{\text{ad}}
$$
s'identifie au changement de base du morphisme canonique $\mu_{n}^{\text{ad}}\rightarrow \mathbb{G}_{m}^{\text{ad}}.$ Le tiré en arrière le long de $F$ induit une équivalence entre la catégorie des faisceaux tordus localement libres sur $\mathfrak{X}_{c^{\prime}}$ et sur $\mathfrak{X}_{c}.$
\end{enumerate}
\end{Lemma_french}
\begin{proof}
Le point (a) est formel. Le point (b) sans la condition 'localement libre' est démontré par Lieblich dans sa thèse, c.f. \cite[Lemma 2.1.1.17.]{LieblichPhD}: la condition d'être $0$-tordu implique simplement que l'action inertielle est triviale. Cela implique que le morphisme d'adjonction (qui vient simplement de la définition d'un morphisme entre topoi annelés)
$$
\pi^{*}\pi_{*}(\mathcal{E})\rightarrow \mathcal{E}
$$
est un isomorphisme; cette étape est expliquée par Lieblich dans la preuve loc.cit.\footnote{Ici $\pi_{*}(\mathcal{E})$ est a priori qu'un $\mathcal{O}_{(X)_{\text{ét}}}$-module (c.f. la définition du stacks project \cite[Tag 03D6]{stacks}).} Il en résulte que $\mathcal{E}$ est de la forme $\pi^{*}(\mathcal{F}),$ où $\mathcal{F}$ est un $\mathcal{O}_{(X)_{\text{ét}}}$-module. J'affirme que $\mathcal{F}$ est un $\mathcal{O}_{(X)_{\text{ét}}}$-module localement libre. Par définition, on peut vérifier ceci localement sur la topologie étale sur $X.$ On peut donc supposer, que la $\mathbb{G}_{m}^{\text{ad}}$-gerbe soit scindée, i.e. de la forme $X\times_{\Spa(K)} B\mathbb{G}_{m}^{\text{ad}}.$ Soit $s\colon X\rightarrow X\times_{\Spa(K)} B\mathbb{G}_{m}^{\text{ad}}$ la section de la projection $\pi\colon X\times_{\Spa(K)} B\mathbb{G}_{m}^{\text{ad}}\rightarrow X.$ Le faisceau localement libre, $0$-tordu $\mathcal{E}$ est maintenant un fibré vectoriel sur $X\times_{\Spa(K)} B\mathbb{G}_{m}^{\text{ad}}.$ On a donc
$$
\mathcal{F}\simeq s^{*}(\pi^{*}(\mathcal{F}))\simeq s^{*}(\mathcal{E})
$$
et $\mathcal{F}$ est bien un fibré vectoriel.
\\
Pour démontrer le point (c), il est préférable de prendre un autre point du vue sur des faisceaux tordus; celui de Caldararu: on pense à la gerbe $\mathfrak{X}$ comme représentant d'une classe de cohomologie $[\mathfrak{X}]\in H^{2}(X,\mathbb{G}_{m}^{\text{ad}}).$ On peut trouver un hyper-recouvrement augmenté\footnote{Je rappelle qu'un hyper-recouvrement augmenté $U_{\bullet}\rightarrow X$ est un object simplicial $U_{\bullet}$ dans $(X)_{\text{ét}}$ avec $U_{-1}=X$ et on exige que pour $n\geq -1,$ 
$$
U_{n+1}\rightarrow (\text{cosk}_{n}(U_{\bullet})_{\leq n})_{n+1}
$$
est un recouvrement pour la topologie étale. } pour la topologie étale $X_{\bullet}\rightarrow X,$ avec une section $\alpha\in H^{0}(X_{2},\mathbb{G}_{m}^{\text{ad}}),$ qui représente la classe $[\mathfrak{X}]$ et telle que le cobord sur $X_{3}$ est trivial (c'est un théorème de Verdier qu'on utilise ici, c.f. \cite[Exposé V.7]{SGA4Tome2}). Un faisceau tordu, localement libre au sens de Caldararu est alors la donnée d'une paire $(\mathcal{F},g),$ où $\mathcal{F}$ est un fibré vectoriel sur $X_{0}$ et 
$$
g\colon \text{pr}_{1}^{(1) *}(\mathcal{F})\simeq \text{pr}_{2}^{(1) *}(\mathcal{F}),
$$
tel que 
$$
\text{pr}_{1}^{(2) *}(g)\circ \text{pr}_{2}^{(2) *}(g)=\alpha \cdot \text{pr}_{3}^{(2) *}(g),
$$
c.f. \cite[Def. 2.1.3.1]{LieblichPhD}. D'après \cite[Prop. 2.1.3.11.]{LieblichPhD}, les catégories des faisceaux tordus localement libres comme dans la définition Déf. \ref{Def getwistete Garben} et les faisceau tordus localement libres au sens de Caldararu sont équivalentes (ici on utilise que le résultat de Lieblich marche dans la généralité d'un topos annelé et que dans l'équivalence qu'il a démontrée les faisceaux localement libres sont identifiés sur les deux côtés).\footnote{La construction du foncteur entre les deux version des faisceaux tordus n'est pas si triviale: d'apres \cite[Prop. 2.1.3.6.]{LieblichPhD} - c'est l'ingrédient non-trivial - un trouve une section $x\colon U_{0}\rightarrow \mathfrak{X},$ tel que les deux tirés en arrières $x_{1},x_{2}\colon U_{1}\rightarrow \mathfrak{X}$ sont isomorphes. Si on fixe un tel isomorphisme $\varphi\colon x_{0}\simeq x_{1},$ on a $\delta(\varphi)=a.$ Si maintenant $\mathcal{E}$ est un faisceau tordu localement libre sur $\mathfrak{X}$ comme défini avant, on peut considérer $\mathcal{F}:=x^{*}(\mathcal{E}).$ Par définition, l'isomorphisme $\varphi$ donne un isomorphisme $\psi\colon \mathcal{F}_{1}\simeq \mathcal{F}_{0}.$ Le cobord est donné par multiplication avec $a,$ parce qu'on a exigé dans la définition d'un faisceau tordu que l'action intertiale coïncide avec la multiplication par les scalaires.} 
\\
En utilisant cette interprétation, il est facile de démontrer que l'on peut recoller les faisceaux tordus localement libres comme cherché: la classe de cohomologie étale $[\mathfrak{X}]\!\mid_{U}\in H^{2}(U,\mathbb{G}_{m}^{\text{ad}})$ (resp. $[\mathfrak{X}]\!\mid_{V}\in H^{2}(V,\mathbb{G}_{m}^{\text{ad}})$ resp. $[\mathfrak{X}]\mid _{U\cap V}\in H^{2}(U\cap V,\mathbb{G}_{m}^{\text{ad}}))$ est représentée par l'hyper-recouvrement augmenté $X_{\bullet}\times_{X} U\rightarrow U$ (resp. pour $V$ resp. $U\cap V$) et la section $\alpha\!\mid_{X_{2}\times_{X}U}\in H^{0}(X_{2}\times_{X}U,\mathbb{G}_{m}^{\text{ad}})$ (resp. $\alpha\!\mid_{X_{2}\times_{X} V}$ resp. $\alpha\!\mid_{X_{2}\times_{X}U\cap V}$).

Étant donné des faisceaux tordus localements libres $\mathcal{P}$ sur $\mathfrak{X}\times_{X}U$ resp. $\mathfrak{X}\times_{X}V;$ ils correspondent à des faisceaux tordus au sens de Caldarau, i.e. à des paires $(\mathcal{F}_{U},g)$ resp. $(\mathcal{F}_{V},h)$ avec un isomorphisme sur $U\cap V.$ Comme $X_{2}$ admet un recouvrement analytique par $X_{2}\times_{X}U$ et $X_{2}\times_{X}V$ avec intersection $X_{2}\times_{X}U\cap V,$ on peut recoller les faisceaux localement libres $\mathcal{P}$ resp. $\mathcal{Q}$ le long de l'isomorphisme donné sur $U\cap V.$
\\
Pour finir la preuve, on aborde la preuve du point (d). L'existence du $1$-morphisme $F\colon \mathfrak{X}_{c}\rightarrow \mathfrak{X}_{c^{\prime}}$ est expliqué dans \cite[Lemma 2.1.2.3]{LieblichPhD}. En utilisant le point de vue des faisceaux tordus au sens de Caldararu, le fait que $F$ induit une équivalence entre faisceaux tordus localement libres devient une trivialité.
\end{proof}
Le lemme suivant est maintenant sans surprise.
\begin{Lemma_french}
Soit $c\in \Br^{\prime}(X)$ et $\mathfrak{X}_{c}\rightarrow X$ une $\mathbb{G}_{m}^{\text{ad}}$-gerbe sur $X,$ représentant de la classe $c.$ Alors $c$ est contenue dans $\Br(X)$ si et seulement s'il existe sur $\mathfrak{X}_{c}$ un faisceau tordu, non-nul, localement libre.
\end{Lemma_french}
\begin{proof}
Si $c=[\mathcal{A}]$ on a $\mathfrak{X}_{c}\simeq \text{Triv}(\mathcal{A})$ et sur $\text{Triv}(\mathcal{A})$ on a le faisceau tordu, non-nul, localement libre $\mathcal{W}^{\text{uni}}$ défini juste après la remarque \ref{Bemerkung nach Def zu getwisteten Garben}. Si $\mathcal{E}$ est un faisceau tordu, non-nul, localement libre sur $\mathfrak{X}_{c},$ alors $\mathcal{E}nd(\mathcal{E})$ est d'après le lemme \ref{Lemma zu getwisteten Garben} (a), (b) équivalent à une $\mathcal{O}_{X}$-algèbre avec $\mathcal{O}_{X}$-module sous-jacent localement libre. Comme la gerbe est étale localement scindée, c'est donc bien une algèbre d'Azumaya sur $X.$
\end{proof}
\subsubsection{Le théorème de Gabber dans le cadre analytique}\label{subsubsection Gabbers satz fuer adische Raueme}
On rappelle d'abord que Gabber a démontré dans sa thèse le résultat suivant:
\begin{Theorem_french}{(Gabber)}\label{Gabbers satz ueber Brauergruppen}
\\
Soit $X$ un schéma séparé\footnote{C'est une hypothèse nécessaire! Un contre-exemple a été donné par Edidin, Hassett, Kresch, Vistoli \cite[Corollary 3.11]{BrauergroupQuotientStacks}.}, qu'on peut écrire
$$
X=U\cup V,
$$
où $U,$ $V$ sont des schémas affines. Alors on a
$$
\Br(X)\simeq \Br^{\prime}(X).
$$
\end{Theorem_french}
La preuve est donnée dans \cite[Chapter II, Theorem 1]{GabberAzumayaAlgebras}. Le plus important, dans cette preuve, est le résultat de recollement suivant (\cite[Lemma 2, Part II]{GabberAzumayaAlgebras} ): soit $c\in \Br^{\prime}(X)$ une classe de cohomologie et soient $\mathcal{A}_{U}$ et $\mathcal{A}_{V}$ sur $U$ resp. $V$ des algèbres d'Azumaya, telles que
$$
c\!\mid_{U}=[\mathcal{A}_{U}] \text{ et } c\!\mid_{V}=[\mathcal{A}_{V}].
$$
Alors on peut trouver une algèbre d'Azumaya $\mathcal{A}$ sur $X,$ telle que $c=[\mathcal{A}].$ Lieblich a donné dans sa thèse \cite{LieblichPhD} une preuve de ce résultat en utilisant des faisceaux tordus: ici on a changé de perspective; au lieu d'essayer recoller des algèbres d'Azumaya (modulo équivalence de Brauer) on veut plutôt recoller des faisceaux tordus (sur des gerbes).
\\
On suit cette approche (i.e. la preuve de la proposition \cite[Prop. 3.1.4.5.]{LieblichTwistedSheavesCompositio}) pour démontrer le résultat suivant:
\begin{proposition}\label{Gabbers resultat fuer adische Raueme}
Soit $X$ un espace adique, fortement noethérien sur $\Spa(K,\mathcal{O}_{K}),$ où $K$ est un corps non archimédien de caractéristique $0,$ tel que
$$
X=U\cup V,
$$
où $U,$ $V$ sont des affinoïdes et $U\cap V$ est toujours affinoïde. Alors on a 
$$
\Br(X)\simeq \Br^{\prime}(X).
$$
\end{proposition}
Maintenant on se tourne vers la preuve de la proposition \ref{Gabbers resultat fuer adische Raueme}. D'abord, on observe qu'il suffit de démontrer le lemme suivant:
\begin{Lemma_french}\label{Reduktionslemma zum Gabber resultat fuer adische Raueme}
Soit $\mathfrak{X}\rightarrow X$ une $\mu_{n}^{\text{ad}}$-gerbe sur $X.$ Sur $\mathfrak{X}$ il existe un faisceau tordu, non-nul, localement libre de rang constant.
\end{Lemma_french}
De fait, soit $c\in \Br^{\prime}(X)$ une classe de cohomologie; elle correspond à une $\mathbb{G}_{m}^{\text{ad}}$-gerbe $\mathfrak{X}_{c}\rightarrow X$ sur $X.$ Par la suite exacte de Kummer, on peut relever la classe $c$ à une classe $c^{\prime}\in H^{2}(X,\mu_{n}^{\text{ad}}),$ pour un $n\in \mathbb{N}.$ Celle-ci correspond à une $\mu_{n}^{\text{ad}}$-gerbe $\mathfrak{X}_{c^{\prime}}\rightarrow X.$ D'après le lemme \ref{Lemma zu getwisteten Garben},(d), on a une équivalence entre faisceaux tordus localement libres sur $\mathfrak{X}_{c}$ et $\mathfrak{X}_{c^{\prime}}.$ On peut déduire du lemme \ref{Reduktionslemma zum Gabber resultat fuer adische Raueme} ci-dessus l'existence d'un faisceau tordu, non-nul, localement libre de rang constant sur $\mathfrak{X}_{c};$ en passant à l'algèbre des endomorphismes on en déduit l'existence d'une algèbre d'Azumaya $\mathcal{A}$ sur $X,$ tel que $c=[\mathcal{A}].$
\\
Il suffit alors de démontrer le lemme \ref{Reduktionslemma zum Gabber resultat fuer adische Raueme} ci-dessus.
\\
Avant d'attaquer la proposition \ref{Gabbers resultat fuer adische Raueme} directement, on va d'abord traiter le cas affinoïde; ici on utilise de fait le résultat de Gabber ci-dessus comme ingrédient. 
\begin{Lemma_french}\label{Gabbers theorem affinoide adische Raeume}
Soit $\Spa(A,A^{+})$ un espace adique affinoïde, fortement noethérien sur $\Spa(K,\mathcal{O}_{K}).$ Alors on a
$$
\Br(\Spa(A,A^{+}))\simeq \Br^{\prime}(\Spa(A,A^{+})).
$$
\end{Lemma_french}
\begin{proof}
On observe d'abord qu'il y a le diagramme commutatif suivant
$$
\xymatrix{
\Br(\Spa(A,A^{+})) \ar[r] & \Br^{\prime}(\Spa(A,A^{+})) \\
\Br(\Spec(A)) \ar[u] \ar[r] & \Br^{\prime}(\Spec(A)) \ar[u]
}
$$
et il faut démontrer que la flèche horizontale en haut est un isomorphisme: par le résultat de Gabber en haut (Théorème \ref{Gabbers satz ueber Brauergruppen}) la flèche $\Br(\Spec(A))\rightarrow \Br^{\prime}(\Spec(A))$ est un isomorphisme et d'après GAGA, lemme \ref{GAGA fuer die Brauergruppe}, la flèche $\Br(\Spec(A))\rightarrow \Br(\Spa(A,A^{+}))$ est un isomorphisme. Il suffit alors de démontrer que $\Br^{\prime}(\Spec(A))\simeq \Br^{\prime}(\Spa(A,A^{+})).$
Pour cela, on rappelle le résultat de comparaison crucial de Huber: comme le faisceau $\mu_{n}^{\text{ad}}$ est par définition le tiré en arrière du faisceau $\mu_{n}$ le long du morphisme $\Spa(A,A^{+})\rightarrow \Spec(A)$ et l'anneau de Tate $A$ est par hypothèse complet, on a
$$
H^{i}(\Spa(A,A^{+}),\,\mu_{n}^{\text{ad}})\simeq H^{i}(\Spec(A),\,\mu_{n})
$$
par \cite[Thm. 3.2.9]{HuberBuch}.\footnote{Dans l'application cherchée on a seulement besoin de l'énoncé du lemme dans le cas où $K$ est de plus algébriquement clos, on peut donc trivialiser $\mu_{n}$ et utiliser directement le résultat comme énoncé par Huber dans \cite[Corollary 3.2.2.]{HuberBuch}.} Maintenant on peut comparer les deux suites exactes de Kummer et en déduire un diagramme commutatif
$$
\xymatrix{
0 \ar[r] & \Pic(\Spec(A))_{[n]} \ar[r] \ar[d] & H^{2}(\Spec(A),\,\mu_{n}) \ar[r] \ar[d] & \Br^{\prime}(\Spec(A))[n]\ar[d] \ar[r] & 0 \\
0 \ar[r] & \Pic(\Spa(A,A^{+}))_{[n]} \ar[r] & H^{2}(\Spa(A,A^{+}),\, \mu_{n}^{ad}) \ar[r] & \Br^{\prime}(\Spa(A,A^{+}))[n] \ar[r] & 0,
}
$$
pour tout $n.$ En utilisant le résultat de Kedlaya-Liu \cite[Thm. 2.7.7.]{KedlayaLiuFoundations}), on voit que
$$
\Pic(\Spec(A))_{[n]}\simeq \Pic(\Spa(A,A^{+}))_{[n]}.
$$
Finalement, on a
$$
\Br^{\prime}(\Spec(A))\simeq \Br^{\prime}(\Spa(A,A^{+})),
$$
d'où le lemme.
\end{proof}
Maintenant on observe le corollaire suivant du lemme \ref{Gabbers theorem affinoide adische Raeume}, qui concerne encore le cas affinoïde:
\begin{Corollary_french}\label{Reduktionslemme zum Gabber resultat fuer adische Raueme affinoid}
Soit $\Spa(A,A^{+})$ un espace adique affinoïde, fortement noethérien sur $\Spa(K,\mathcal{O}_{K})$ et soit $\mathfrak{X}\rightarrow \Spa(A,A^{+})$ une $\mu_{n}^{\text{ad}}$-gerbe. Alors il existe un faisceau tordu, non-nul, localement libre et de rang constant sur $\mathfrak{X}.$ 
\end{Corollary_french}
\begin{proof}
Ici on utilise simplement 'l'invariance sous le groupe' pour des faisceaux tordus localement libres, \ref{Lemma zu getwisteten Garben}, (d), et le lemme \ref{Gabbers theorem affinoide adische Raeume} qu'on vient de démontrer.
\end{proof}
Le cas affinoïde compris, on tourne vers la preuve du lemme \ref{Reduktionslemma zum Gabber resultat fuer adische Raueme}.  A partir de maintenant, on suit la preuve donné par Lieblich dans \cite[Prop.3.1.4.5.]{LieblichTwistedSheavesCompositio} presque mot pour mot. Comme lui, on a besoin d'une petite injection de la K-théorie:
\begin{Lemma_french}\label{Elementares Lemma K-Theorie}\footnote{Même si ce lemme a l'air élémentaire, pour la preuve on doit utiliser des résultats de K-théorie; notamment le fait que le noyau de l'application 'rang' est un nil-idéal.}
Soit $R$ un anneau commutatif avec unité.
\begin{enumerate}
\item[(a):] Soit $P\neq 0$ un $R$-module projectif, de rang constant. Si $P^{\otimes n}$ est libre pour un $n\neq 0 \in \mathbb{N},$ alors il existe $m\neq 0 \in \mathbb{N},$ tel que $P^{\oplus m}$ est libre.
\item[(b):] Soit $P\neq 0$ de nouveau un $R$-module projectif, de rang constant et $n\in \mathbb{N}$ donné. Alors on peut trouver des $R$-modules non-nuls, finis libres $F_{0},F_{1}$ et un $R$-module projectif, de rang constant $\overline{P}\neq 0,$ tels que
$$
P\otimes_{R}\otimes \overline{P}^{\otimes n}\otimes_{R}F_{0}\simeq F_{1}.
$$
\end{enumerate}
\end{Lemma_french}
Pour la preuve c.f. \cite[Lemma 3.1.4.3 et Corollary 3.1.4.4]{LieblichTwistedSheavesCompositio}.
\\
Finalement, on peut commencer la preuve du lemme \ref{Reduktionslemma zum Gabber resultat fuer adische Raueme}: soit $\mathfrak{X}\rightarrow X$ une $\mu_{n}^{\text{ad}}$-gerbe sur $X.$ On note $U=\Spa(A_{1},A_{1}^{+}),$ $V=\Spa(A_{2},A_{2}^{+})$ et $U\cap V=\Spa(A_{3},A_{3}^{+}).$ On peut prendrer la restriction de la gerbe donnée et considérer $\mathfrak{X}\!\mid_{U}\rightarrow U$ et $\mathfrak{X}\!\mid_{V}\rightarrow V.$ D'après le corollaire \ref{Reduktionslemme zum Gabber resultat fuer adische Raueme affinoid}, on peut trouver des faisceaux tordus, non-nuls, localement libres de rangs constants $\mathcal{P}$ sur $\mathfrak{X}\!\mid_{V}$ et $\mathcal{Q}$ sur $\mathfrak{X}\!\mid_{U}.$ L'objectif est bien sûr de les recoller pour construire le faisceau tordu sur $\mathfrak{X}\rightarrow X.$
\\
En considérant des sommes directes de copies de $\mathcal{P}$ resp. $\mathcal{Q},$ on peut supposer que $\text{rg}(\mathcal{P})=\text{rg}(\mathcal{Q}).$ Comme ces sont des faisceaux tordus sur des $\mu_{n}^{\text{ad}}$-gerbes, $\mathcal{P}^{\otimes n}$ resp. $\mathcal{Q}^{\otimes n}$ sont $0$-tordus; d'après le lemme \ref{Lemma zu getwisteten Garben},(b), $\mathcal{P}^{\otimes n}$ (resp. $\mathcal{Q}^{\otimes n})$ correspond donc à un $A_{1}$-module (resp. $A_{2}$-module) non-nul, fini projectif de rang constant. On se réduit d'abord au cas où ce sont des modules libres.
\\
Par lemme \ref{Elementares Lemma K-Theorie}, (b), on peut trouver des $A_{1}$-modules finis libres $F_{0},F_{1}$ (resp. des $A_{2}$-modules finis libres $G_{0},G_{1}$) et un $A_{1}$-module $\overline{P}\neq 0$ fini projectif de rang constant (resp. un $A_{2}$-module $\overline{Q}\neq 0$ fini projectif de rang constant), tel qu'on a
\begin{equation}
\mathcal{P}^{\otimes n}\otimes_{A_{1}}\overline{P}^{\otimes n} \otimes_{A_{1}} F_{0} \simeq F_{1}
\end{equation}
(resp. 
\begin{equation}
\mathcal{Q}^{\otimes n}\otimes_{A_{2}} \overline{Q}^{\otimes n}\otimes_{A_{2}}G_{0} \simeq G_{1}.)
\end{equation}
On considère maintenant les $A_{1}$-modules $F_{0},F_{1},\overline{P}$ (resp. les $A_{2}$-modules $G_{0},G_{1},\overline{Q}$) comme des fibrés vectoriels sur $U$ (resp. $V$); en utilisant le lemme \ref{Lemma zu getwisteten Garben},(b), on peut les considérer aussi comme des faisceaux localement libres $0$-tordus sur la gerbe $\mathfrak{X}\!\mid_{U}$ (resp. $\mathfrak{X}\!\mid_{V}$). D'après le lemme \ref{Lemma zu getwisteten Garben}, (a), on peut remplacer le faisceau tordu $\mathcal{P}$ (resp. $\mathcal{Q}$) par le faisceau tordu $\mathcal{P} \otimes \overline{P} \otimes F_{0}$ (resp. $\mathcal{Q}\otimes \overline{Q} \otimes G_{0}$); ces faisceaux tordus vérifient bien que leur $n$-ième puissance tensorielle soit finie libre.
\\
Maintenant on veut expliquer pourquoi on peut recoller ces faisceaux tordus: sur l'intersection $U\cap V$ on peut considérer les faisceaux $1-1=0$-tordus suivants:
$$
P:=\mathcal{Q}^{\vee}\!\mid_{U\cap V}\otimes \mathcal{Q}\!\mid_{U\cap V},
$$
et
$$
Q:=\mathcal{P}\!\mid_{U\cap V}\otimes \mathcal{Q}^{\vee}\!\mid_{U\cap V}.
$$
Ici on note par $(.)\!\mid_{U\cap V}$ le tiré en arrière des faisceaux tordus localement libres le long du morphisme de $\mu_{n}^{\text{ad}}$-gerbes $\mathfrak{X}\!\mid_{U\cap V}\rightarrow \mathfrak{X}\!\mid_{U}$ resp. $\mathfrak{X}\!\mid_{U\cap V}\rightarrow \mathfrak{X}\!\mid_{V}.$ Trivialement, on a un isomorphisme entre faisceaux tordus
\begin{equation}\label{Gleichung Isomorphismus zum Verkleben der getwisteten Garben}
\mathcal{P}\!\mid_{U\cap V}\otimes P \simeq \mathcal{Q}\!\mid_{U\cap V}\otimes Q.
\end{equation}
On l'utiliser pour construire la donnée de recollement: on observe d'abord qu'après la réduction qu'on vient d'effectuer, on est dans une situation où les $A_{3}$-modules projectifs $P^{\otimes n}$ et $Q^{\otimes n}$ sont en effet libres. D'après le lemme \ref{Elementares Lemma K-Theorie} (a), il résulte qu'on peut trouver un $m>>0,$ tel que $P^{\oplus m}$ et $Q^{\oplus m}$ sont libres. L'isomorphisme (\ref{Gleichung Isomorphismus zum Verkleben der getwisteten Garben}), implique un isomorphisme
$$
\mathcal{P}\!\mid_{U\cap V}\otimes P^{\oplus m} \simeq \mathcal{Q}\!\mid_{U\cap V}\otimes Q^{\oplus m},
$$
i.e. pour un $M\geq 0,$ on a
$$
\mathcal{P}^{\oplus M}\!\mid_{U\cap V}\simeq \mathcal{Q}^{\oplus M}\!\mid_{U\cap V}.
$$
Comme on sait d'après le lemme \ref{Lemma zu getwisteten Garben},(c), que l'on peut recoller les faisceaux tordus le long du recouvrement $X=U\cup V,$ on peut appliquer ce fait aux faisceaux tordus $\mathcal{P}^{\oplus M}$ et $\mathcal{Q}^{\oplus M}$ en utilisant l'isomorphisme qu'on vient de construire. Cela permet de finir la preuve du lemme \ref{Reduktionslemma zum Gabber resultat fuer adische Raueme} et donc finalement aussi la preuve de la proposition \ref{Gabbers resultat fuer adische Raueme}.
\subsubsection{L'annulation implique la comparaison: Preuve}\label{subsubsection Verschwinden impliziert Vergleich Beweise}
\begin{proposition}\label{Vergleich in Graden 0,1,2 verzweigte Ueberdeckungen}
Soit $\ell$ un nombre premier et $\pi\colon X^{\prime}\rightarrow X^{\text{alg}}_{E,F}$ un revêtement ramifié. Le morphisme de comparaison induit un isomorphisme
$$
H^{i}(X^{\prime},\,\mathbb{F}_{\ell})\simeq H^{i}(X^{\prime \text{ad}},\,\mathbb{F}_{\ell}),
$$
pour $i=0,1,2.$
\end{proposition}
\begin{proof}
On commence par le cas $i=0:$ ici on a $H^{0}(X^{\prime},\,\mathbb{F}_{\ell})=(\mathbb{F}_{\ell})^{\pi_{0}(X^{\prime})}$ et $H^{0}(X^{\prime \text{ad}},\,\mathbb{F}_{\ell})=(\mathbb{F}_{\ell})^{\pi_{0}(X^{\prime \text{ad}})}.$ Il faut donc expliquer pourquoi il existe une bijection entre les composantes connexes de $X^{\prime}$ et $X^{\prime \text{ad}}.$ Comme le schéma $X^{\prime}$ est noethérien, les composantes connexes sont en bijection avec les idempotents indécomposables dans $H^{0}(X^{\prime},\mathcal{O}_{X^{\prime}}).$ De plus, comme l'espace adique $X^{\prime \text{ad}}$ est localement connexe, il en résulte qu'il y a aussi une bijection entre les composantes connexes de $X^{\prime \text{ad}}$ et les idempotentes indécomposables dans $H^{0}(X^{\prime \text{ad}},\mathcal{O}_{X^{\prime \text{ad}}}).$ Finalement, on peut appliquer GAGA pour la courbe et le fait que les morphismes $\pi$ et $\pi^{\text{ad}}$ sont finis localement libres, pour en déduire la suite des égalités suivantes:
\begin{align*}
H^{0}(X^{\prime},\,\mathcal{O}_{X^{\prime}}) & = H^{0}(X^{\text{alg}}_{E,F},\,\pi_{*}\mathcal{O}_{X^{\text{alg}}_{E,F}}) \\
& \simeq H^{0}(X^{\text{ad}}_{E,F},\,\pi^{\text{ad}}_{*}\mathcal{O}_{X^{\text{ad}} E,F}) \\
& = H^{0}(X^{\prime \text{ad}},\,\mathcal{O}_{X^{\prime \text{ad}}}).
\end{align*}
La bijection entre les idempotents indécomposables en résulte.
\\
Pour le cas $i=1,$ on affirme qu'il existe une équivalence de catégories entre la catégorie des revêtements finis étales de $X^{\prime}$ et de $X^{\prime \text{ad}}.$ Cela suffit pour comparer les $\mathbb{F}_{\ell}$-torseurs sur $X^{\prime}$ et $X^{\prime \text{ad}}.$ On observe d'abord que l'équivalence GAGA pour $X^{\prime}$ resp. $X^{\prime \text{ad}}$ implique une équivalence entre revêtements finis localement libres de $X^{\prime}$ et $X^{\prime \text{ad}};$ il faut donc expliquer pourquoi la propriété d'être fini étale se correspond. Soit $f\colon Y\rightarrow X^{\prime}$ un morphisme fini localement libre et $f^{\text{ad}}\colon Y^{\text{ad}}\rightarrow X^{\prime \text{ad}}$ l'adification associé. Si $f$ est fini étale, le morphisme $f^{\text{ad}}$ l'est aussi. On suppose maintenant que $f^{\text{ad}}$ est fini étale et on démontre que $f$ est aussi fini étale. Comme $f$ est plat, il suffit de démontrer que $\Omega_{Y/X^{\prime}}=0.$ Comme le lieu dans $Y,$ où $f$ est non-ramifié est un ouvert et $Y$ est localement noethérien\footnote{Cela implique qu'un point quelconque se spécialise dans un point fermé \cite[02IL]{stacks}.}, il suffit de démontrer que $\Omega_{Y/X^{\prime}, y_{0}}=0$ pour tout point fermé $y_{0}\in |Y|.$ 
\\
Maintenant, on observe qu'on peut trouver un point maximal $y^{\text{ad}}\in |Y^{\text{ad}}|,$ tel que $y^{\text{ad}}\mapsto y_{0}.$ De fait, $f(y_{0})$ est un point fermé dans $X^{\prime}$ et on trouve un point maximal $x^{\prime \text{ad}}\in |X^{\prime \text{ad}}|,$ tel que $x^{\prime \text{ad}}\mapsto f(y_{0})$ et $x^{\prime \text{ad}}$ s'envoie vers un point classique dans $X^{\text{ad}}_{E,F}.$ Maintenant il faut utiliser qu'il existe une bijection entre la fibre du morphisme $f\colon Y\rightarrow X^{\prime}$ au-dessus du point $f(y_{0})$ et la fibre du morphisme $f^{\text{ad}}\colon Y^{\text{ad}}\rightarrow X^{\prime \text{ad}}$ au-dessus du point $x^{\prime \text{ad}}.$
\\
Pour finir la preuve, il suffit de démontrer que
$$
(\Omega_{Y/X^{\prime}})^{\text{ad}}\simeq \Omega_{Y^{\text{ad}}/X^{\prime \text{ad}}}.
$$
De fait, cet isomorphisme implique
$$
(\Omega_{Y/X^{\prime}})_{y_{0}}\otimes_{\mathcal{O}_{Y,y_{0}}}\mathcal{O}_{Y^{\text{ad}},y^{\text{ad}}}\simeq (\Omega_{Y^{\text{ad}}/X^{\prime \text{ad}}})_{y^{\text{ad}}}=0.
$$
Comme le morphisme $\mathcal{O}_{Y,y_{0}}\rightarrow \mathcal{O}_{Y^{\text{ad}},y^{\text{ad}}}$ est fidèlement plat (c.f. la preuve du lemme \ref{Technisches Lemma GAGA fpqc}), il en résulte que $(\Omega_{Y/X^{\prime}})_{y_{0}}=0.$
\\
Pour démontrer l'isomorphisme ci-dessus, il suffit de le faire localement sur $Y^{\text{ad}}.$ On travaille sur le recouvrement affinoïde suivant: soit $\lbrace \Spec(B_{e^{\prime}_{i}}) \rbrace_{i}$ un recouvrement affine du schéma $X^{\prime}$ et $\lbrace \Spa(B^{\prime}_{i}) \rbrace_{i}$ un recouvrement affinoïde, tel que $B_{e^{\prime}_{i}}\rightarrow B^{\prime}_{i}$ (ici on utilise le recouvrement construit dans la discussion qui a précédé la preuve du lemme \ref{GAGA verzweigte Sachen}). Soit $\lbrace f^{-1}(\Spec(B_{e^{\prime}_{i}})) \rbrace_{i}=\lbrace \Spec(A_{i}) \rbrace_{i},$ un recouvrement de $Y.$ $\lbrace \Spa(A\otimes_{B_{e^{\prime}_{i}}} B^{\prime}_{i}) \rbrace_{i}$ est alors un recouvrement de $Y^{\text{ad}}.$ En utilisant que dans cette situation, on a
$$
\Omega_{\Spa(A\otimes_{B_{e^{\prime}_{i}}} B^{\prime}_{i})/\Spa(B^{\prime}_{i})}\simeq \Omega_{A\otimes_{B_{e^{\prime}_{i}}} B^{\prime}_{i}/B^{\prime}_{i}}
$$
(i.e. la topologie ne joue pas de rôle, c.f. \cite[Construction (1.6.2.),(iii)]{HuberBuch}) on peut bien vérifier l'isomorphisme cherché.
\\
Finalement, pour $i=2,$ on a besoin d'un isomorphisme $\mathbb{F}_{\ell}\simeq \mu_{\ell}.$ Soit donc $E^{\prime}$ une extension galoisienne de $E,$ qui contient une racine primitive $\ell$-ième de l'unité. Il y a deux suites spectrales: 
$$
H^{i}(\Gal(E^{\prime}/E),\,H^{j}(X^{\prime}_{E^{\prime}},\mathbb{F}_{\ell})\Rightarrow H^{i+j}(X^{\prime},\,\mathbb{F}_{\ell})
$$
et
$$
H^{i}(\Gal(E^{\prime}/E),\,H^{j}(X^{\prime ad}_{E^{\prime}},\mathbb{F}_{\ell}))\Rightarrow H^{i+j}(X^{\prime ad},\,\mathbb{F}_{\ell}).
$$
Il suffit donc de démontrer qu'il y a un isomorphisme entre $\Gal(E^{\prime}/E)$-modules 
$$
H^{j}(X^{\prime}_{E^{\prime}},\,\mathbb{F}_{\ell})\simeq H^{j}(X^{\prime \text{ad}}_{E^{\prime}},\,\mathbb{F}_{\ell}),
$$
pour $j=0,1,2.$ D'après ce qu'on a déjà expliqué avant, il suffit d'expliquer le cas $j=2.$ Néanmoins ici on peut utiliser l'identification entre le groupe de Brauer cohomologique et cohérent de $X^{\prime \text{ad}}_{E^{\prime}}$ (proposition \ref{Gabbers resultat fuer adische Raueme}); ce qui implique
$$
\Br^{\prime}(X^{\prime}_{E^{\prime}})\simeq \Br^{\prime}(X^{\prime \text{ad}}_{E^{\prime}}),
$$
GAGA pour $X^{\prime}_{E^{\prime}}$ (proposition \ref{GAGA verzweigte Sachen}, lemme \ref{GAGA fuer die Brauergruppe} (b)) et la suite exacte de Kummer, pour déduire
$$
H^{2}(X^{\prime}_{E^{\prime}},\,\mu_{\ell})\simeq H^{2}(X^{\prime \text{ad}}_{E^{\prime}},\,\mu_{\ell}).
$$
\end{proof}
\begin{Corollary_french}\label{Verschwinden impliziert Vergleich}
Si pour tout revêtement ramifié $X^{\prime}\rightarrow X^{\text{alg}}_{E,F},$ on a
$$
H^{i}(X^{\prime},\,\mathbb{F}_{\ell})=0,
$$
pour $i\geq 3,$ alors le morphisme de comparaison est un isomorphisme, i.e.
$$
H^{i}(X^{\alg}_{E,F},\,\mathcal{F})\simeq H^{i}(X^{\ad}_{E,F},\,\mathcal{F}^{\ad}),
$$
pour $i\geq 0$ et $\mathcal{F}$ un $\mathbb{F}_{\ell}$-module constructible sur le site étale de $X^{\text{alg}}_{E,F}.$ 
\end{Corollary_french}
\begin{proof}
Cela résulte directement du lemme \ref{Reduktion von allgemein konstru zu konstant verzweigte} et le lemme \ref{Vergleich in Graden 0,1,2 verzweigte Ueberdeckungen}, qu'on vient d'expliquer. \footnote{On rappelle qu'ici qu'on est en train d'utiliser des résultats d'annulation qu'on explique un peu plus tard; notamment prop. \ref{l-kohomologische Dimension algebraische Kurve}, prop. \ref{l-torsion adische Kurve}, la remarque \ref{Quatschbemerkung zu l=p und Grad} juste après prop. \ref{Prop Fall l=p}.}
\end{proof}
\begin{Remark_french}
Cette comparaison est aussi vraie pour les adifications des ensembles ouverts $U\subseteq X^{\text{alg}}_{E,F}$ avec coefficients $\mathbb{F}_{\ell},$ où $\ell \neq p,$ i.e. on a
$$
R\Gamma(U,\,\mathbb{F}_{\ell})\simeq R\Gamma(U^{\text{ad}},\,\mathbb{F}_{\ell}).
$$
Néanmoins on n'a pas besoin de cette observation pour la suite.
\\
De fait, soient $i\colon Z\hookrightarrow X^{\text{alg}}_{E,F}$ et $i^{\text{ad}}\colon Z^{\text{ad}}\hookrightarrow X^{\text{ad}}_{E,F}$ les inclusions des fermés complémentaires. Comme d'habitude, parce qu'on a déjà établi la comparaison $R\Gamma(X^{\text{alg}}_{E,F},\,\mathbb{F}_{\ell})\simeq R\Gamma(X^{\text{ad}}_{E,F},\,\mathbb{F}_{\ell})$, il suffit de vérifier la pureté cohomologique pour $i$ resp. $i^{\text{ad}},$ i.e. il s'agit d'expliquer pourquoi on a des isomorphismes $Ri^{!}\mathbb{F}_{\ell}(1)[2]\simeq \mathbb{F}_{\ell}$ et $Ri^{\text{ad} !}\mathbb{F}_{\ell}(1)[2]\simeq \mathbb{F}_{\ell}.$ L'isomorphisme $Ri^{!}\mathbb{F}_{\ell}(1)[2]\simeq \mathbb{F}_{\ell}$ résulte de la pureté absolue de Gabber \cite{FujiwaraPurity}.\footnote{Comme on travaille ici qu'en dimension $1$ on n'a pas vraiment besoin de toute la force de la pureté absolue et dans SGA 5 I. 5.1. on peut trouver un argument plus facile.} Pour l'isomorphisme dans le monde adique, on observe d'abord qu'il suffit de démontrer la pureté pour une paire $(Y_{I},Z^{\text{ad}}),$ où $Y_{I}$ est une couronne affinoïde suffisamment petite, telle que $Z^{\text{ad}}\subset Y_{I}.$ Maintenant, il suffit de démontrer que 
$$
Rf^{\lozenge !}\mathbb{F}_{\ell}\simeq \mathbb{F}_{\ell}(1)[2],
$$
où $f\colon Y_{I}^{\lozenge}\rightarrow \Spa(F)$ est le morphisme structurel vers $\Spa(F)$ du diamant. En utilisant la présentation $Y_{I}^{\lozenge}=A_{I}/\underline{\mathbb{Z}_{p}},$ où $A_{I}$ est la perfection d'une couronne dans la boule ouverte épointée sur $\Spa(F),$ on peut déduire l'isomorphisme $Rf^{\lozenge !}\mathbb{F}_{\ell}\simeq \mathbb{F}_{\ell}(1)[2]$ du fait qu'on connait déjà le complexe dualisant sur $A_{I}$ (\cite[Thm. 24.1]{ScholzeDiamonds}) en utilisant le résultat \cite[Prop. 24.2.]{ScholzeDiamonds} (on rappelle ici qu'en utilisant le dernier résultat on doit choisir une mesure de Haar sur $\mathbb{Z}_{p}$).
\end{Remark_french}
\section{Quelques résultats sur l'annulation de la cohomologie}\label{section Einige resultate Verschwinden}
\subsection{$\ell$-dimension cohomologique de $X^{\text{ad}}_{E,F}$}\label{subsection ell koho Dim der adischen Kurve}
Soit $\ell$ un nombre premier, $\ell \neq p.$ L'objectif ici est de démontrer le résultat suivant:
\begin{proposition}\label{l-torsion adische Kurve}
Soit $\mathcal{F}$ un $\mathbb{F}_{\ell}$-module dans $\widetilde{(X^{\text{ad}}_{E,F})_{\text{ét}}}.$ Alors on a
$$
H^{i}_{\text{ét}}(X^{\text{ad}}_{E,F},\,\mathcal{F})=0
$$
pour $i\geq 3.$
\end{proposition}
Pour la preuve on utilise la présentation du diamant $(X^{\text{ad}}_{E,F})^{\lozenge}$ suivante: soit $E(\mu_{p^{\infty}})$ l'extension cyclotomique infinie de $E,$ construite en ajoutant tous les racines $p$-ème d'unité dans une clôture algébrique de $E$ choisie; c'est une extension galoisienne infinie. Le caractère cyclotomique identifie le groupe $\Gal(E(\mu_{p^{\infty}})/E)$ à un sous-groupe infini et fermé de $\mathbb{Z}_{p}^{*}.$ Comme la partie de torsion d'un tel groupe est forcement finie, on en déduit un isomorphisme
$$
\Gal(E(\mu_{p^{\infty}})/E)\simeq \mathbb{Z}_{p}\times \Delta,
$$
où $\Delta$ est un groupe fini. Soit $E_{\infty}=(E(\mu_{p^{\infty}}))^{\Delta};$ c'est une extension de Galois infinie de $E$ avec groupe de Galois isomorphe à $\mathbb{Z}_{p}:$ cette extension s'appelle la $\mathbb{Z}_{p}$-extension cyclotomique de $E.$ La complétion $p$-adique $\widehat{E_{\infty}}$ est un corps perfectoïde. On peut alors appliquer au corps $E_{\infty}$ la théorie du corps des normes de Fontaine-Wintenberger: soit $X_{E}(E_{\infty})$ le corps des normes associé; c'est un corps de valuation discrète de caractéristique $p,$ avec corps résiduel isomorphe à celui de $E_{\infty}$ (et, comme $E_{\infty}$ est une colimite d'extension totalement ramifiées de $E,$ c'est simplement $\mathbb{F}_{q}$). En particulier, on en déduit un isomorphisme non-canonique entre $X_{E}(E_{\infty})$ et $\mathbb{F}_{q}(\!(t)\!).$ De plus, la théorie des corps du normes dit que le basculement $(\widehat{E_{\infty}})^{\flat}$ s'identifie à la perfection completée de $X_{E}(E_{\infty}),$ c.f. \cite[Thm. 2.1.6]{EmertonGeeStack}; on en déduit que $(\widehat{E_{\infty}})^{\flat}$ est non-canoniquement isomorphe à $\mathbb{F}_{q}(\!(t^{1/p^{\infty}})\!).$ Cela implique un isomorphisme
$$
(X^{\text{ad}}_{E,F})^{\lozenge}\simeq (\mathbb{B}^{1,\circ,*,\text{perf}}_{F}/\varphi_{F}^{\mathbb{Z}})/\underline{\mathbb{Z}_{p}}.
$$
Ici le choix de la variable $t$ sur la boule unité ouverte épointée $\mathbb{B}^{1,\circ,*}_{F}$ sur $\Spa(F),$ correspond au choix d'un isomorphisme $(\widehat{E_{\infty}})^{\flat}\simeq \mathbb{F}_{q}(\!(t^{1/p^{\infty}})\!).$ Maintenant on peut donner la preuve de la proposition ci-dessus:
\begin{proof}
Soit $\mathcal{F}$ un $\mathbb{F}_{\ell}$-module étale sur $X^{\text{ad}}_{E,F}.$ Comme le site et la cohomologie étale d'un espace adique analytique sur $\Spa(\mathbb{Z}_{p})$ ne changent pas si on passe au diamant associé, \cite[Lemma 15.6]{ScholzeDiamonds}, on a la formule suivante:
$$
R\Gamma(X^{\text{ad}}_{E,F},\,\mathcal{F})\simeq R\Gamma((X^{\text{ad}}_{E,F})^{\lozenge},\,\mathcal{F}^{\lozenge})\simeq R\Gamma(\mathbb{Z}_{p},\,R\Gamma(\mathbb{B}^{1,\circ,*,\text{perf}}_{F}/\varphi_{F}^{\mathbb{Z}},\, \mathcal{G})),
$$
où $\mathcal{G}$ est un $\mathbb{F}_{\ell}$-module étale sur $ \mathbb{B}^{1,\circ,*,\text{perf}}_{F}/\varphi_{F}^{\mathbb{Z}}.$ La raison pour laquelle on a choisi cette présentation du diamant de $X^{\text{ad}}_{E,F},$ est que maintenant on peut utiliser que $\mathbb{Z}_{p}$ est un pro-$p$ groupe et comme on est dans le cas où $\ell \neq p,$ la suite de Hochschild-Serre\footnote{Cette suite spectrale existe parce qu'on considère ici seulement des faisceaux, avec des valeurs discrètes, qui viennent du site étale.} implique
$$
H^{i}(X^{\text{ad}}_{E,F})^{\lozenge},\, \mathcal{F}^{\lozenge})=H^{0}(\mathbb{Z}_{p},\,H^{i}(\mathbb{B}^{1,\circ,*,\text{perf}}_{F}/\varphi_{F}^{\mathbb{Z}},\, \mathcal{G})),
$$
pour $i\geq 0.$ Il reste donc à démontrer que $\text{cd}_{\ell}(\mathbb{B}^{1,\circ,*,\text{perf}}_{F}/\varphi_{F}^{\mathbb{Z}})\leq 2.$ L'espace perfectoïde $\mathbb{B}^{1,\circ,*,\text{perf}}_{F}/\varphi_{F}^{\mathbb{Z}}$ est qcqs et donc en particulier un diamant spatial. Par \cite[Prop. 21.11]{ScholzeDiamonds}, on a
$$
\text{cd}_{\ell}(\mathbb{B}^{1,\circ,*,\text{perf}}_{F}/\varphi_{F}^{\mathbb{Z}})\leq \dim_{\text{Krull}}(|\mathbb{B}^{1,\circ,*,\text{perf}}_{F}/\varphi_{F}^{\mathbb{Z}}|)+\text{sup}_{y}\text{cd}_{\ell}(\kappa(y)),
$$
où le supremum est pris sur tous les points maximaux dans $\mathbb{B}^{1,\circ,*,\text{perf}}_{F}/\varphi_{F}^{\mathbb{Z}}.$ Comme l'action de $\varphi_{F}$ sur $\mathbb{B}^{1,\circ,*,\text{perf}}_{F}$ est proprement discontinue, 
$$
\dim_{\text{Krull}}(|\mathbb{B}^{1,\circ,*,\text{perf}}_{F}/\varphi_{F}^{\mathbb{Z}}|)=\dim_{\text{Krull}}(|\mathbb{B}^{1,\circ,*,\text{perf}}_{F}|)=1.
$$
Il faut donc déterminer $\text{cd}_{\ell}(\kappa(y)),$ où $y\in \mathbb{B}^{1,\circ,*,\text{perf}}_{F}/\varphi_{F}^{\mathbb{Z}}$  est un point maximal. Ici on peut soit utiliser un résultat de Scholze, qui dit que parce que $f\colon \mathbb{B}^{1,\circ,*,\text{perf}}_{F}/\varphi_{F}^{\mathbb{Z}}\rightarrow \Spa(F,\,\mathcal{O}_{F})$ est un morphisme entre diamants spatiaux avec $\text{dim.trg}(f)= 1,$ on a $\text{cd}_{\ell}(\kappa(y))\leq 1$ c.f. \cite[Prop. 21.16]{ScholzeDiamonds}, soit on peut utiliser des résultats de Berkovich: le corps résiduel complété $\kappa(y)$ s'identifie au corps résiduel complété d'un point maximal dans la perfection d'une couronne suffisamment petite $A(\rho, \rho^{\prime})\subseteq \mathbb{B}^{1,\circ,*}_{F}.$ \footnote{Cette notation est différent de celle dans \cite{BerkovichIHES}: il note $\kappa(x)$ le corps résiduel non-complété et $\mathcal{H}(x)$ le corps résiduel complété.} Soit $y^{\prime}\in A(\rho, \rho^{\prime})$ le point maximal qui correspond à $y$ et soient $L(y^{\prime})$ le corps résiduel  non-complété resp. $\kappa(y^{\prime})$ le corps résiduel complété de $A(\rho, \rho^{\prime})$ en $y^{\prime}.$ Il faut donc comprendre $\text{cd}_{\ell}(\widehat{\kappa(y^{\prime})^{\text{perf}}}).$ En utilisant \cite[Prop. 2.4.1.]{BerkovichIHES}(ici on utilise \cite[Prop. 2.4.3.]{BerkovichIHES} pour s'assurer que les hypothèses sont bien vérifiées: la condition que l'anneau des entiers est hensélien est stable sous perfection et complétion), on déduit
$$
\Gal((\widehat{\kappa(y^{\prime})^{\text{perf}}})^{\text{sep}}/\widehat{\kappa(y^{\prime})^{\text{perf}}})\simeq \Gal((\kappa(y^{\prime})^{\text{perf}})^{\text{sep}}/\kappa(y^{\prime})^{\text{perf}})\simeq \Gal(\kappa(y^{\prime})^{\text{sep}}/\kappa(y^{\prime})).
$$
Le résultat \cite[Thm. 2.3.3.]{BerkovichIHES} dit que $L(y^{\prime})$ est quasi-complet (c.f. \cite[Def.2.3.1.]{BerkovichIHES}), donc on peut oublier le procédé de complétion (encore par \cite[Prop. 2.4.1.]{BerkovichIHES}) pour calculer la $\ell$-dimension cohomologique et à ce point là on peut simplement utiliser \cite[Thm. 2.5.1.]{BerkovichIHES}; théorème qui dit
$$
\text{cd}_{\ell}(\Gal(L(y^{\prime}))^{\text{sep}}/L(y^{\prime}))\leq \text{cd}_{\ell}(F)+\dim(|A(\rho,\rho^{\prime})^{\text{Berk}}|)=1.
$$
\end{proof}
\subsection{$\ell$-dimension cohomologique de $E(X^{\alg}_{E,F})$}\label{subsection ell koho dim des Funktionenkoerpers}
Soient $\ell$ toujours un nombre premier, $\ell \neq p$ et $E(X^{\alg}_{E,F})$ le corps de fonctions de la courbe algébrique. 
\begin{proposition}\label{Prop l dimension Funktionenkoerper}
On a 
$$
\text{cd}_{\ell}(E(X^{\alg}_{E,F}))\leq 1.
$$
\end{proposition}
L'idée est la suivante: on reformule d'abord l'énoncé que $\text{cd}_{\ell}(E(X^{\alg}_{E,F}))\leq 1$  en une assertion sur l'annulation des groupes de Brauer des revêtements ramifiés. Puis on peut appliquer GAGA (proposition \ref{GAGA verzweigte Sachen}) pour ces revêtements ramifiés afin de se réduire à démontrer l'annulation de la $\ell$-torsion dans le groupe de Brauer de l'adification. Finalement on peut finir l'argument en utilisant qu'il résulte de la preuve de la proposition \ref{l-torsion adische Kurve} que pour tout point maximal $x\in X^{\text{ad}}_{E,F}$ on a $\text{cd}_{\ell}(\kappa(x))\leq 1.$\footnote{Ceci était peut-être un peu implicite dans la preuve de la proposition \ref{l-torsion adische Kurve}; voici plus de détails: on peut calculer le corps résiduel complété $\kappa(x)$ dans un voisinage affinoïde $Y_{I}$ du point maximal $x\in |X^{\text{ad}}_{E,F}|.$ Après on peut identifier le point maximal $x\in |Y_{I}|$ à un point maximal dans le diamant associé $Y_{I}^{\lozenge};$ noté $x^{\lozenge}.$ On rappelle qu'il y a une définition de la $\ell$-dimension cohomologique du point $x^{\lozenge}\in Y_{I}^{\lozenge},$ c.f. \cite[Def. 21.10]{ScholzeDiamonds}, tel que $\cd_{\ell}(x)=\cd_{\ell}(\kappa(x)).$ Soit maintenant $f\colon Y_{I}^{\lozenge}\rightarrow \Spa(F)$ le morphisme structurel; c'est un morphisme entre diamants spatiaux avec $\text{dim.trg}(f)\leq 1.$ Le résultat \cite[Prop. 21.16]{ScholzeDiamonds} implique par conséquent bien
$$
\cd_{\ell}(x)=\cd_{\ell}(x^{\lozenge})\leq 1.
$$}
\\
On commence par le lemme suivant; la preuve ressemble à celle donnée par Fargues dans \cite[Cor. 2.5.]{FarguesGtorseurs}.
\begin{Lemma_french}
Pour démontrer $\text{cd}(E(X^{\alg}_{E,F}))\leq 1,$ il suffit de démontrer
$$
\Br(X^{\prime})=0,
$$
pour tout revêtement ramifié $X^{\prime}$ de $X^{\alg}_{E,F}.$
\end{Lemma_french}
\begin{Remark_french}\label{Bemerkung zu l-torsion in der Brauergruppe verzweiger Erweiterungen}
On observe que l'énoncé du lemme n'est pas limité à la $\ell$-dimension cohomologique. Plus tard, on explique comment démontrer l'analogue suivant: pour démontrer $\text{cd}_{\ell}(E(X^{\alg}_{E,F}))\leq 1,$ il suffit de démontrer
$$
\Br(X^{\prime})[\ell]=0,
$$
pour tout revêtement ramifié $X^{\prime}$ de $X^{\alg}_{E,F}.$
\end{Remark_french}
\begin{proof}
On observe d'abord que pour démontrer que $\text{cd}(E(X^{\alg}_{E,F}))\leq 1,$ il suffit de démontrer que $\Br(L)=0$ pour toute extension finie $L/E(X^{\alg}_{E,F}).$ En effet, c'est exactement \cite[TAG 03R8]{stacks}. Ensuite, on prend la normalisation de $L$ dans $X^{\alg}_{E,F};$ c'est un revêtements ramifié $X^{\prime}\rightarrow X^{\alg}_{E,F}$ avec corps de fonctions identifié à $L$ (c.f. la preuve du lemme \ref{Reduktion von allgemein konstru zu konstant verzweigte}). La suite exacte de Kummer sur $\Spec(L)_{\text{ét}}$ implique
$$
\Br(L)[\ell]=H^{2}_{\text{ét}}(\Spec(L),\,\mu_{\ell})=\colim_{U}H^{2}(U,\,\mu_{\ell}),
$$ 
où $\ell$ est un nombre premier arbitraire (i.e. peut-être que $\ell=p$) et la colimite porte sur tous les ouverts affines non-vides dans $X^{\prime}.$\footnote{Ici on a utilisé les faits suivants: le revêtement ramifié admet un faisceau inversible ample, parce que le morphisme $X^{\prime}\rightarrow X^{\alg}_{E,F}$ est fini, et il en résulte qu'on peut raffiner un ouvert quelconque dans $X^{\prime}$ par un ouvert affine (c.f. \cite[Tag 01ZY]{stacks}); cela implique la formule  $L=\underset{U \text{affine ouvert}}\colim \mathcal{O}_{X^{\prime}}(U)$ et le résultat \cite[Tag 09YU]{stacks}. Question: est-il vraie que tous les ouverts dans $X^{\prime}$ sont automatiquement affines?} Soit $Z=X^{\prime}-U.$ Le couple $(X^{\prime},Z)$ est pur d'après le théorème de pureté absolue de Gabber \cite[Thm. 2.1.1]{FujiwaraPurity}; ici on utilise que $X^{\prime}$ est un schéma noethérien normal de dimension $1$ (donc régulier) et $Z$ est une collection finie de spectres des corps, donc régulier lui aussi. On en déduit une surjection
$$
H^{2}(X^{\prime},\, \mu_{\ell})\rightarrow H^{2}(U,\, \mu_{\ell}).
$$
D'après Grothendieck, \cite[Cor. 2.2.]{GrothendieckBrauerII}, on sait que le groupe de Brauer de $X^{\prime}$ est isomorphe au groupe de Brauer cohomologique. En utilisant l'hypothèse $\Br(X^{\prime})=0,$ la suite exacte de Kummer donne alors
$$
\text{Pic}(X^{\prime})_{[\ell]}\simeq H^{2}(X^{\prime},\, \mu_{\ell})
$$
(ici $\text{Pic}(X^{\prime})_{[\ell]}$ est le co-noyau de la multiplication par $\ell.$) Pour démontrer $\Br(L)[\ell]=0,$ on doit trouver pour tout $x\in H^{2}(U,\, \mu_{\ell})$ un ouvert affine $V\subseteq U,$ tel que $x\!\mid_{V}=0\in H^{2}(V,\, \mu_{\ell}).$ Puisqu'on sait déjà que l'application
$$
\text{Pic}(U)_{[\ell]}\rightarrow H^{2}(U,\, \mu_{\ell})
$$
est surjective, on peut trouver une pré-image $[\mathcal{L}]\in \Pic(U)$ de la classe $x.$ Comme $0=\text{Pic}(L)=\colim_{U \text{affine }}\text{Pic}(U),$ on peut trouver un ouvert affine $V\subseteq U,$ tel que $[\mathcal{L}]\!\mid_{V}=0.$ En regardant le diagramme commutatif suivant
$$
\xymatrix{
\text{Pic}(U)_{[\ell]} \ar[r] \ar[d] & H^{2}(U,\, \mu_{\ell}) \ar[d] \\
\text{Pic}(V)_{[\ell]} \ar[r] & H^{2}(V,\, \mu_{\ell}),
}
$$
on trouve $x\!\mid_{V}=0 \in H^{2}(V,\, \mu_{\ell})=0.$ Cela implique $\Br(L)[\ell]=0$ pour tout nombre premier $\ell,$ i.e. $\Br(L)=0.$
\end{proof}
\begin{Remark_french}
Cette remarque est la suite de la remarque faite avant c.f. remarque \ref{Bemerkung zu l-torsion in der Brauergruppe verzweiger Erweiterungen}. On affirme que pour démontrer $\text{cd}_{\ell}(E(X^{\alg}_{E,F}))\leq 1,$ il suffit de démontrer $\Br(X^{\prime})[\ell]=0$ pour tout revêtement ramifié de $X^{\text{alg}}_{E,F}.$
\\
Afin d'expliquer ceci, il suffit d'expliquer le fait suivant: pour démontrer $\text{cd}_{\ell}(E(X^{\alg}_{E,F}))\leq 1,$ il suffit de démontrer que $\Br(L)[\ell]=0,$ pour toute extension finie de $E(X^{\alg}_{E,F}).$ En effet, une fois que l'on a expliqué cette réduction, on peut simplement suivre la preuve donnée auparavant.
\\
On suppose donc $\Br(L)[\ell]=0$ pour toute extension finie de $E(X^{\alg}_{E,F}).$ Soit $\mathcal{F}$ un faisceau de $\ell$-torsion sur $\Spec(E(X^{\alg}_{E,F}))_{\text{ét}}.$ L'objectif est de démontrer que $H^{i}(\Spec(E(X^{\alg}_{E,F})), \mathcal{F})=0,$ pour $i\geq 2.$ Pour une extension finie $L$ de $E(X^{\alg}_{E,F}),$ on sait par hypothèse que $H^{2}(\Spec(L),\,\mu_{\ell})=0.$ Soit $$H\subset \Gal(\overline{E(X^{\alg}_{E,F}})/E(X^{\alg}_{E,F}))$$ le sous-groupe pro-$\ell$ maximal et $M\subset \overline{E(X^{\alg}_{E,F})}$ l'extension associée. Comme $M$ est une colimite d'extensions finies, on sait que $H^{2}(\Spec(M),\mu_{\ell})=0.$ 
\\
On observe que $\mu_{\ell}\subset M^{*}.$ De fait, sinon $M$ aurait une extension galoisienne de degré premier à $\ell,$ mais c'est impossible par construction de $M.$ Maintenant on peut appliquer \cite[Tag 0DV9]{stacks} et en déduire que 
$$H^{q}(\Spec(E(X^{\alg}_{E,F})),\,\mathcal{F})=0,$$
pour $q\geq 2.$
\end{Remark_french}
En utilisant l'équivalence GAGA pour des revêtements ramifiés (prop. \ref{GAGA verzweigte Sachen}) pour déduire que
$$
\Br(X^{\prime})\simeq \Br(X^{\prime \text{ad}}),
$$
(c.f. lemme \ref{GAGA fuer die Brauergruppe}),
il suffit alors d'expliquer pourquoi le résultat suivant est vrai:
\begin{Lemma_french}\label{Lemma l-torsion Brauer adisch}
Pour tout $\ell\neq p,$ on a
$$
\Br(X^{\prime \ad})[\ell]=0.
$$
\end{Lemma_french}
\begin{proof}
On considère la projection du site étale vers le site analytique de l'espace adique analytique noethérien $X^{\prime \text{ad}}:$
$$
r\colon (X^{\prime \text{ad}})_{\text{ét}}\rightarrow (X^{\prime \text{ad}})_{\text{an}}.
$$
On a la formule suivante pour la fibre de $R^{i}r_{*}\mathbb{G}_{m}$ en un point $x\in |X^{\prime \text{ad}}|:$ soit $X^{\prime \text{ad}}_{x}$ la localisation de l'espace $X^{\prime \text{ad}}$ en le point $x;$ c'est l'espace pseudo-adique au sens de Huber $(X^{\prime \text{ad}},G(x)),$ où $G(x)$ est l'ensemble des généralisations du point $x.$ Alors on a:
$$
(R^{i}r_{*}\mathbb{G}_{m})_{x}=H^{i}(X^{\prime \text{ad}}_{x},\,\mathbb{G}_{m}).
$$
En effet, cela résulte de l'énoncé de \cite[Corollary 2.4.6]{HuberBuch}, en utilisant $G(x)\sim \underset{x\in U\text{ ouvert qcqs}}\lim U.$ On affirme que cette fibre ne dépend que de la généralisation maximale de $x.$ On observe dans un premier temps que $\mathbb{G}_{m}$ est un faisceau sur-convergent ('overconvergent') sur $(X^{\prime \text{ad}})_{\text{an}};$ par \cite{HuberBuch}, début de la page 400, on doit se convaincre que $\mathbb{G}_{m}$ est constant sur les généralisations des points $x\in X^{\prime \text{ad}};$ ceci est le cas parce que, par définition, les sections de $\mathbb{G}_{m}$ ne dépendent pas de la composante $+$ dans une paire de Huber affinoïde. On peut aussi observer que $\mathbb{G}_{m}$ est sur-convergent dans le sens étale, \cite[Def. 8.2.1]{HuberBuch}: en utilisant \cite[Rem. 8.2.2. (b)]{HuberBuch}, on doit vérifier le fait suivant: pour $Y\rightarrow X^{\prime \text{ad}}$ étale, $\mathbb{G}_{m}\!\mid_{Y}$ est un faisceau sur-convergent sur le site $(Y)_{\text{an}}:$ ceci resulte du même argument que celui qu'on donné avant.
\\
En particulier, on en déduit que pour toute spécialisation $\eta_{1}\rightarrow \eta_{2}$ entre points géométriques dans $(X^{\prime\text{ad}}_{x})_{\text{ét}},$ on a
$$
(\mathbb{G}_{m}\!\mid_{X^{\prime \text{ad}}_{x}})_{\eta_{2}}\simeq (\mathbb{G}_{m}\!\mid_{X^{\prime\text{ad}}_{x}})_{\eta_{1}}.
$$
En utilisant \cite[Cor. 2.6.7]{HuberBuch}, on voit bien que la fibre $(R^{i}r_{*}\mathbb{G}_{m})_{x}$ ne dépend que de la généralisation maximale de $x.$
\\
 Finalement, on peut appliquer \cite[Prop. 2.3.10.]{HuberBuch}, pour avoir la formule assez explicite dans les mains:
\begin{equation}\label{Beschreibung des Stalks an einem maximalen Punkt}
(R^{i}r_{*}\mathbb{G}_{m})_{x}=H^{i}_{\text{ét}}(\Spec(K),\,\mathbb{G}_{m}),
\end{equation}
où $K$ est le corps résiduel complété (!) du point $x.$
Maintenant, on affirme que la suite spectrale de Leray pour $r$ implique un isomorphisme
$$
H^{2}(X^{\prime \text{ad}},\,\mathbb{G}_{m})=H^{0}((X^{\prime})_{\text{an}},\, R^{2}r_{*}\mathbb{G}_{m}).
$$
De fait, on doit expliquer pourquoi $H^{1}((X^{\prime})_{\text{an}}, R^{1}r_{*}\mathbb{G}_{m})=H^{2}((X^{\prime})_{\text{an}},\, R^{0}r_{*}\mathbb{G}_{m})=0.$ La première annulation résulte simplement de 
$$
(R^{1}r_{*}\mathbb{G}_{m})_{x}=\Pic(K)=0,
$$
pour tout point maximal $x\in X^{\prime \text{ad}};$ ici on a utilisé la formule (\ref{Beschreibung des Stalks an einem maximalen Punkt}) et $H^{1}_{\text{ét}}(\Spec(K),\,\mathbb{G}_{m})=\Pic(K).$ La seconde annulation résulte du fait que $|X^{\prime \text{ad}}|$ est un espace topologique qcqs et spectral de dimension de Krull $1$, c.f. \cite[Cor. 4.6.]{ScheidererCohoDim}. Ensuite, on veut comprendre $H^{0}((X^{\prime})_{\text{an}},\, R^{2}r_{*}\mathbb{G}_{m}).$ On a une injection
$$
H^{0}((X^{\prime})_{\text{an}},\, R^{2}r_{*}\mathbb{G}_{m})\hookrightarrow \prod_{x \text{ maximal}}(R^{2}r_{*}\mathbb{G}_{m})_{x},
$$
qui induit une injection analogue sur les parties de $\ell$-torsion:
$$
H^{0}((X^{\prime})_{\text{an}},\, R^{2}r_{*}\mathbb{G}_{m})[\ell]\hookrightarrow \prod_{x \text{ maximal}}(R^{2}r_{*}\mathbb{G}_{m})_{x}[\ell].
$$
Cependant, on peut désormais observer qu'on a $(R^{2}r_{*}\mathbb{G}_{m})_{x}[\ell]=0:$ en effet, soient $x\in X^{\prime \ad}$ un point maximal et $y\in X^{\text{ad}}_{E,F}$ son image; c'est encore un point maximal et l'extension des corps résiduels complétés $\kappa(x)/\kappa(y)$ est finie. Il en résulte $\text{cd}_{\ell}(x)\leq\text{cd}_{\ell}(y)\leq 1$ - ce qu'on a déjà expliqué dans la preuve de la proposition \ref{l-torsion adische Kurve}. En utilisant encore la formule (\ref{Beschreibung des Stalks an einem maximalen Punkt}), on obtient l'énoncé d'annulation suivant
$$
0=H^{2}_{\text{ét}}(\Spec(\kappa(x)),\,\mu_{\ell})=\Br(\kappa(x))[\ell]=(R^{2}r_{*}\mathbb{G}_{m})_{x}[\ell].
$$
Au total, on en déduit par la suite spectrale de Leray l'annulation $H^{2}(X^{\prime \text{ad}},\mathbb{G}_{m})[\ell]=0.$ Comme le groupe de Brauer défini en utilisant les algèbres d'Azumaya s'injecte dans le groupe de Brauer cohomologique, on a une injection $\Br(X^{\prime \text{ad}})[\ell]\hookrightarrow H^{2}(X^{\prime\text{ad}},\,\mathbb{G}_{m})[\ell]=0,$ ce qui termine la preuve.
\end{proof}
\subsection{$\ell$-dimension cohomologique de $X^{\alg}_{E,F}$}\label{subsection ell koho dim der algebraischen Kurve}
Soit $\ell$ toujours un nombre premier, $\ell \neq p.$ 
\begin{proposition}\label{l-kohomologische Dimension algebraische Kurve}
Soit $\mathcal{F}$ un $\mathbb{F}_{\ell}$-module constructible dans $\widetilde{(X^{\text{alg}}_{E,F})_{\text{ét}}}.$ Alors on a
$$
H^{i}(X^{\alg}_{E,F},\,\mathcal{F})=0,
$$
pour $i\geq 3.$
\end{proposition}
\begin{proof}
L'idée est d'utiliser le résultat sur la $\ell$-dimension cohomologique du corps de fonctions de $X^{\text{alg}}_{E,F},$ couplé avec le théorème de pureté absolue de Gabber afin de pouvoir démontrer l'annulation cherchée de la cohomologie étale des revêtements ramifiés.
\\
On commence avec quelques réductions assez classiques ('la méthode de la trace') pour se ramener au cas d'un revêtement ramifié et le cas des coefficients constants.
\\
On peut trouver un ouvert non-vide $j\colon U\subseteq X^{\text{alg}}_{E,F},$ tel que $\mathcal{F}\!\mid_{U}$ est un système local en $\mathbb{F}_{\ell}$-modules; soit $\mathcal{L}$ ce système sur $U.$ On observe d'abord qu'il suffit de démontrer que
$$
H^{i}(X^{\text{alg}}_{E,F},\,j_{!}\mathcal{L})=0,
$$
pour $i\geq 3.$ Cela résulte du fait que le complement de $U$ dans $X^{\text{alg}}_{E,F}$ est une collection de spectres de corps algébriquement clos. D'après  \cite[Tag 0GJ0]{stacks} on peut trouver un revêtement étale fini $\dot{\pi}\colon V\rightarrow U$ de degré premier à $\ell,$ tel que $\dot{\pi}^{*}(\mathcal{L})$ admet une filtration finie avec des graduées de la forme $\underline{\mathbb{F}_{\ell}}_{V}.$ Maintenant on prend la normalisation de $V\rightarrow U$ dans $X^{\text{alg}}_{E,F};$ c'est un revêtement ramifié $\pi\colon X^{\prime}\rightarrow X^{\text{alg}}_{E,F}$ de degré premier à $\ell$ (c.f. la preuve du lemme \ref{Reduktion von allgemein konstru zu konstant verzweigte}). Cette situation est décrit par le diagramme commutatif suivant:
$$
\xymatrix{
V \ar[r]^{j^{\prime}} \ar[d]^{\dot{\pi}} & X^{\prime} \ar[d]^{\pi} & Z^{\prime} \ar[l]^{i^{\prime}} \ar[d] \\
U \ar[r]^{j} & X^{\text{alg}}_{E,F} & Z \ar[l]^{i}.
}
$$
On en déduit une injection
$$
H^{i}(X^{\text{alg}}_{E,F},\,j_{!}\mathcal{L}) \hookrightarrow H^{i}(X^{\prime},\,j^{\prime}_{!}(\dot{\pi}^{*}\mathcal{L}));
$$
ici on applique le fait que le degré du revêtement ramifié est premier à $\ell.$ En utilisant la filtration de $\dot\pi^{*}(\mathcal{L})$ on se réduit au cas où $\dot{\pi}^{*}(\mathcal{L})=\underline{\mathbb{F}_{\ell}}_{V}.$ Puis on peut observer qu'il suffit de démontrer que
$$
H^{i}(X^{\prime},\,\mathbb{F}_{\ell})=0,
$$
pour $i\geq 3.$ C'est encore le fait que le complement de $V$ dans $X^{\prime}$ est une collection finie de spectres de corps algébriquement clos, parce que $X^{\prime}$ est un schéma irréductible de dimension $1$, qui est fini sur $X^{\text{alg}}_{E,F}.$
\\
Finalement, on peut attaquer cette annulation. Soit $E(X^{\prime})$ le corps de fonctions sur $X^{\prime}$ et soit $$j_{\eta^{\prime}}\colon \Spec(E(X^{\prime}))\rightarrow X^{\prime}$$ l'inclusion du point générique. Soit $c\in H^{i}(X^{\prime},\mathbb{F}_{\ell}),$ pour $i\geq 3$ une classe de cohomologie. D'après le résultat sur la $\ell$-dimension cohomologique de $E(X^{\text{alg}}_{E,F}),$ on sait que
$$
j_{\eta^{\prime}}^{*}(c)\in H^{i}(\Spec(E(X^{\prime})),\mathbb{F}_{\ell})=0,
$$
pour $i\geq 2.$ Comme $H^{i}(\Spec(E(X^{\prime})),\,\mathbb{F}_{\ell})=\colim_{U^{\prime}}H^{i}(U^{\prime},\mathbb{F}_{\ell}),$ où la colimite porte sur tous les ouverts affines non-vides $U^{\prime}$ dans $X^{\prime},$ on en déduit l'existence d'un ouvert - qui dépend de la classe $c$ - $$j_{U^{\prime}}\colon U^{\prime}\hookrightarrow X^{\text{alg}}_{E,F},$$ tel que $j_{U^{\prime}}^{*}(c)=0$ dans $H^{i}(U^{\prime},\mathbb{F}_{\ell}).$ On peut appliquer le théorème de pureté absolue de Gabber \cite[Thm. 2.1.1.]{FujiwaraPurity} au complement de $U^{\prime}$ - ici on utilise que $X^{\prime}$ est noethérien et régulier et que ce complement est encore une collection finie de spectres de corps, donc régulier - pour déduire que pour $i\geq 3,$ on a
$$
H^{i}(X^{\prime},\,\mathbb{F}_{\ell})\simeq H^{i}(U^{\prime},\,\mathbb{F}_{\ell}),
$$
où l'isomorphisme est induit par $j_{U^{\prime}}^{*}(.).$ En effet, soit $T$ le complement de $U^{\prime}$ avec immersion fermée $i_{T}\colon T\hookrightarrow X^{\prime}.$ La pureté absolue donne
$$
i_{T}^{!}(\mathbb{F}_{\ell})\simeq \underline{\mathbb{F}_{\ell}}_{T}(-1)[-2]
$$
et il ne reste qu'à appliquer le triangle exact
$$
\xymatrix{
i_{T *}i_{T}^{!}\mathbb{F}_{\ell} \ar[r] & \mathbb{F}_{\ell} \ar[r] & Rj_{U^{\prime} *}\mathbb{F}_{\ell} \ar[r]^{+1} &.
}
$$
Si on applique cet argument à toute classe de cohomologie (ce qui force à changer les ensembles ouverts $U^{\prime}$) on en déduit l'annulation cherchée.
\end{proof}
\begin{Remark_french}
On peut aussi éviter l'utilisation de la pureté absolue en utilisant que $X^{\text{alg}}_{E,F}$ est un schéma noethérien de dimension $1$ avec tous les corps résiduels en tous les points fermés algébriquement clos, pour déduire que pour tout faisceau constructible $\mathcal{F}$ sur $X^{\text{alg}}_{E,F},$ le morphisme $\mathcal{F}\rightarrow j_{\eta *}j_{\eta}^{*}\mathcal{F}$ induit un isomorphisme sur les groupes de cohomologie en degrés $i\geq 2.$ Ici $j_{\eta}\colon \Spec(E(X^{\text{alg}}_{E,F}))\rightarrow X^{\text{alg}}_{E,F}$ est l'inclusion du point générique.\footnote{Voici l'explication: on observe d'abord que le noyau et le co-noyau du morphisme d'adjonction
$$
\psi\colon \mathcal{F}\rightarrow j_{\eta *}j_{\eta}^{*}\mathcal{F}
$$
sont des faisceaux grattes-ciels. Comme tous les corps résiduels des points fermés sont algébriquement clos, $H^{i}(\ker(\psi))=H^{i}(\text{coker}(\psi))=0,$ pour $i>0.$ Après on peut regarder les suites exactes
$$
\xymatrix{
0 \ar[r] & \ker(\psi) \ar[r] & \mathcal{F} \ar[r] & \text{im}(\psi) \ar[r] & 0
}
$$
et
$$
\xymatrix{
0 \ar[r] & \text{im}(\psi) \ar[r] &  j_{\eta *}j_{\eta}^{*}\mathcal{F}  \ar[r] & \text{coker}(\psi) \ar[r] & 0,
}
$$
et en déduire que $H^{i}(X,\mathcal{F})\simeq H^{i}(X,j_{\eta *}j_{\eta}^{*}\mathcal{F}),$ pour $i\geq 2.$ Maintenant on peut utiliser la suite spectrale de Leray pour $F:=j_{\eta}^{*}\mathcal{F}$ et le morphisme $j_{\eta};$ elle dit qu'il faut analyser
$$
H^{i}(\text{Frac}(\mathcal{O}_{X^{\text{alg}},x}^{\text{h}}),F).
$$ 
Néanmoins le corps de fraction du hensélisé $\mathcal{O}_{X^{\text{alg}},x}^{\text{h}}$ est une extension algébrique du corps des fonctions et donc $\text{cd}_{\ell}(\text{Frac}(\mathcal{O}_{X^{\text{alg}},x}^{\text{h}}))\leq \text{cd}_{\ell}(E(X^{\text{alg}}_{E,F})).$ Si on met tout cela ensemble, on en déduit l'énoncé cherché.
}
\end{Remark_french}
\newpage
\subsection{Le cas $\ell=p$}\label{subsection der Fall l=p}
Dans le cas $\ell=p$ je n'arrive à démontrer l'annulation de la cohomologie étale en degrés $\geq 3$ d'un faisceau de la forme $\mathcal{F}^{\text{ad}}$ dans $\widetilde{(X^{\text{ad}}_{E,F})_{\text{ét}}}$, où $\mathcal{F}$ est un $\mathbb{F}_{p}$-module constructible sur le site étale de $X^{\text{alg}}_{E,F},$ que si l'hypothèse suivante est vérifiée:
\begin{enumerate}
\item[(Stein)]Soit $U\subset X^{\text{alg}}_{E,F}$ un ouvert algébrique avec adification $U^{ad}.$ Alors l'espace adique fortement noethérien $U^{\text{ad}}$ est pré-perfectoïde Stein, i.e. soit $E_{\infty}$ la $\mathbb{Z}_{p}$-extension cyclotomique de $E,$ on suppose donc 
$$
U^{\text{ad}}=\bigcup_{n\in \mathbb{N}} U_{n}
$$ 
est une réunion croissante d'ouverts affinoïdes, $\text{res}\colon \mathcal{O}(U_{n+1})\rightarrow \mathcal{O}(U_{n})$ est d'image dense et $$U_{n}\times_{\Spa(E)}\Spa(\widehat{E}_{\infty})$$ est affinoïde perfectoïde.
\end{enumerate}
\begin{proposition}\label{Prop Fall l=p}
Si l'hypothèse (Stein) est vérifiée, 
$$
H^{i}(X^{\text{ad}}_{E,F},\,\mathcal{F}^{\text{ad}})=0,
$$
pour tout $\mathbb{F}_{p}$-module constructible $\mathcal{F}$ sur le site étale de $X^{\alg}_{E,F}$ et $i\geq 3.$
\end{proposition}
Sous l'hypothèse (Stein), la conjecture de comparaison \ref{Vermutung Vergleich} implique ainsi la conjecture d'annulation \ref{Vermutung Verschwinden}.
\\
Maintenant je donne quelques commentaires sur la condition (Stein).
\begin{Remark_french}\label{Bermerkung Diskussion Hypothese Stein}
\begin{enumerate}
\item[(a):] On observe d'abord que l'hypothèse (Stein) est en fait plus forte que ce dont on a vraiment besoin. Soit $\dot{\pi}^{\ad}\colon V^{\ad}\rightarrow U^{\ad}$ l'adification d'un revêtement fini étale $\dot{\pi}\colon V\rightarrow U$ (de degré premier à $p$). Comme avant (c.f. lemme \ref{Reduktion von allgemein konstru zu konstant verzweigte}) on prend la normalisation pour construire un revêtement ramifié $\pi\colon X^{\prime}\rightarrow X^{\alg}_{E,F}$ avec adification $\pi^{\ad}\colon X^{\prime \ad}\rightarrow X^{\ad}_{E,F}.$ Alors on a besoin de savoir que
$$
H^{i}(V^{\text{ad}},\,\mathbb{F}_{p})=0,
$$
pour $i\geq 3.$ Il suffit de savoir que l'on a
$$
H^{i}(V^{\text{ad} \flat}_{\infty},\,\mathcal{O})=0,
$$
pour $i\geq 1,$ où $V^{\text{ad} \flat}_{\infty}:=V^{\text{ad}}\times_{\Spa(E)}\Spa(\widehat{E}_{\infty}).$ C'est une conséquence de l'hypothèse (Stein), mais je peux imaginer qu'on pourrait attaque cet énoncé d'annulation directement.
\item[(b):]
On garde les notation du point (a) avant. 
Je dois expliquer un peu pourquoi je pense que l'hypothèse (Stein) est vraie: sur le plan conceptuel, on peut observer que par analogie entre $X^{\text{ad}}_{E,F}$ et une surface de Riemann compacte il est naturel de penser que chaque ouvert de Zariski dans $X^{\text{ad}}_{E,F}$ est un espace de Stein. Cela prédit qu'on a
\begin{equation}\label{Gleichung Vorhersage Stein}
H^{i}(V^{\text{ad}},\,\mathcal{O})=0,
\end{equation}
pour $i\geq 1.$ 
\\
Afin de donner un peu de substance à cette prédiction je veux expliquer comment on peut démontrer l'égalité (\ref{Gleichung Vorhersage Stein}).
\\
Comme $\dot{\pi}^{\text{ad}}$ est fini étale, il suffit de démontrer qu'on a
$$
H^{i}(U^{\text{ad}},\,\dot{\pi}^{\text{ad}}_{*}\mathcal{O})=0,
$$
pour $i\geq 1.$ Le fibré vectoriel $\mathcal{E}=\pi^{\text{ad}}_{*}(\mathcal{O})$ étend $\dot{\pi}^{\text{ad}}_{*}(\mathcal{O}).$ D'après un résultat d'amplitude dû à Kedlaya-Liu (et de nouveau écrit par Fargues-Scholze \cite[Thm. II.2.6.]{FarguesScholze}) il existe un nombre naturel $n>>0,$ tel que 

$$H^{i}(X^{\text{ad}}_{E,F},\,\mathcal{E}(n))=0$$ 
pour $i>0.$
\\
J'utilise maintenant la suite de Mayer-Vietoris:
\\
Comme $U^{\text{ad}}=X^{\text{ad}}-\lbrace x_{1},...,x_{n} \rbrace,$ où $x_{i}\in |X^{\text{ad}}_{E,F}|^{\text{cl}}$ (ces sont des points classiques), on peut trouver un ouvert affinoïde $Y_{I}\subset Y_{[1,q]},$ où $I\subset (0,1)$ est un intervalle compact, tel que $\lbrace x_{1},...,x_{n} \rbrace \subset Y_{I}$ et le complémentaire $Y_{I}-\lbrace x_{1},...,x_{n} \rbrace$ est de Stein (c.f. la preuve du lemme \ref{Lemma Verschwinden lokal Stein}). La dernière observation dont on a besoin est que le fibré en droites $\mathcal{O}_{X^{\text{ad}}}(n)$ est trivial sur $U^{\text{ad}}.$\footnote{De fait, $U$ est le spectre d'un anneau principal et la restriction de $\mathcal{O}_{X^{\text{ad}}}(n)$ à $U^{\text{ad}}$ s'identifie à le tiré en arrière de la restriction $\mathcal{O}_{X^{\text{alg}}}(n)\!\mid_{U}$ le long du morphisme
$$
U^{\text{ad}}\rightarrow U.
$$}
\\
J'applique la suite exacte de Mayer-Vietoris au recouvrement suivant:
$$
X^{\text{ad}}_{E,F}=Y_{I}\cup U^{\text{ad}}, \text{ }Y_{I}\cap U^{\text{ad}}=Y_{I}-\lbrace x_{1},...,x_{n} \rbrace
$$
et au faisceau $\mathcal{E}(n).$ Comme $H^{i}(X^{\text{ad}}_{E,F},\mathcal{E}(n))=0$ et $H^{i}(Y_{I}-\lbrace x_{1},...,x_{n} \rbrace,\mathcal{E}(n))=0$ pour $i>0,$ on en déduit 
$$
H^{i}(U^{\text{ad}},\,\mathcal{E}(n)\!\mid_{U^{\text{ad}}})=H^{i}(U^{\text{ad}},\,\dot{\pi}^{\text{ad}}_{*}\mathcal{O})=0,
$$
pour $i>0.$
\end{enumerate}
\end{Remark_french}
La preuve de la proposition \ref{Prop Fall l=p} dépend de quelques lemmes et on va donner la preuve après cette préparation. 
\subsubsection{Préparations concernant la cohomologie des groupes:}\label{subsubsection wie man Zp Koho toetet}
Dans la preuve on a besoin du fait suivant: soient $$G=\underset{n\in \mathbb{N}} \lim G_{n}$$ un groupe pro-fini, présenté comme limite inverse d'un système de groupes finis $((G_{n})_{n\in \mathbb{N}},\pi_{m,n}),$ et $((N_{n})_{n\in \mathbb{N}},f_{n,m})$ un système inductif de $G_{n}$-modules. On dit qu'une telle paire 

$$((G_{n},\pi_{m,n})_{n\in \mathbb{N}},(N_{n}, f_{n,m})_{n\in \mathbb{N}})$$

forme une paire compatible si pour $m\geq n,$ on a
$$
f_{n,m}(\pi_{n,m}(g_{m})\cdot x_{n})=g_{m}\cdot f_{n,m}(x_{n}),
$$
pour tout
$g_{m}\in G_{m},$ $x_{n}\in N_{n}$.
Dans une telle situation on peut munir $N= \colim_{n} N_{n}$ d'une structure de $G$-module, où $G$ agit sur $N_{n}$ via le quotient $G\rightarrow G_{n}.$  La $G$-cohomologie du module $N$ est décrite par le lemme suivant:
\begin{Lemma_french}\label{Lemma zum Verschwinden der Gruppenkohomologie}
Si $((G_{n},\pi_{m,n})_{n\in \mathbb{N}},(N_{n}, f_{n,m})_{n\in \mathbb{N}})$ est une paire compatible, on a
$$
H^{j}(G,\,N)=\underset{n\in \mathbb{N}} \colim  H^{j}(G_{n},\,N_{n}).
$$
\end{Lemma_french}
\begin{proof}
Ce lemme est exactement \cite[Proposition 1.5.1.]{CohoofNumberFields}.
\end{proof}
On va appliquer ce lemme dans la situation suivante: soient $(M_{n})_{n\in \mathbb{N}}$ des groupes abéliens discrets, munis de l'action triviale de $G_{n},$ et on suppose qu'ils forment un système inductif, i.e. pour $n \leq m,$ on a une application $M_{n}\rightarrow M_{m}.$ Soit $N_{n}=\text{Ind}_{\lbrace e \rbrace}^{G_{n}}(M_{n})=\text{Hom}(G_{n},M_{n});$ on a alors une application $N_{n}\rightarrow N_{m}$ donnée par 
\begin{align*}
\text{Hom}(G_{n},M_{n}) & \rightarrow \text{Hom}(G_{n+1},M_{n}) \\
 & \rightarrow \text{Hom}(G_{n+1},M_{n+1})
\end{align*} et $((G_{n})_{n\in \mathbb{N}},(N_{n})_{n\in \mathbb{N}})$ forment ainsi une paire compatible. On en déduit par le lemme de Shapiro (\cite[Proposition 1.3.7.]{CohoofNumberFields}) le fait suivant:
\begin{equation}\label{Verschwinden der Gruppenkohomologie}
H^{j}(G,\, \underset{n\in \mathbb{N}} \colim N_{n})=\underset{n\in \mathbb{N}} \colim H^{j}(G_{n},\,\text{Ind}_{\lbrace e \rbrace}^{G_{n}}(M_{n}))=0,
\end{equation}
égalité cruciale pour la preuve de la proposition \ref{Prop Fall l=p}.
\subsubsection{Préparations géométriques}\label{subsubsection wie man die koho des offenen toetet}
De plus, on doit expliquer deux énoncés d'annulation qui sont essentiels dans la preuve de la propostion \ref{Prop Fall l=p}. Le premier lemme est standard:
\begin{Lemma_french}\label{Lemma verschwinden pre-perfectoid Stein}
Soit $S$ un espace pré-perfectoïde Stein sur $\Spa(E).$ Alors on a
$$
H^{i}(S,\,\mathbb{F}_{p})=0,
$$
pour $i\geq 3.$
\end{Lemma_french}
\begin{proof}
Soit $E_{\infty}/E$ la $\mathbb{Z}_{p}$-extension cyclotomique. Par hypothèse on peut trouver un recouvrement
$$
S=\bigcup_{n\in \mathbb{N}}S_{n},
$$
tel que
$$
S_{\infty}:=S\times_{\Spa(E)}\Spa(\widehat{E}_{\infty})=\bigcup_{n\in \mathbb{N}}S_{n}\times_{\Spa(E)}\Spa(\widehat{E}_{\infty}),
$$
où $S_{n,\infty}:=S_{n}\times_{\Spa(E)}\Spa(\widehat{E}_{\infty})$ est un espace affinoïde perfectoïde et les morphismes de restriction $\text{res}\colon \mathcal{O}(S_{n+1,\infty})\rightarrow \mathcal{O}(S_{n,\infty})$ ont image dense. En utilisant la suite spectrale de Hochschild-Serre et $\text{cd}_{p}(\mathbb{Z}_{p})=1,$ il suffit de démontrer
$$
H^{i}(S_{\infty},\,\mathbb{F}_{p})=0,
$$
pour $i\geq 2.$ L'espace perfectoïde $S_{\infty}^{\flat}$ est un espace de Stein: comme on suppose que tous les $S_{n,\infty}$ sont affinoïdes perfectoïdes, on a 
$$
(S_{\infty})^{\flat}=\bigcup_{n\in \mathbb{N}} (S_{n,\infty})^{\flat},
$$
où chaque $S_{n,\infty}^{\flat}$ est encore affinoïde. Si $A\rightarrow B$ est une application continue entre anneaux de Tate qui sont perfectoïdes et telle que l'image est dense c'est aussi le cas pour l'application $A^{\flat}\rightarrow B^{\flat},$ c.f. \cite[Lemma 2.8.4]{KedlayaArizona}.
\\
D'après le résultat de Kedlaya-Liu \cite[Thm. 2.6.5]{KedlayaLiuImperfect}, on a
$$
H^{i}(S_{\infty}^{\flat},\,\mathcal{O})=0,
$$
pour $i\geq 1.$ 
\\
La suite exacte d'Artin-Schreier implique donc bien que 
$$
H^{i}(S_{\infty},\,\mathbb{F}_{p})=H^{i}(S_{\infty}^{\flat},\,\mathbb{F}_{p})=0,
$$
pour $i\geq 2.$
\end{proof}
\begin{Corollary_french}\label{Korollar Hyp Stein impliziert Verschwinden offenes}
Soient $U\hookrightarrow X^{\alg}_{E,F}$ un sous-schéma ouvert non-vide et $U^{\ad}\hookrightarrow X^{\ad}_{E,F}$ l'adification associé. De plus, on considère un revêtement finit étale $V\rightarrow U$ et $V^{\ad}\rightarrow U^{\ad}$ l'adification associé. Si l'hypothèse (Stein) est vérifiée, on a
$$
H^{i}(V^{\ad},\mathbb{F}_{p})=0,
$$
pour $i\geq 3.$
\end{Corollary_french}
\begin{proof}
On observe d'abord qu'un revêtement fini d'un espace de Stein est encore un espace de Stein. Le théorème de presque pureté de Scholze \cite[Thm. 1.10]{ScholzePerfectoid} implique alors que si $U^{\text{ad}}$ est pré-perfectoïde Stein aussi $V^{\text{ad}}$ l'est. On conclut en utilisant le lemme \ref{Lemma verschwinden pre-perfectoid Stein}.
\end{proof}
Ce corollaire est le seul endroit où j'ai besoin de l'hypothèse (Stein).
\\
On va aussi utiliser un autre ingrédient géométrique, qui est inconditionnel, parce que c'est seulement un énonce local. Celui est le suivant:
\begin{Lemma_french}\label{Lemma Verschwinden lokal Stein}
Soient $X^{\prime}\rightarrow X^{\alg}_{E,F}$ un revêtement fini plat, qui est ramifie le long de $Z\hookrightarrow X^{\alg}_{E,F}$ et $X^{\prime \ad}\rightarrow X^{\ad}_{E,F}$ l'adification associé. Soit $Z^{\prime \ad}$ le pré-image de $Z^{\ad}
$ dans $X^{\prime \ad}.$ On considère de plus la $\mathbb{Z}_{p}$-extension cyclotomique $E_{\infty}$de $E$ et le diamant $X^{\prime \ad \lozenge}_{\infty}:= X^{\prime \ad \lozenge}\times_{\Spd(E)}\Spd(\widehat{E}_{\infty})$ et $Z^{\prime \ad}_{\infty}$ le pré-image de $Z^{\prime \ad}$ dans $X^{\prime \ad \lozenge}_{\infty}$. Soit $\widetilde{\mathcal{W}}^{\prime}$ un système fondamental des voisinages qcqs ouverts $\underline{\mathbb{Z}_{p}}$-invariants de $Z^{\prime \ad}_{\infty}\hookrightarrow X^{\prime \text{ad}}_{\infty}.$  Alors on a
$$
\underset{\widetilde{W}^{\prime}\in \widetilde{\mathcal{W}^{\prime}}}\colim H^{i}(\widetilde{W}^{\prime}-(Z^{\prime \ad})_{\infty},\,\mathbb{F}_{p})=0,
$$
pour $i\geq 2.$
\end{Lemma_french}
\begin{Remark_french}
On peut observer que le diamant $X^{\prime \ad \lozenge}_{\infty}$ est réprésentable par un espace perfectoïde en utilisant l'existence d'une 'perfectoïdisafiction' \cite{bhatt2022prisms}. 1.16.(1)]: c'est un problème local et on prend $\Spa(R,R^{+})$ affinoïde dans $X^{\text{ad}}_{E,F}$ avec pré-image $\Spa(R^{\prime},R^{\prime +})$ dans $X^{\prime \text{ad}},$ tel qu'on a $\Spa(R,R^{+})\times_{\Spa(E)}\Spa(\widehat{E}_{\infty})=\Spa(S,S^{+}),$ où $(S,S^{+})$ est affinoïde perfectoïde. On peut par exemple travailler avec un recouvrement de $X^{\ad}_{E,F}=Y_{[1,q]}/Y_{[1,1]}\simeq Y_{[q,q]}$ donné par $\Spa(B_{[1,\rho]},B_{[1,\rho]}^{+})$ et $\Spa(B_{[\rho,q]},B_{[\rho,q]}^{+}),$ $\rho\in [1,q].$ Comme le foncteur qui associe le diamant à un espace adique analytique commute aux limites, on a $\Spd(R^{\prime},R^{\prime +})\times_{\Spd(E)}\Spd(\widehat{E}_{\infty})=\Spd(S^{\prime},S^{\prime +}),$ où $S^{\prime}=S\otimes_{R}R^{\prime}$ et $S^{\prime +}$ est la clôture integrale de $S^{+}$ dans $S^{\prime}.$ Puisque $S^{+}\rightarrow S^{\prime +}$ est un morphisme intègre, on peut appliquer le résultat de Bhatt-Scholze ci-dessus pour trouver un anneau perfectoïde entier $S^{\prime +}_{\text{perfd}},$ universel pour l'existence d'un morphisme $S^{\prime +}\rightarrow S^{\prime +}_{\text{perfd}}.$ Si $x$ est une pseudo-uniformisante dans $S,$ la paire affinoïde perfectoïde $(S^{\prime +}_{\text{perfd}}[1/x],(S^{\prime +}_{\text{perfd}}[1/x])^{+}),$ où $(S^{\prime +}_{\text{perfd}}[1/x])^{+}$ est la clôture integrale de $S^{\prime +}_{\text{perfd}}$ dans $S^{\prime +}_{\text{perfd}}[1/x],$  est l'espace perfectoïde cherché. Après, on peut recoller ces espaces perfectoïdes affinoïdes afin de construire l'espace perfectoïde qui représente $(X^{\prime \text{ad}})^{\lozenge}_{\infty}.$
\\
L'espace $Z^{\prime \ad}_{\infty}$ est un sous-espace Zariski-fermé dans l'espace perfectoïde qui représente $X^{\prime \ad \lozenge}_{\infty}.$
\end{Remark_french}
\begin{proof}
La preuve du lemme \ref{Lemma Verschwinden lokal Stein} consiste en l'explication de la suite d'identifications suivantes:
\begin{align*}
\underset{\tilde{W}^{\prime}} \colim H^{i}(\widetilde{W}^{\prime}-(Z^{\prime \text{ad}})_{\infty},\,\mathbb{F}_{p}) & = \underset{Z^{\prime \text{ad}}\subset W^{\prime} \subset X^{\prime \text{ad}}}\colim H^{i}((W^{\prime}-Z^{\prime \text{ad}})_{\infty},\,\mathbb{F}_{p}) \\
& = \underset{Z^{\text{ad}}\subset W \subset X^{\text{ad}}}\colim H^{i}((\pi^{-1}(W)-Z^{\prime \text{ad}})_{\infty},\,\mathbb{F}_{p}) \\
& = \underset{Z^{\text{ad}}\subset W \subset X^{\text{ad}}, W\subset Y_{[\rho^{q},\rho]} \text{affd}, \varphi(W)\cap W = \emptyset}\colim H^{i}((\pi^{-1}(W)-Z^{\prime \text{ad}})_{\infty},\,\mathbb{F}_{p}) \\
& = 0,
\end{align*}
où la dernière égalité est vraie pour $i\geq 2.$ Ici on note par $(.)_{\infty}$ le changement de base $.\times_{\Spd(E)}\Spd(\widehat{E}_{\infty}).$
\\
En effet, dans la première égalité on utilise le fait que chaque voisinage ouvert qcqs $\underline{\mathbb{Z}_{p}}$-invariant $\widetilde{W}\subset (X^{\prime \text{ad}})^{\lozenge}_{\infty}$ de $(Z^{\prime \text{ad}})_{\infty}$ descend en un voisinages ouvert qcqs $W^{\prime}\subset X^{\prime \text{ad}}$ de $Z^{\prime \text{ad}}.$ Dans la deuxième égalité, on utilise la possibilité de trouver un système cofinal de tels voisinages de la forme $\pi^{-1}(W),$ où $W$ est un voisinage qcqs ouvert de $Z^{\text{ad}}.$ De fait, comme les points différents $z^{\prime}\in Z^{\prime \ad}$ n'ont pas de généralisations communes, on peut dans un premier temps remplacer le voisinage $W^{\prime}$ par un voisinage de la forme $\coprod_{i=1}^{m}W_{i}^{\prime},$ où $W_{i}^{\prime}$ est un voisinage ouvert qcqs de $z^{\prime}_{i}\in Z^{\prime \text{ad}}.$ Dans un deuxième temps, comme $\pi^{\text{ad}}\colon X^{\prime \text{ad}}\rightarrow X^{\text{ad}}_{E,F}$ est fini et ainsi une application fermée, il existe un voisinage qcqs ouvert $W$ de $Z^{\text{ad}},$ tel que $$\pi^{-1}(W)\subseteq \coprod_{i=1}^{m}W_{i}^{\prime}.$$ Dans la troisième égalité on a utilisé l'uniformisation $X^{\text{ad}}_{E,F}=Y_{[1,q]}/\varphi^{\mathbb{Z}}$ (i.e. on utilise le Frobenius pour identifier $Y_{[1,1]}$ et $Y_{[q,q]}$); on trouve ainsi un système cofinal de la forme $W\subset Y_{[1,q]}$ affinoïde, tel que $\varphi(W)\cap W=\emptyset.$
\\
Finalement, on va utiliser qu'on peut trouver un système cofinal des voisinages affinoïdes $W\subset Y_{[1,q]}$ de $Z^{\ad},$ tel que $W-Z^{\ad}$ est une union disjointe finie de espaces pré-perfectoïdes Stein: on écrit $Z^{\text{ad}}=\lbrace z_{1},..,z_{m} \rbrace,$ comme les localisations rationnels sont une base de la topologie de l'espace affinoïde $Y_{[1,q]},$ on peut raffiner chaque ouvert affinoïde dans $Y_{[1,q]}$, voisinage de $Z^{\ad},$ par un ouvert affinoïde de la forme 
$$
\coprod_{i=1}^{m}U_{i},
$$
où $U_{i}$ sont des localisation rationnels de $Y_{[1,q]},$ voisinage du point classique $z_{i}\in Z^{\ad}.$
Chaque point classique $z_{i}=V(f_{i})$ est un fermé de Zariski dans $U_{i}$ et on peut par conséquent considérer le recouvrement
$$
U_{i}-\lbrace z_{i} \rbrace=\bigcup_{n\in \mathbb{N}}W_{n}(z_{i}),
$$
où $W_{n}(z_{i})=\lbrace y\in U_{i}\colon |f_{i}(y)|\geq |\pi(y)|^{n}\neq 0 \rbrace;$ puisque l'ouvert $U_{i}$ est bien affinoïde, les $W_{n}(z_{i})$ restent affinoïdes et c'est donc bien un recouvrement qui fait de $U_{i}-\lbrace z_{i} \rbrace$ un espace de Stein.\footnote{Comme $U_{i}-\lbrace z_{i} \rbrace$ est l'ensemble Zariski-fermé défini par $Tf_{i}-1$ dans $U_{i}\times_{\Spa(E)}\mathbb{A}^{1 \ad}_{\Spa(E)}$.} Les $W_{n}(z_{i})\times_{\Spa(E)}\Spa(\widehat{E}_{\infty})$ sont des espaces affinoïdes perfectoïdes: dans un premier temps, on observe que $U_{i}\times_{\Spa(E)}\Spa(\widehat{E}_{\infty})$ sont des espaces affinoïdes perfectoïdes; en effet, c'est une localisation rational de l'espace affinoïde perfectoïde $Y_{[1,q]}\times_{\Spa(E)}\Spa(\widehat{E}_{\infty})$ (\cite[Thm. 6.3.(ii)]{ScholzePerfectoid}). Soit $f_{\infty}$ l'image de $f$ par l'application 
$$
\mathcal{O}(Y_{[1,q]})\rightarrow \mathcal{O}(Y_{[1,q],\infty}).
$$
On a donc
$$
W_{n}(z_{i})\times_{\Spa(E)}\Spa(\widehat{E}_{\infty})= \lbrace y\in U_{i,\infty}\colon |f_{\infty}(y)|\geq |\pi(y)^{n}|\neq 0 \rbrace,
$$
et c'est un ouvert rational dans l'espace affinoïde perfectoïde $Y_{[1,q],\infty}$ (c.f. la preuve de la proposition \cite[Prop. II.1.1]{FarguesScholze});par conséquent affinoïde perfectoïde lui-même.
Puisque $\pi$ est fini, $\pi^{-1}(\coprod_{i=1}^{m}[U_{i}-\lbrace z_{i}\rbrace])=\pi^{-1}(\coprod_{i=1}^{m}U_{i})-Z^{\prime \text{ad}}$ reste une union disjointe finie des espaces de Stein et comme on travaille ici sur le lieu où $\pi$ est fini étale, on en déduit que $(\pi^{-1}(\coprod_{i=1}^{m}U_{i})-Z^{\prime \text{ad}})_{\infty}$ est une union disjointe finie des espace de Stein perfectoïde. En passant au basculement, comme dans la preuve du lemme \ref{Lemma verschwinden pre-perfectoid Stein} la suite exacte de Artin-Schreier implique alors de nouveau qu'on a 
$$
H^{i}((\pi^{-1}(\coprod_{i=1}^{m}U_{i})-Z^{\prime \text{ad}})_{\infty},\,\mathbb{F}_{p})=0,
$$
pour $i\geq 2.$
\end{proof}
Le dernier point essentiel dans la preuve de la proposition \ref{Prop Fall l=p} est la construction d'un bon système fondamental des voisinages ouverts qcqs $\underline{\mathbb{Z}_{p}}$-invariants de $Z^{\prime \ad}_{\infty}$ dans $X^{\prime \ad \lozenge}_{\infty}.$
\begin{Lemma_french}\label{Lemma Konstruktion des guten Systems der Umgebungen}
Soient $X^{\prime}\rightarrow X^{\alg}_{E,F}$ un revêtement fini plat, ramifié le long de $Z\hookrightarrow X^{\alg}_{E,F},$ $U$ le complement ouvert de $Z$ dans $X^{\alg}_{E,F},$ $Z^{\prime}$ le pré-image de $Z$ dans $X^{\prime}$ et $V$ le complement ouvert de $Z^{\prime}$ dans $X^{\prime}.$ Si on passe au monde adique, on obtient le Zariski-fermé $Z^{\ad}\hookrightarrow X^{\ad}_{E,F},$ avec complement ouvert $U^{\ad}$ dans $X^{\ad}_{E,F}$ et le Zariski-fermé $Z^{\prime \ad}\hookrightarrow X^{\prime \ad}$ avec complement ouvert $V^{\ad}.$
\\
On écrit la $\mathbb{Z}_{p}$-extension cyclotomique $E_{\infty}$ de $E$ comme union croissante des sous-extension galoisiennes $E_{n}/E,$ tel que $G_{n}:=\Gal(E_{n}/E)\simeq \mathbb{Z}/p^{n}\mathbb{Z}.$ Maintenant, on dessine le diagramme suivant (pour aider à la construction):
$$
\xymatrix{
(Z^{\prime \ad})_{\infty} \ar[r] \ar[d] & (X^{\prime \ad \lozenge})_{\infty} \ar[d]^{f_{n}} & (V^{\ad \lozenge})_{\infty} \ar[l] \ar[d] \\
(Z^{\prime \ad})_{n} \ar[r] \ar[d] &  (X^{\prime \ad \lozenge})_{n} \ar[d]^{g_{n}} & (V^{\ad \lozenge})_{n} \ar[l] \ar[d] \\
(Z^{\prime \ad}) \ar[r] & X^{\prime \ad \lozenge} & V^{\ad \lozenge} \ar[l],
}
$$
où $(.)_{n}$ (resp. $(.)_{\infty}$) signifie le changement de base de $\Spd(E)$ vers $\Spd(E_{n})$ (resp. $\Spd(\widehat{E}_{\infty})$), le morphisme $f_{n}$ est un $\ker(\mathbb{Z}_{p}\rightarrow G_{n})=H_{n}$-torseur et $g_{n}$ est un $G_{n}$-torseur.
\\
On appelle un ouvert $W^{\prime (n)}\subset X^{\prime \ad}_{n}$ 'suffisamment petit', si
$$
gW^{\prime (n)}\cap hW^{\prime (n)}=\emptyset,
$$
pour $g\neq h\in G_{n}.$
\begin{enumerate}
\item[(a):] Soit
$$
\mathcal{W}^{\prime (n)}_{i}=\lbrace \coprod_{g\in G_{n}} g\cdot W^{\prime (n)}\colon W^{\prime (n)} \text{voisinage qcqs ouvert 'suffisament petit' de } \tilde{z}_{i}^{\prime (n)} \rbrace,
$$
pour $i=1,...,m.$ Alors le système $\coprod_{i=1}^{m} W^{\prime (n)}_{i},$ où $W^{\prime (n)}_{i}\in \mathcal{W}^{\prime (n)}_{i}$ est un système fondamental des voisinages ouverts qcqs $G_{n}$-invariants de $Z^{\prime \ad}_{n}$ dans $X^{\prime \ad}_{n}.$
\item[(b):] Le système $(f^{-1}_{n}(\coprod_{i=1}^{m} W^{\prime (n)}_{i}))_{n},$ où $W^{\prime (n)}_{i}\in \mathcal{W}^{(n)}_{i},$ $n\in \mathbb{N},$ est un système fondamental des voisinages ouverts qcqs $\underline{\mathbb{Z}_{p}}$-invariants de $Z^{\prime \ad}_{\infty}$ dans $X^{\prime \ad \lozenge}_{\infty}.$
\end{enumerate}
\end{Lemma_french}
\begin{proof}
On peut se réduire au cas où $Z^{\ad}=\lbrace \infty \rbrace$ est un seul point à l'infini (c'est juste pour relaxer un peu les notations). Soit $Z^{\prime \text{ad}}=\lbrace z_{1}^{\prime},...,z_{m}^{\prime} \rbrace.$ On observe d'abord que tous les points $z_{i}^{\prime}\in Z^{\prime \text{ad}}$ sont fermés et maximaux: ils sont fermés parce que $Z^{\prime \text{ad}}$ est fermé, donc spectral, et parce que c'est un ensemble fini - c.f. \cite[Tag 0902 et Tag 0905]{stacks}; ils sont maximaux parce qu'il n'existe pas de généralisations communes pour deux points $z_{i}^{\prime}\neq z_{j}^{\prime},$ $z_{i}^{\prime},z_{j}^{\prime}\in Z^{\prime \text{ad}},$ car les valuations sur le corps résiduels $\kappa(z_{i}),$ $z_{i}\in Z^{\text{ad}}$ s'étendent d'une façon unique. Comme $$g_{n}\colon (X^{\prime \text{ad} \lozenge})_{n} \rightarrow X^{\prime \text{ad} \lozenge} $$ est un $G_{n}$-torseur, on peut choisir des pré-images compatibles (en $n\in \mathbb{N}$) $\tilde{z}_{i}^{\prime (n)}\in (X^{\prime \text{ad} \lozenge})_{n},$ $i=1,...,m,$ des points $z_{i}^{\prime}\in X^{\prime \text{ad}},$ et écrire $(Z^{\prime \text{ad}})_{n}$ comme réunion disjointe des orbites $G_{n}\cdot \tilde{z}_{i}^{\prime (n)}.$ Puisque $g_{n}$ est fini étale tous les points dans $(Z^{\prime \text{ad}})_{n}$ sont encore (fermés et) maximaux. 
\\
On démontre le point (a).
L'énoncé clef ici est le suivant: on peut toujours raffiner un voisinage de $\tilde{z}_{i}^{\prime (n)}$ par un voisinage qui est 'suffisamment petit'. Pour expliquer ceci on considère la localisation de l'espace adique analytique $X^{\prime \text{ad}}_{n}$ en un point $\tilde{z}_{i}^{\prime (n)}:$ c'est l'ensemble des généralisations du point $\tilde{z}_{i}^{\prime (n)}.$ Comme les points $\tilde{z}_{i}^{\prime (n)}$ sont maximaux, on a
$$
\lbrace \tilde{z}_{i}^{\prime (n)} \rbrace = (X^{\prime \text{ad}}_{n})_{\tilde{z}_{i}^{\prime (n)}}=\bigcap_{\tilde{z}_{i}^{\prime (n)}\in U}U,
$$
où l'intersection porte sur tous les voisinages ouverts qcqs de $\tilde{z}_{i}^{\prime (n)}.$ Il suffit ainsi de démontrer 
$$
g((X^{\prime \text{ad}}_{n})_{\tilde{z}_{i}^{\prime (n)}})\cap (X^{\prime \text{ad}}_{n})_{\tilde{z}_{i}^{\prime (n)}}=\emptyset,
$$
pour $g\neq e\in G_{n};$ comme $G_{n}$ agit sans points fixes sur l'espace topologique $|X^{\prime \text{ad}}_{n}|,$ on a $g(\tilde{z}_{i}^{\prime (n)})\neq \tilde{z}_{i}^{\prime (n)},$ pour $g\neq e\in G_{n}.$ Cela implique l'égalité cherchée.
\\
Soit maintenant $\mathcal{U}_{n}(i)$ un voisinage ouvert qcqs $G_{n}$-invariant de $G_{n}\cdot \tilde{z}_{i}^{\prime (n)};$ d'après l'énoncé qu'on vient d'expliquer, on peut trouver un voisinage 'suffisamment petit' $W^{\prime (n)}_{i}$ de $\tilde{z}_{i}^{\prime (n)}$ inclus dans $\mathcal{U}_{n}.(i)$ Grâce à l'invariance par $G_{n}$ du voisinage $\mathcal{U}_{n}(i),$ on a $\coprod_{g\in G_{n}}g\cdot W^{\prime (n)}_{i}\subseteq \mathcal{U}_{n}(i).$ 
\\
Soit maintenant $\mathcal{U}_{n}$ un voisinage ouvert qcqs $G_{n}$-invariant de $(Z^{\prime \ad})_{n}.$ C'est en particulier un voisinage ouvert qcqs $G_{n}$-invariant de $G_{n}\cdot \tilde{z}_{i}^{\prime (n)};$ on peut par conséquent trouver des voisinages suffisamment petits $W^{\prime (n)}_{i}\subset \mathcal{U}_{n}$ de $\tilde{z}_{i}^{\prime (n)},$ $i=1,...m.$ On observe qu'on peut supposer que
$$
gW^{\prime (n)}_{i}\cap hW^{\prime (n)}_{j}=\emptyset,
$$
pour $g,h\in G_{n},$ $g\neq h,$ $i\neq j\in \lbrace 1,...,m \rbrace.$ \footnote{C'est le même argument qu'on vient de donner pour trouver un voisinage suffisamment petit d'un point $\tilde{z}_{i}^{\prime (n)}.$ De fait, pour $i\neq j,$ on a $g\cdot \tilde{z}_{i}^{\prime (n)}\neq \tilde{z}_{j}^{\prime (n)},$ parce que sinon $f_{n}(g\cdot \tilde{z}_{i}^{\prime (n)})=z_{i}=z_{j}=f_{n}(\tilde{z}_{j}^{\prime (n)}).$ Alors on a
$$
g((X^{\prime \text{ad}}_{n})_{\tilde{z}_{i}^{\prime (n)}})\cap (X^{\prime \text{ad}}_{n})_{\tilde{z}_{j}^{\prime (n)}}=\emptyset.
$$} On peut en déduire que
$$
\coprod_{i=1}^{m} \coprod_{g\in G_{n}}g\cdot W^{\prime (n)}_{i}\subseteq \mathcal{U}_{n}.
$$
Ceci implique le point (a).
\\
On affirme maintenant que le système
$$
\lbrace f_{n}^{-1}(\mathcal{W}^{\prime (n)}_{i}) \rbrace_{i=1,...,m,n\in \mathbb{N}}
$$
forme un système fondamental de voisinages $\underline{\mathbb{Z}_{p}}$-invariants de $(Z^{\prime \text{ad}})_{\infty}$ dans $(X^{\prime \text{ad} \lozenge})_{\infty}.$ Il faut expliquer d'abord pourquoi c'est bien un voisinage $\underline{\mathbb{Z}_{p}}$-invariant: soit $\mathcal{R}_{n}$ un système de représentants de $G_{n}$ dans $\mathbb{Z}_{p}.$ On observe donc qu'on a pour tout $n\in \mathbb{N}$ et chaque voisinage ouvert qcqs $W^{\prime}_{n}$ de $\tilde{z}_{j}^{\prime (n)}$
$$
f_{n}^{-1}(\coprod_{g\in G_{n}}g\cdot W^{\prime}_{n})=\coprod_{\widetilde{g}_{k,n}\in \mathcal{R}_{n}\subset \mathbb{Z}_{p}} \widetilde{g}_{k,n}\cdot f_{n}^{-1}(W^{\prime}_{n}).
$$
En effet, ceci suit du fait que $f_{n}\colon (X^{\prime \text{ad} \lozenge})_{\infty}\rightarrow (X^{\prime \text{ad}\lozenge})_{n}$ est un $H_{n}$-torseur.\footnote{En effet, soit $V^{\prime}_{n}$ un ouvert qcqs dans $X^{\prime \text{ad}}_{n},$ $g\in G_{n}$ et $\tilde{g}\in \mathbb{Z}_{p}$ une pre-image de $g_{n}.$ Alors, on affirme qu'on a
$$
f_{n}^{-1}(g\cdot V^{\prime}_{n})=\tilde{g}\cdot f_{n}^{-1}(V^{\prime}_{n}).
$$
Soit $\tilde{w}\in f_{n}^{-1}(V^{\prime}_{n}),$ alors on a $f_{n}(\tilde{g}\cdot \tilde{w})=g\cdot f_{n}(\tilde{w})\in g\cdot V^{\prime}_{n}.$ Pour l'autre inclusion, soit $\tilde{w}\in f_{n}^{-1}(g\cdot V^{\prime}_{n}).$ On prend $w^{\prime}\in V^{\prime}_{n},$ tel que $f_{n}(\tilde{w})=g\cdot w^{\prime}.$ Soit $\tilde{w}^{\prime}\in f^{-1}(V^{\prime}_{n})$ une pre-image de $w^{\prime}$ sous $f_{n},$ alors on a $f_{n}(\tilde{w})=f_{n}(\tilde{g}\cdot \tilde{w}^{\prime}).$ On trouve un $h\in H_{n},$ tel que $\tilde{w}=\tilde{g}\cdot h\cdot \tilde{w}^{\prime}.$ Comme $f_{n}(h\cdot \tilde{w}^{\prime})=f_{n}(\tilde{w}^{\prime})\in V^{\prime}_{n},$ cela implique que $\tilde{w}\in \tilde{g}\cdot f_{n}^{-1}(V^{\prime}_{n}).$
} 
Il en résulte qu'il s'agit bien de voisinages $\underline{\mathbb{Z}_{p}}$-invariants. 
\\
Il reste à démontrer que c'est bien un système fondamental: soit $V^{\prime}_{\infty}$ un voisinage ouvert qcqs $\underline{\mathbb{Z}_{p}}$-invariant de $(Z^{\prime \text{ad}})_{\infty}.$ Le voisinage $V^{\prime}_{\infty}$ descend vers un voisinage $G_{n}$-invariant qcqs ouvert $V^{\prime}_{n}$ de $(Z^{\prime \text{ad}})_{n},$ pour $n\in \mathbb{N}$ suffisamment grand. On peut ainsi trouver des voisinages 'suffisamment petits' $W_{i}^{\prime (n)}$ de $ \tilde{z}_{i}^{\prime (n)},$ tels que
$$
\coprod_{i=1}^{m}\coprod_{g\in G_{n}}g\cdot W_{i}^{\prime (n)}\subseteq V^{\prime}_{n}.
$$
Il en résulte que
$$
f_{n}^{-1}(\coprod_{i=1}^{m}\coprod_{g\in G_{n}}g\cdot W_{i}^{\prime (n)})\subseteq f_{n}^{-1}(V^{\prime}_{n})=V^{\prime}_{\infty},
$$
ce qui implique l'énoncé.
\end{proof}
\subsubsection{Preuve de la proposition \ref{Prop Fall l=p}}
On est finalement prêt à démontrer la proposition \ref{Prop Fall l=p}. Parce que ça fait déjà un petit moment qu'on a énoncé cette proposition, je vais rappeller ce qu'on veut démontrer: Soit $\mathcal{F}$ un $\mathbb{F}_{p}$-module constructible sur le site étale de $X^{\alg}_{E,F}.$ Alors on veut expliquer pourquoi l'hypothèse (Stein) implique
$$
H^{i}(X^{\ad}_{E,F},\mathcal{F}^{\ad})=0,
$$
pour $i\geq 3.$
\begin{proof}
On fixe un $\mathbb{F}_{p}$-module constructible $\mathcal{F}$ sur le site étale de $X^{\alg}_{E,F}.$ Comme dans la preuve de la proposition \ref{l-kohomologische Dimension algebraische Kurve}, on applique d'abord la méthode de la trace: soit $U$ un ouvert dans $X^{\alg}_{E,F},$ sur lequel $\mathcal{F}$ est un système local, $\dot{\pi}\colon V\rightarrow U$ le revêtement fini étale de degré premier à $p$ construit en utilisant la méthode de la trace et $\pi\colon X^{\prime}\rightarrow X^{\text{alg}}_{E,F}$ le revêtement ramifié construit en prenant la normalisation de $V$ dans $X^{\text{alg}}_{E,F};$ je rappelle que $\pi$ est un morphisme fini plat. Cette situation est de nouveau décrite par le diagramme suivant:
$$
\xymatrix{
V \ar[r]^{j^{\prime}} \ar[d]^{\dot{\pi}} & X^{\prime} \ar[d]^{\pi} &  \ar[l]^{i^{\prime}}Z^{\prime} \ar[d] \\
U \ar[r]^{j} & X^{\alg}_{E,F} & \ar[l]^{i} Z .
}
$$
On considère la composition suivante
$$
\xymatrix{
j_{!}\mathcal{L} \ar[r] & j_{!}\dot{\pi}_{*}\dot{\pi}^{*}(\mathcal{L}) \ar[r]^-{\text{tr}} & j_{!}\mathcal{L}.
}
$$
Par construction, elle s'identifie à la multiplication par $\text{deg}(\dot{\pi}).$ Comme $\dot{\pi}$ est fini étale, on a
$$
j_{!}\dot{\pi}_{*}\dot{\pi}^{*}(\mathcal{L})\simeq \pi_{*}j^{\prime}_{!}\dot{\pi}^{*}(\mathcal{L}).
$$
Appliquant à la composition ci-dessus le morphisme $u^{*}\colon \widetilde{(X^{\text{alg}}_{E,F})}_{\text{ét}}\rightarrow \widetilde{(X^{\text{ad}}_{E,F})}_{\text{ét}},$ on trouve 
$$
\xymatrix{
j^{\text{ad}}_{!}\mathcal{L}^{\text{ad}} \ar[r] & \pi^{\text{ad}}_{*}j^{\prime \text{ad}}_{!}(\dot{\pi}^{*}(\mathcal{L}))^{\text{ad}} \ar[r]^-{\text{tr}^{\text{ad}}} & j^{\text{ad}}_{!}\mathcal{L}^{\text{ad}}.
}
$$ 
La composition s'identifie toujours à la multiplication $\cdot \text{deg}(\pi^{\text{ad}}).$ On emploie ici le fait que $\pi_{*}$ et $j^{\prime}_{!}$ commutent avec l'adification (le cas de $j^{\prime}_{!}$ est évident et pour le cas du foncteur $\pi_{*}$ on peut utiliser \cite[Thm. 3.7.2]{HuberBuch}).
\\
Afin de démontrer $H^{i}(X^{\text{ad}}_{E,F},\,\mathcal{F}^{\text{ad}})=0,$ pour $i\geq 3,$ il suffit de prouver
\begin{equation}\label{Gleichung Reduktion auf ausdehnen durch null von etwas konstantem verzweigte ueberdeckung}
H^{i}(X^{\prime \text{ad}},\, j_{!}^{\prime \text{ad}}\mathbb{F}_{p})=0,
\end{equation}
pour $i\geq 3.$ 
\\
De fait, comme tous les corps résiduels des points dans $Z^{\text{ad}}$ sont algébriquement clos, il suffit de démontrer $H^{i}(X^{\text{ad}}_{E,F},\,j^{\text{ad}}_{!}(\mathcal{L}^{\text{ad}}))=0,$ $i\geq 3.$ Étant donné que le degré de $\pi^{\text{ad}}$ est premier à $p,$ on voit qu'il suffit à démontrer que $H^{i}(X^{\prime \text{ad}},\,j^{\prime \text{ad}}_{!}(\dot{\pi}^{*}(\mathcal{L}))^{\text{ad}})=0,$ $i\geq 3.$ Comme le foncteur $u^{*}(.)$ est exact, $(\dot{\pi}^{*}(\mathcal{L}))^{\text{ad}}$ est toujours une extension successive du faisceau $\underline{\mathbb{F}_{p}}_{V^{\text{ad}}}$ et on est finalement réduit à démontrer l'égalité (\ref{Gleichung Reduktion auf ausdehnen durch null von etwas konstantem verzweigte ueberdeckung}) ci-dessus.
\\
On considère le triangle suivant:
$$
\xymatrix{
j^{\prime \text{ad}}_{!}\mathbb{F}_{p} \ar[r] & Rj^{\prime \text{ad}}_{*}\mathbb{F}_{p} \ar[r] & i^{\prime \text{ad}}_{*} i^{\prime \text{ad} *}Rj^{\prime \text{ad}}_{*} \mathbb{F}_{p} \ar[r]^-{+1} &.
}
$$
Par corollaire \ref{Korollar Hyp Stein impliziert Verschwinden offenes}, on a
$$
H^{i}(V^{\ad},\mathbb{F}_{p})=0,
$$
pour $i\geq 3.$
 On est réduit à démontrer
$$
H^{i}(Z^{\prime \text{ad}},\,i^{\prime \text{ad} *}Rj^{\prime \text{ad}}_{*}\mathbb{F}_{p} )=0,
$$
pour $i\geq 2.$ Maintenant on affirme la formule suivante pour ces groupes de cohomologie:
\begin{Lemma_french}\label{Lemma Formel Kohomologie von Z ad}
$$
H^{i}(Z^{\prime \ad},\,i^{\prime \ad *}Rj^{\prime \ad}_{*}\mathbb{F}_{p} )= \underset{W^{\prime}} \colim H^{i}(W^{\prime}-Z^{\prime \ad},\,\mathbb{F}_{p}),
$$
où la colimite porte sur tous les ouverts qcqs $W^{\prime}\subseteq X^{\prime \ad},$ tels que $Z^{\prime \ad}\subset W^{\prime}.$
\end{Lemma_french}
 On doit expliquer cette formule, car a priori on a besoin d'utiliser tous les voisinages étales de $Z^{\prime \text{ad}}$ dans cette colimite. 
\begin{proof}{(du lemme \ref{Lemma Formel Kohomologie von Z ad}.)}
Les ingrédients sont les deux faits: premièrement, tous les corps résiduels des points $z_{i}^{\prime}\in Z^{\prime \text{ad}}$ sont algébriquement clos. Deuxièmement, dans le monde des espaces adiques on peut raffiner un voisinage étale donné par un voisinage étale, que l'on peut écrire comme composition d'une immersion ouverte et d'un morphisme fini étale (et on peut supposer que la source et la cible sont connexes c.f. \cite[Lemma 2.2.8]{HuberBuch}). 
\\
Soit $Z^{\prime \text{ad}}=\lbrace z_{1}^{\prime},...,z_{m}^{\prime}\rbrace;$ on a la formule suivante:
$$
H^{i}(Z^{\prime \text{ad}},\, i^{\prime \text{ad} *}Rj^{\prime \text{ad}}_{*}\mathbb{F}_{p} )=\bigoplus_{i=1}^{m}H^{i}(\lbrace z_{i} \rbrace,\, i_{ z_{i}^{\prime *} }Rj^{\prime \text{ad}}_{*}\mathbb{F}_{p} ),
$$
où $i_{ z_{i}^{\prime}}\colon \lbrace  z_{i}^{\prime} \rbrace \hookrightarrow X^{\prime \text{ad}}$ est l'inclusion du point $\lbrace z_{i}^{\prime} \rbrace.$ 
\\
Ensuite on peut analyser les groupes de cohomologie $H^{i}(\lbrace z_{i} \rbrace,\, i_{ z_{i}^{\prime}}^{*}Rj^{\prime \text{ad}}_{*}\mathbb{F}_{p} ).$ Ici on trouve la suite des identifications suivantes:
\begin{align*}
H^{i}(\lbrace z_{i}^{\prime} \rbrace,\,i_{ z_{i}^{\prime}}^{*}Rj^{\prime \text{ad}}_{*}\mathbb{F}_{p} ) & = \underset{f\colon(W^{\prime},v)\rightarrow (\Spa(B^{\prime}),z^{\prime}_{i}) \text{ étale et }W^{\prime}\text{ affd.}}  \colim  H^{i}(W^{\prime}-f^{-1}(v),\,\mathbb{F}_{p}) \\
 & = \underset{f\colon(W^{\prime},v)\rightarrow (X^{\prime \text{ad}},z^{\prime}_{i}) \text{ immersion ouverte qcqs, }W^{\prime}\text{ affd. connexe}} \colim H^{i}(W^{\prime}-\lbrace z_{i}^{\prime} \rbrace,\,\mathbb{F}_{p}),
d\end{align*}
où $\Spa(B^{\prime},B^{\prime +})$ est un voisinage affinoïde, connexe du point $z_{i}^{\prime}\in Z^{\prime \text{ad}}.$ De fait, dans un premier temps on observe que les voisinages étales $f\colon (W^{\prime},v)\rightarrow (\Spa(B^{\prime}),z_{i}^{\prime})$ avec $W^{\prime}$ affinoïde tels que $f^{-1}(z_{i}^{\prime})=\lbrace v \rbrace$ sont cofinaux. Afin de le voir, il faut se convaincre que chaque point dans la fibre $f^{-1}(z_{i}^{\prime})$ est fermé: $f$ est quasi-compact et localement quasi-fini, donc quasi-fini et $f^{-1}(z_{i}^{\prime})$ est un ensemble fini. Chaque point dans $f^{-1}(z_{i}^{\prime})$ est ainsi (ouvert et) fermé; et comme $z_{i}^{\prime}$ est fermé il en résulte que chaque point de $f^{-1}(z_{i}^{\prime})$ est fermé dans $W^{\prime}.$ Après, on peut trouver un système cofinal de voisinages étales $f\colon (W^{\prime},v)\rightarrow (\Spa(B^{\prime}),z_{i}^{\prime}),$ tel que $W^{\prime}$ et $\Spa(B^{\prime})$ sont affinoïdes connexes, $f^{-1}(z_{i}^{\prime})=\lbrace v \rbrace$ et $f$ est la composition d'une immersion ouverte et un morphisme fini étale. Puisque les corps résiduels $\kappa(z_{i}^{\prime})$ sont algébriquement clos, la partie finie étale est triviale et cela explique la deuxième égalité ci-dessus.
\\
On en déduit que
$$
H^{i}(Z^{\prime \text{ad}},\,i^{\prime \text{ad} *}Rj^{\prime \text{ad}}_{*}\mathbb{F}_{p})=\bigoplus_{i=1}^{m} H^{i}(i_{ z_{i}^{\prime}}^{*}Rj^{\prime \text{ad}}_{*}\mathbb{F}_{p} ) = \bigoplus_{i=1}^{m} \underset{z_{i}^{\prime}\in W^{\prime}_{i} \subset X^{\prime \text{ad}} \text{ affd}} \colim H^{i}(W^{\prime}_{i}-\lbrace z_{i}^{\prime} \rbrace,\,\mathbb{F}_{p}).
$$
Vu que les points $z_{i}^{\prime}$ sont maximaux, on peut trouver des voisinages ouverts $W_{i}^{\prime}$ des points $z_{i}^{\prime},$ tels que $W_{i}^{\prime}\cap W_{j}^{\prime}=\emptyset,$ c.f. \cite[Tag 0904]{stacks}. Il en résulte qu'on peut raffiner un voisinage ouvert qcqs $W^{\prime}$ de $Z^{\prime \text{ad}}$ par un voisinage de la forme $\coprod_{i=1}^{m}W_{i}^{\prime}.$ Finalement, on en déduit que
$$
\underset{Z^{\prime \text{ad}}\subset W^{\prime} \subset X^{\prime \text{ad}} \text{ affd}} \colim H^{i}(W^{\prime}-Z^{\prime \text{ad}},\,\mathbb{F}_{p})\simeq \bigoplus_{i=1}^{m} H^{i}(i_{ z_{i}^{\prime}}^{*}Rj^{\prime \text{ad}}_{*}\mathbb{F}_{p} ),
$$
comme cherché.
\end{proof}
On revient à la démonstration de la proposition \ref{Prop Fall l=p}.
D'après le lemme qu'on vient de démontrer, il  faut maintenant analyser la colimite 
$$\underset{W^{\prime}} \colim H^{i}(W^{\prime}-Z^{\prime \text{ad}},\,\mathbb{F}_{p}).$$
On effectue cette analyse par passage au revêtement perfectoide:
Soit $E_{\infty}$ encore la $\mathbb{Z}_{p}$-extension de $E.$ On peut l'écrire comme réunion croissante d'extensions galoisiennes $E_{n}$ de $E$ (telles que $G_{n}:=\Gal(E_{n}/E)\simeq \mathbb{Z}/p^{n}\mathbb{Z}$ et $\mathbb{Z}_{p}= \lim_{n} \Gal(E_{n}/E)$). Le morphisme $\Spd(\widehat{E}_{\infty})\rightarrow \Spd(E)$ est ainsi un $\underline{\mathbb{Z}_{p}}$-torseur pro-étale; par changement de base on a donc le $\underline{\mathbb{Z}_{p}}$-torseur pro-étale
$$
X^{\prime \text{ad} \lozenge}\times_{\Spd(E)}\Spd(\widehat{E}_{\infty})\rightarrow X^{\prime \text{ad} \lozenge}.
$$
Par la suite on écrit simplement $X^{\prime \text{ad} \lozenge}\times_{\Spd(E)}\Spd(\widehat{E}_{\infty}):=(X^{\prime \text{ad} \lozenge})_{\infty}.$

On considère maintenant le système filtrant des voisinages ouverts qcqs $\underline{\mathbb{Z}_{p}}$-invariants $\widetilde{W}^{\prime}\subset (X^{\prime \text{ad} \lozenge})_{\infty}$ de $(Z^{\prime \text{ad}})_{\infty}.$ On trouve ainsi l'égalité suivante:
\begin{equation}\label{Gleichung Ausdruck Kolimes via Diamanten}
\underset{W^{\prime}}\colim R\Gamma(W^{\prime}-Z^{\prime \text{ad}},\, \mathbb{F}_{p})=R\Gamma(\mathbb{Z}_{p},\underset{\widetilde{W^{\prime}}}\colim R\Gamma(\widetilde{W^{\prime}}-(Z^{\prime \text{ad}})_{\infty},\,\mathbb{F}_{p})).
\end{equation}
On a fait recours aux faits suivants: l'exactitude de la colimite, le fait que la cohomologie étale d'un espace adique analytique sur $\mathbb{Z}_{p}$ ne change pas si on passe aux diamants, et l'existence d'une identification entre les ouverts qcqs $W^{\prime}\subset X^{\prime \text{ad}}$ tels que $Z^{\prime \text{ad}}\subset W$ et les ensembles ouverts qcqs $\underline{\mathbb{Z}_{p}}$-invariants $\widetilde{W}^{\prime}\subset (X^{\prime \text{ad} \lozenge})_{\infty}$ tels que $(Z^{\prime \text{ad}})_{\infty} \subset \widetilde{W}.$ D'après le lemme \ref{Lemma Verschwinden lokal Stein}, on sait que le complexe $$\colim_{\widetilde{W}^{\prime}} R\Gamma(\widetilde{W}^{\prime}-(Z^{\prime \text{ad}})_{\infty},\,\mathbb{F}_{p})$$ n'a pas de cohomologie en degrés supérieur ou égale à $2.$
\\
Le fait que $\cd_{p}(\mathbb{Z}_{p})=1$ représente le problème principal: a priori ce résultat d'annulation ne suffit qu'à démontrer l'annulation de $H^{i}(i^{\prime \text{ad} *}Rj^{\prime \text{ad}}_{*}\mathbb{F}_{p})$ pour $i\geq 3.$ Cependant l'observation suivante sauve la situation: il y a un isomorphisme entre $\mathbb{Z}_{p}$-représentations\footnote{Les $\mathbb{F}_{p}$-modules discrets $\underset{\widetilde{W}^{\prime}}\colim H^{j}(\widetilde{W}^{\prime}-(Z^{\prime \text{ad}})_{\infty},\,\mathbb{F}_{p})$ sont bien sûr munis d'une action de $\mathbb{Z}_{p}$ - par l'action de $\underline{\mathbb{Z}_{p}}$ sur $\widetilde{W}^{\prime}-(Z^{\prime \text{ad}})_{\infty}$ et la fonctorialité de la cohomologie.}
\begin{equation}\label{Gleichung key claim fuer darstellungen}
\underset{\widetilde{W}^{\prime}} \colim H^{j}(\widetilde{W}^{\prime}-(Z^{\prime \text{ad}})_{\infty},\,\mathbb{F}_{p})=\underset{n\in \mathbb{N}} \colim \text{Ind}_{\lbrace e \rbrace}^{G_{n}}(M_{n}^{(j)}),
\end{equation}
où $M_{n}^{(j)}$ sont des $\mathbb{F}_{p}$-modules discrets, munis de l'action triviale de $G_{n}:$ ils dépendent de $j$ et forment un système inductif, i.e. on a une application $M_{n}^{(j)}\rightarrow M_{n+1}^{(j)}.$ Le lemme \ref{Lemma zum Verschwinden der Gruppenkohomologie} implique que la $\mathbb{Z}_{p}$-cohomologie d'une représentation du groupe $\mathbb{Z}_{p}$ de la forme 
$$
\underset{n\in \mathbb{N}} \colim \text{Ind}_{\lbrace e \rbrace}^{G_{n}}(M_{n}^{(j)})
$$
est triviale (c.f. également la discussion après ce lemme pour voir comment on a muni $\colim_{n} \text{Ind}_{\lbrace e \rbrace}^{G_{n}}(M_{n})$ d'une action du groupe $\mathbb{Z}_{p}$). L'égalité (\ref{Gleichung Ausdruck Kolimes via Diamanten}) implique que, en utilisant la suite spectrale de Hochschild-Serre,
$$
H^{i}(Z^{\prime \text{ad}},\,i^{\prime \text{ad} *}Rj^{\prime \text{ad}}_{*}\mathbb{F}_{p})=H^{0}(\mathbb{Z}_{p},\underset{\widetilde{W}^{\prime}}\colim H^{i}(\widetilde{W}^{\prime}-(Z^{\prime \text{ad}})_{\infty}),\,\mathbb{F}_{p}).
$$
Grâce au lemme \ref{Lemma Verschwinden lokal Stein} on pourra donc conclure. 
\\
L'énoncé clef pour finir la preuve est ainsi l'égalité (\ref{Gleichung key claim fuer darstellungen}).
\\
Mais maintenant on peut utiliser le système fondamental fourni par le lemme \ref{Lemma Konstruktion des guten Systems der Umgebungen} pour calculer la colimite:
\begin{align*}
\underset{\widetilde{W}^{\prime}} \colim H^{j}(\widetilde{W}^{\prime}-(Z^{\prime \text{ad}})_{\infty},\,\mathbb{F}_{p}) & = \underset{n\in \mathbb{N}} \colim \bigoplus_{g\in G_{n}} M_{n}^{(j)},
\end{align*}
où
$$
M_{n}^{(j)}=\underset{W^{\prime (n)}_{i} \text{ voisinage 'suffisament petit' de } \tilde{z}_{i}^{\prime (n)}} \colim H^{j}(f_{n}^{-1}(\coprod_{i=1}^{m} \cdot W_{i}^{\prime (n)}-(\lbrace \tilde{z}_{i}^{\prime (n)} \rbrace_{i=1,...,m})),\,\mathbb{F}_{p}).
$$
\footnote{Peut-être qu'il n'est pas vrai que $H^{j}(f_{n}^{-1}(\coprod_{i=1}^{m} \cdot W_{i}^{\prime (n)}-(\lbrace \tilde{z}_{i}^{\prime (n)} \rbrace_{i=1,...,m}),\,\mathbb{F}_{p})=0,$ pour $j\geq 2$ - mais cela n'est pas grave, parce qu'on a seulement besoin de cette expression de la colimite pour tuer la $\mathbb{Z}_{p}$-cohomologie...}
Sur les $M_{n}^{(j)}$ le groupe $\mathbb{Z}_{p}$ agit via le quotient $G_{n}$ et cette dernière action est triviale; de plus $\lbrace M_{n}^{(j)} \rbrace_{n\in \mathbb{N}}$ forment bien un système direct. En effet, on observe d'abord l'égalité suivante:
$$
M_{n}^{(j)}=\underset{W^{\prime}\text{ voisinage de }Z^{\prime \text{ad}}_{n}}\colim H^{i}(f_{n}^{-1}(W^{\prime}-Z^{\prime \text{ad}}_{n}),\mathbb{F}_{p}).
$$
De fait, si $Z^{\prime \text{ad}}_{n}\subset W^{\prime}$ est un tel voisinage, on peut trouver un voisinage suffisamment petit $W_{i}^{\prime (n)}\subset W^{\prime}$ pour chaque point $\tilde{z}^{(n)}_{i}\in Z^{\prime \text{ad}}_{n}.$ De plus, parce que les points $\tilde{z}^{(n)}_{i}$ sont tous maximaux, on peut supposer que $W_{i}^{\prime (n)}\cap W_{j}^{\prime (n)}=\emptyset;$ l'égalité en résulte. Après on considère le morphisme de transition dans la tour
$$
f_{n+1,n}\colon X^{\prime \text{ad}}_{n+1}\rightarrow X^{\prime \text{ad}}_{n}.
$$
Si $W^{\prime}$ est un voisinage qcqs de $Z^{\prime \text{ad}}_{n},$ $f_{n+1,n}^{-1}(W^{\prime})$ est un voisinage qcqs de $Z^{\prime \text{ad}}_{n+1}.$ On a par conséquent une application
$$
H^{i}(f_{n}^{-1}(W^{\prime}-Z^{\prime \text{ad}}_{n}),\mathbb{F}_{p})=H^{i}(f_{n+1}^{-1}(f_{n+1,n}^{-1}(W^{\prime})-Z^{\prime \text{ad}}_{n+1}),\mathbb{F}_{p})\rightarrow M_{n+1}^{(j)}.
$$
Elles sont compatibles si on raffine les voisinages et on a ainsi une application bien définie
$$
M_{n}^{(j)}\rightarrow M_{n+1}^{(j)}.
$$
Cela permet de finalement terminer la preuve.
\end{proof}
\begin{Remark_french}\label{Quatschbemerkung zu l=p und Grad}
On remarque ici que dans la preuve ci-dessus on a seulement utilisé le fait que le degré de $$\pi\colon X^{\prime}\rightarrow X^{\text{alg}}_{E,F}$$ est premier à $p$ pour pouvoir démontrer que $H^{i}(X^{\prime \text{ad}},\mathbb{F}_{p})=0$ implique que $H^{i}(X^{\text{ad}}_{E,F},\mathcal{F}^{\text{ad}})=0,$ où $i\geq 3.$ Après on n'a pas utilisé l'hypothèse sur le degré; ce qui implique que $H^{i}(X^{\prime \text{ad}},\mathbb{F}_{p})=0,$ $i\geq 3,$ pour tout revêtement ramifié $X^{\prime}$ de $X^{\text{alg}}_{E,F}$.
\end{Remark_french}
\subsection{Et que se passe-t-il pour des faisceaux non-Zariski constructibles?}\label{subsection Quatsch zu nicht-Zariski konstru}
Il n'est pas difficile de démontrer que l'on a
$$
\cd_{p}(X^{\text{ad}}_{E,F})\leq 3.
$$\footnote{Voici l'argument: on utilise la présentation $(X^{\text{ad}}_{E,F})^{\lozenge}=(\mathbb{B}^{1,\circ,*,\text{perf}}_{F}/\varphi_{F}^{\mathbb{Z}})/\underline{\mathbb{Z}_{p}};$ comme $\cd_{p}(\mathbb{Z}_{p})=1,$ la suite spectrale de Hochschild-Serre implique qu'il suffit de démontrer que $\cd_{p}(\mathbb{B}^{1,\circ,*,\text{perf}}_{F}/\varphi_{F}^{\mathbb{Z}})\leq 2.$ Cela résulte de la formule \cite[Prop. 21.11]{ScholzeDiamonds}, qui est aussi valide pour $\ell=p,$ des faits que $\text{dim}_{\text{Krull}}|\mathbb{B}^{1,\circ,*,\text{perf}}_{F}/\varphi_{F}^{\mathbb{Z}}|=1$ et que $\cd_{p}(M)\leq 1,$ où $M$ est un corps de caractéristique $p.$}
Dans mon approche pour contrôler la $p$-dimension cohomologique de la courbe algébrique, on n'a travaillé qu'avec des faisceaux sur la courbe adique $X^{\text{ad}}_{E,F}$ qui sont 'algébriques'. On peut ainsi se poser la question suivante:
\begin{Question}
Peut on trouver un $\mathbb{F}_{p}$-module constructible $\mathcal{F}$ (dans le sens de Huber \cite[Def. 2.7.2.]{HuberBuch}) sur l'espace adique $X^{\text{ad}}_{E,F},$ tel que
$$
H^{3}(X^{\text{ad}}_{E,F},\,\mathcal{F})\neq 0?
$$
\end{Question}
C'est une question intéressante parce que l'intuition que $X^{\text{ad}}_{E,F}$ est 'une surface de Riemann compacte' prédit que de fait $\text{cd}_{p}(X^{\text{ad}}_{E,F})\leq 2.$
\\
Dans le reste de cette section je veux analyser cette question dans le cas d'un faisceau constructible de la forme $j_{!}\mathbb{F}_{p},$ où
$$
j\colon Y_{[1,\rho]}\hookrightarrow X^{\text{ad}}_{E,F}=Y_{[1,q]}/(Y_{[1,1]}\sim Y_{[q,q]})
$$ 
et $\rho\in [1,q].$
Comme d'habitude, on considère le changement de base au niveau infini:
$$
j_{\infty}\colon Y_{[1,\rho],\infty}\hookrightarrow X^{\text{ad}}_{E_{\infty},F}.
$$
Ici on note $Y_{[1,\rho],\infty}:=Y_{[1,\rho]}\times_{\Spa(E)}\Spa(\widehat{E}_{\infty})$ et $X^{\text{ad}}_{E_{\infty},F}:=X^{\text{ad}}_{E,F}\times_{\Spa(E)}\Spa(\widehat{E}_{\infty}).$ En utilisant $$\text{cd}_{p}(\mathbb{Z}_{p})=1$$ et $\text{cd}_{p}(X^{\text{ad}}_{E_{\infty},F})\leq 2,$ on voit facilement que 
$$
H^{1}(\mathbb{Z}_{p},H^{2}(X^{\text{ad}}_{E_{\infty},F},\,j_{\infty !}\mathbb{F}_{p}))\simeq H^{3}(X^{\text{ad}}_{E,F},\,j_{!}\mathbb{F}_{p}).
$$
Maintenant on analyse le groupe
$$
H^{1}(\mathbb{Z}_{p},H^{2}(X^{\text{ad}}_{E_{\infty},F},\,j_{\infty !}\mathbb{F}_{p})).
$$
On rappelle que $X^{\text{ad} \flat}_{E_{\infty},F}=(\mathbb{B}^{1,\circ,*,\text{perf}}_{F}/\varphi_{F}^{\mathbb{Z}}).$ On considère de plus la décomposition ouverte/fermée
$$
\xymatrix{
A_{[1,\rho],F}^{\text{perf}} \ar[r]^-{k} & (\mathbb{B}^{1,\circ,*,\text{perf}}_{F}/\varphi_{F}^{\mathbb{Z}}) & \ar[l]^-{i} A_{(\rho,q],F}^{\text{perf}}.
}
$$
Cette décomposition implique la suite exacte courte suivante
$$
\xymatrix{
k_{!} \mathbb{F}_{p} \ar[r] & \mathbb{F}_{p} \ar[r] & i_{*}\mathbb{F}_{p}.
}
$$
Les espaces $(X^{\text{ad}}_{E_{\infty},F})^{\flat}$ et $A_{(\rho,q],F}$ sont connexes; il en résulte 
$$H^{0}(k_{!}\mathbb{F}_{p})=H^{1}(k_{!}\mathbb{F}_{p})=0$$
en utilisant la suite exacte longue en cohomologique induite par la suite exacte courte ci-dessus. Comme $H^{2}(X^{\text{ad} \flat}_{E_{\infty},F},\,\mathbb{F}_{p})=0,$ cette suite exacte longue donne la suite exacte courte suivante
$$
\xymatrix{
0 \ar[r] & H^{1}(X^{\text{ad} \flat}_{E_{\infty},F},\,\mathbb{F}_{p}) \ar[r] & H^{1}(A^{\text{perf}}_{(\rho,q],F},\,\mathbb{F}_{p}) \ar[r] & H^{2}(k_{!}\mathbb{F}_{p}) \ar[r] & 0.
}
$$
C'est une suite exacte courte de $\mathbb{Z}_{p}$-modules discrets et en appliquant la $\mathbb{Z}_{p}$-cohomologie, on peut en déduire une suite exacte longue. Il en résulte une surjection
$$
H^{1}(\mathbb{Z}_{p},H^{1}(A^{\text{perf}}_{(\rho,q],F},\, \mathbb{F}_{p}))\twoheadrightarrow H^{1}(\mathbb{Z}_{p}, H^{2}(k_{!}\mathbb{F}_{p})).
$$
Le noyau de cette surjection s'identifie à l'image de $H^{1}(\mathbb{Z}_{p},H^{1}(X^{\text{ad} \flat}_{E_{\infty},F},\,\mathbb{F}_{p})).$ On observe que $$H^{1}(\mathbb{Z}_{p},H^{1}(X^{\text{ad} \flat}_{E_{\infty},F},\,\mathbb{F}_{p}))=H^{2}(X^{\text{ad}}_{E,F},\,\mathbb{F}_{p}).$$ Si on suppose de plus que $\mu_{p}\subset E,$ ce dernier groupe de cohomologie s'identifie à $\mathbb{F}_{p}.$
\\
Maintenant on analyse $H^{1}(\mathbb{Z}_{p},H^{1}(A^{\text{perf}}_{(\rho,q],F},\, \mathbb{F}_{p})).$ D'abord, on observe que
$$
H^{1}(\mathbb{Z}_{p},H^{1}(A^{\text{perf}}_{(\rho,q],F},\, \mathbb{F}_{p}))=H^{2}(Y_{(\rho,q]},\,\mathbb{F}_{p}).
$$
L'espace $Y_{(\rho,q]}$ admet un recouvrement de Stein
$$
Y_{(\rho,q]}=\bigcup_{n\in \mathbb{N}}Y_{[\rho_{n},q]}.
$$
Cela implique la suite exacte courte suivante
$$
\xymatrix{
0 \ar[r] & \text{R}^{1}\lim_{n} H^{1}(Y_{[\rho_{n},q]},\,\mathbb{F}_{p}) \ar[r] & H^{2}(Y_{(\rho,q]},\,\mathbb{F}_{p}) \ar[r] & \lim_{n}H^{2}(Y_{[\rho_{n},q]},\,\mathbb{F}_{p}) \ar[r] & 0.
}
$$
Sous des hypothèses supplémentaires on peut vérifier que la dernière limite s'annule:
\begin{Lemma_french}
On suppose que le corps de fonctions $E(X^{\text{alg}}_{E,F})$ est de dimension cohomologique $1$ et $\mu_{p}\subset E^{*}$. Alors on a
$$
H^{i}(Y_{I},\,\mathbb{F}_{p})=0,
$$
pour $i\geq 2$ et $I=[a,b]\subset (0,\infty)$ un intervalle compact avec $a,b\in \mathbb{Q}$, tel que $\varphi(Y_{I})\cap Y_{I}=\emptyset.$
\end{Lemma_french}
\begin{proof}
Il suffit d'expliquer l'annulation pour $i=2:$ en utilisant la présentation $Y_{I}^{\lozenge}=A_{I,F}^{\text{perf}}/\underline{\mathbb{Z}_{p}}$ on démontre
$$
H^{i}(Y_{I},\mathbb{F}_{p})=0,
$$
pour $i\geq 3.$
\\
Soit $t\in B^{\varphi=\pi},$ tel que $t$ ne s'annule pas dans $Y_{I}.$ D'après un résultat de Fargues-Fontaine (\cite[Prop. 7.9.1.]{FarguesFontaine}) l'inclusion canonique $B_{e}=(B[1/t])^{\varphi=1}\hookrightarrow B_{I}$ est d'image dense. Soit $$A_{e}=\lbrace f\in B_{e}\colon ||f||_{I}\leq 1 \rbrace.$$ On a $A_{e}[1/\pi]=B_{e}.$ D'après le résultat de comparaison de Huber (\cite[Cor. 3.2.2.]{HuberBuch}), on a
$$
H^{i}(Y_{I},\,\mathbb{F}_{p})=H^{i}(\Spec(A_{e}^{h}[1/\pi]),\,\mathbb{F}_{p}).
$$
Ici $A_{e}^{h}= \colim_{A_{e}\rightarrow A_{e}^{\prime}} A_{e}^{\prime}$ est le hensélisé de $A_{e}$ le long de $\pi,$ i.e. la colimite porte sur tous les $A_{e}\rightarrow A_{e}^{\prime}$ étale, tel que $A_{e}/\pi\simeq A_{e}^{\prime}/\pi.$ Par compatibilité de la cohomologie étale aux colimites filtrantes, on a
$$
H^{i}(\Spec(A_{e}^{h}[1/\pi]),\,\mathbb{F}_{p})=\underset{A_{e}\rightarrow A_{e}^{\prime}} \colim H^{i}(\Spec(A_{e}^{\prime}[1/\pi]),\,\mathbb{F}_{p}).
$$
On affirme que $\Br(\Spec(A_{e}^{\prime}[1/\pi]))=0.$ C'est ici que l'on doit utiliser l'hypothèse que $\text{cd}(E(X^{\text{alg}}_{E,F}))=1.$ De fait, $\Spec(A_{e}^{\prime}[1/\pi])$ est un schéma noethérien de dimension $1$ qui est régulier. Soient $\lbrace C_{i} \rbrace_{i\in I}$ ses composantes connexes. Il y en a un nombre fini et ce sont aussi les composantes irréductibles. Chaque $C_{i}$ est le spectre d'un anneau de Dedekind. Il suffit de démontrer que $\Br(C_{i})=0$ et pour cela il suffit de démontrer que $\Br(\text{Frac}(\mathcal{O}(C_{i})))=0,$ car pour un schéma régulier intègre $S,$ avec point générique $\eta_{S}$, de dimension $1$ on a une injection
$$
\Br(S)\hookrightarrow \Br(\kappa(\eta_{S})).
$$
Le morphisme induit $C_{i}\rightarrow \Spec(B_{e})$ est toujours étale. L'image $U$ est un ouvert et comme $B_{e}$ est un anneau principal l'ouvert $U$ est affine. Le morphisme $C_{i}\rightarrow U=\Spec(R)$ est étale et surjectif entre schémas affines. Il correspond ainsi à une injection sur les anneaux et il est donc dominant. On observe $\text{Frac}(B_{e})=\text{Frac}(R)$ et le fait que l'extension $\text{Frac}(\mathcal{O}(C_{i}))$ de $\text{Frac}(B_{e})$ est finie. On a donc $\Br(\text{Frac}(\mathcal{O}(C_{i})))=0,$ parce que $\text{cd}(E(X^{\text{alg}}_{E,F}))=1.$
\\
 En utilisant la suite exacte de Kummer, on déduit de l'hypothèse $\mu_{p}\subset E^{\times}$ l'isomorphisme $$H^{2}(\Spec(A_{e}^{\prime}[1/\pi]),\,\mathbb{F}_{p})\simeq \text{Pic}(\Spec(A_{e}^{\prime}[1/\pi]))_{[p]}.$$ Maintenant, on affirme que
$$
\underset{A_{e}\rightarrow A_{e}^{\prime}} \colim \text{Pic}(\Spec(A_{e}^{\prime}[1/\pi]))=\text{Pic}(B_{I})=0,
$$
ce qui finit la preuve car $B_{I}$ est un anneau principal.
\\
D'abord, on utilise la compatibilité de la cohomologie étale aux colimites filtrantes
$$
\underset{A_{e}\rightarrow A_{e}^{\prime}} \colim \text{Pic}(\Spec(A_{e}^{\prime}[1/\pi]))=H^{1}_{\text{ét}}(\Spec(A_{e}^{h}[1/\pi]),\,\mathbb{G}_{m}).
$$
On observe que $A_{e}^{h}$ et $\widehat{A}_{e}$ sont sans $\pi$-torsion. On peut donc appliquer le résultat de Bouthier-Česnavičius  \cite[Thm. 2.3.3. (c)]{BouthierCesnavicius} à $G=\mathbb{G}_{m}$ et au morphisme 
$$
(A_{e}^{h},\pi)\rightarrow (\widehat{A}_{e},\pi)
$$
et en déduire
$$
\text{Pic}(A_{e}^{h}[1/\pi])\simeq \text{Pic}(\widehat{A}_{e}[1/\pi]).
$$
Comme $B_{e}\hookrightarrow B_{I}$ est d'image dense, on a $\widehat{A}_{e}[1/\pi]=B_{I}.$
\end{proof}
La question finale est la suivante
\begin{Question}
Dans la situation ci-dessus: A-t-on
$$
\text{R}^{1}\lim_{n} H^{1}(Y_{[\rho_{n},q]},\,\mathbb{F}_{p})=0?
$$
\end{Question}
Il est équivalent d'essayer de comprendre 
$$
\text{R}^{1}\lim_{n}H^{0}(\mathbb{Z}_{p},H^{1}(A_{[\rho_{n},q],F},\mathbb{F}_{p})).
$$
Si on suppose que $\mu_{p}\subset E^{*}$ et que $\text{cd}(E(X^{\alg}_{E,F}))\leq 1,$ la question suivante est équivalente:
\begin{Question}
A-t-on
$$
H^{2}(Y_{(\rho,q]},\,\mu_{p})\neq 0?
$$
\end{Question}
Par l'analogie entre $Y_{E,F}$ et la boule ouverte unité épointée $\mathbb{B}^{\circ,1, *}_{C},$ où $C$ est un corps non archimedien algébriquement clos, il est probable que cette question a une réponse positive si le corps $F$ n'est pas sphériquement clos. Peut on par exemple adapter les résultats de la section \cite[Prop. A.2.]{Colmezetccourbes}?
\\
La conclusion de cette discussion est ainsi la suivante: si $F$ n'est pas sphériquement clos je peux imaginer que 
$$
H^{3}(X^{\text{ad}}_{E,F},j_{!}(\mathbb{F}_{p}))\neq 0.
$$
\newpage
\bibliography{/Users/sebastianbartling/Documents/Bibtex/mybib}{}
\bibliographystyle{alpha}
\end{document}